\documentclass[10pt]{article}
\usepackage[margin=1in]{geometry}

\RequirePackage{amsthm,amsmath,amsfonts,amssymb}
\RequirePackage[round]{natbib}
\RequirePackage[colorlinks,linkcolor=black,citecolor=blue,urlcolor=blue]{hyperref}
\RequirePackage{enumerate,accents,rotating,bbm}

\newtheorem{Theorem}{Theorem}[section]
\newtheorem*{Theorem*}{Theorem}
\newtheorem{Lemma}[Theorem]{Lemma}
\newtheorem{Proposition}[Theorem]{Proposition}
\newtheorem{Corollary}[Theorem]{Corollary}
\newtheorem{Assumption}[Theorem]{Assumption}

\theoremstyle{definition}
\newtheorem{Remark}[Theorem]{Remark}

\newcommand{\Ln}{\operatorname{Ln}}
\newcommand{\iid}{\mathrm{iid}}

\newcommand{\R}{\mathbb{R}}
\newcommand{\C}{\mathbb{C}}

\newcommand{\Z}{\mathbb{Z}}

\newcommand{\Unif}{\operatorname{Unif}}
\DeclareMathOperator*{\argmax}{arg\,max}
\DeclareMathOperator*{\argmin}{arg\,min}
\newcommand{\diag}{\operatorname{diag}}
\newcommand{\Cov}{\operatorname{Cov}}
\newcommand{\Var}{\operatorname{Var}}
\newcommand{\Tr}{\operatorname{Tr}}
\newcommand{\1}{\mathbbm{1}}

\renewcommand{\Re}{\operatorname{Re}}
\renewcommand{\Im}{\operatorname{Im}}
\newcommand{\cP}{\mathcal{P}}
\newcommand{\cS}{\mathcal{S}}
\newcommand{\cE}{\mathcal{E}}
\newcommand{\cF}{\mathcal{F}}
\newcommand{\KL}{\mathrm{KL}}
\newcommand{\TV}{\mathrm{TV}}

\newcommand{\orbit}{\mathcal{O}}
\newcommand{\der}{\mathrm{d}}
\newcommand{\oracle}{\mathrm{oracle}}
\newcommand{\opt}{\mathrm{opt}}
\newcommand{\fm}{\mathrm{fm}}
\newcommand{\MLE}{\mathrm{MLE}}
\newcommand{\gen}{\mathrm{gen}}

\newcommand{\rlower}{\underaccent{\bar}{r}}
\newcommand{\rupper}{\bar{r}}
\newcommand{\clower}{\underaccent{\bar}{c}}
\newcommand{\cupper}{\bar{c}}

\newcommand{\cI}{\mathcal{I}}
\newcommand{\cA}{\mathcal{A}}
\newcommand{\cB}{\mathcal{B}}
\renewcommand{\P}{\mathbb{P}}

\newcommand{\E}{\mathbb{E}}
\newcommand{\eps}{\varepsilon}
\newcommand{\Normal}{\mathcal{N}}
\newcommand{\Arg}{\operatorname{Arg}}

\renewcommand\footnotemark{}

\newcommand{\revise}[1]{#1}

\title{Rates of estimation for high-dimensional multi-reference alignment}

\author{Zehao Dou, Zhou Fan, Harrison Zhou
\footnote{Department of Statistics and Data Science, Yale University, USA.
\newline \texttt{zehao.dou@yale.edu, zhou.fan@yale.edu, huibin.zhou@yale.edu}}}

\begin{document}
\maketitle
\allowdisplaybreaks[4]

\begin{abstract}
We study the continuous multi-reference alignment model of estimating a
periodic function on the circle from noisy and circularly-rotated observations.
Motivated by analogous high-dimensional problems that arise in cryo-electron microscopy,
we establish minimax rates for estimating generic
signals that are explicit in the dimension $K$. In a high-noise regime with
noise variance $\sigma^2 \gtrsim K$, \revise{for signals with Fourier
coefficients of roughly uniform magnitude}, the rate scales as $\sigma^6$ and has
no further dependence on the dimension. This rate is achieved by a
bispectrum inversion procedure, and our analyses provide new stability bounds
for bispectrum inversion that may be of independent interest. In a low-noise
regime where $\sigma^2 \lesssim K/\log K$, the rate scales instead as
$K\sigma^2$, and we establish this rate by a sharp analysis of the maximum
likelihood estimator that marginalizes over latent rotations. A
complementary lower bound that interpolates between these two regimes is 
obtained using Assouad's hypercube lemma. \revise{We extend
these analyses also to signals whose Fourier coefficients have a slow
power law decay.}
\end{abstract}

\tableofcontents

\section{Introduction}

Multi-reference alignment (MRA) refers to the problem of estimating an unknown
signal from noisy samples that are subject to latent rotational transformations
\citep{ritov1989estimating,bandeira2014multireference}. This problem has seen
renewed interest in recent years, as a simplified model for
molecular reconstruction in cryo-electron microscopy (cryo-EM) and
related methods of molecular imaging
\citep{bendory2020single,singer2020computational}. It arises also in various
other applications in structural biology and image registration
\citep{sadler1992shift,brown1992survey,diamond1992multiple}. Recent literature has established rates of estimation for MRA
in fixed dimensions \citep{perry2019sample,bandeira2020optimal,abbe2018multireference,ghosh2021multi}, describing a rich picture of how these rates may
depend on the signal-to-noise ratio and properties of the underlying
signal. However, many applications of MRA involve high-dimensional signals,
and there is currently limited understanding of optimal rates of estimation in
high-dimensional settings.

In the continuous MRA model---the focus of this work---the signal is a smooth
periodic function $f$ on the circular domain $[-\pi,\pi)$. We observe
independent samples of $f$ in additive white noise, where each sample has a
uniformly random latent rotation of its domain
\citep{bandeira2020optimal,fan2021maximum}.
The true function $f$ is identifiable only up to
rotation, and we will study its estimation under the rotation-invariant
squared-error loss
\begin{equation}\label{eq:functionloss}
L(\hat{f},f)=\min_{\alpha \in [-\pi,\pi)} \int_{-\pi}^\pi
\Big(\hat{f}(t)-f(t-\alpha \bmod 2\pi)\Big)^2\,\der t.
\end{equation}
In the closely related discrete MRA model, the signal is instead a
vector $x \in \R^K$, observed in additive Gaussian noise with cyclic
permutations of its coordinates \citep{bandeira2014multireference,perry2019sample}. The continuous and discrete models
are similar, in that both rotational actions are diagonalized in the
(continuous or discrete, resp.) Fourier basis, and these diagonal actions have
similar forms.

A recent line of work has studied rates of estimation for MRA in ``low
dimensions'', treating as constant the dimension $K$ for discrete MRA,
or the maximum Fourier frequency $K$ for continuous MRA. Many such results have
specifically focused on a regime of high noise: In this regime,
\cite{perry2019sample} showed that the squared-error risk for estimating
``generic'' signals scales with the noise standard deviation
as $\sigma^6$. \cite{bandeira2020optimal} showed that this scaling for
estimating a ``non-generic'' signal depends on its pattern of zero and non-zero
Fourier coefficients, and derived rate-optimal upper and lower bounds over
minimax classes of such signals. Rates of estimation
for MRA with non-uniform rotations were studied in
\cite{abbe2018multireference}, with a dihedral group of both
rotations and reflections in \cite{bendory2022dihedral}, with sparse signals in
\cite{ghosh2021multi}, and with down-sampled observations in a super-resolution
context in \cite{bendory2020super}.

It is empirically observed, for example in \citet[Section 5]{fan2021maximum},
that electric potential functions of protein molecules in cryo-EM applications
may require basis representations with dimensions in the thousands
to capture secondary structure, and even higher dimensions to achieve
near-atomic resolution. Motivated by this observation, in this paper, we
extend the above line of work to study the continuous
MRA model in potentially high dimensions, in both high-noise
and low-noise regimes. Our main results are described informally
as follows: \revise{Let
\[\theta^*=(\theta_{1,1}^*,\theta_{1,2}^*,\theta_{2,1}^*,\theta_{2,2}^*,\theta_{3,1}^*,\theta_{3,2}^*,\ldots)\]
be the coefficients of $f$ in the real Fourier basis over $[-\pi,\pi)$,
i.e.,
\[f(t) = \sum_{k=1}^{\infty} \frac{1}{\sqrt{\pi}}\theta_{k,1}^*\cos kt +
\frac{1}{\sqrt{\pi}}\theta_{k,2}^*\sin kt,\]
and let
\begin{equation}\label{eq:polarcoordinates}
(r_k \cos \phi_k, r_k\sin\phi_k)=(\theta_{k,1}^*,\theta_{k,2}^*)
\end{equation}
be the representation of the $k^\text{th}$ Fourier
frequency in terms of the magnitude $r_k$ and phase $\phi_k$. Fixing a decay
parameter $\beta \in [0,\frac{1}{2})$, we
consider a class of signals $f$ represented by
\[\Theta_\beta=\Big\{f:\,r_k \asymp k^{-\beta} \text{ for } k=1,\ldots,K,\;
r_k=0 \text{ for all } k \geq K+1\Big\}\]
where we bandlimit $f$ to its first $K$ Fourier frequencies.} Our results distinguish two separate signal-to-noise regimes for estimating $f$,
based on the size of the entrywise noise variance $\sigma^2$ in the
Fourier basis. We establish sharp minimax rates of
estimation in both regimes, \revise{for sufficiently large sample size
$N$}, that are explicit in their dependence on the dimension $K$.\revise{
\begin{Theorem*}[Informal]
Let $\beta \in [0,\frac{1}{2})$.
\begin{enumerate}[(a)]
\item (High noise) If $\sigma^2 \gtrsim K^{1-2\beta}$ and $N \gtrsim K^{6\beta}\sigma^6\log K$, then
\[\inf_{\hat{f}} \sup_{f \in \Theta_\beta} \E[L(\hat{f},f)] \asymp 
\frac{K^{4\beta}\sigma^6}{N}.\]
\item (Low noise) If $\sigma^2 \lesssim K^{1-2\beta}/\log K$ and $N \gtrsim
K^{1+2\beta}\sigma^2 \log K$, then
\[\inf_{\hat{f}} \sup_{f \in \Theta_\beta} \E[L(\hat{f},f)] \asymp
\frac{K\sigma^2}{N}.\]
\end{enumerate}
\end{Theorem*}
\noindent We refer to Theorems \ref{thm:highnoise} and \ref{thm:lownoise}
for precise statements of these results. Our signal class with power law decay
$\beta<1/2$ is representative of a setting where the average power per Fourier
frequency, $\|\theta^*\|^2/K \asymp K^{-2\beta}$, is of comparable magnitude to
the power $r_k^2$ at a typical frequency $k \in \{1,\ldots,K\}$. Our
analyses of the estimators that achieve these minimax rates apply more
generally to signals of this form
(c.f.\ Theorems \ref{thm:oracleMoM} and \ref{thm-mle}).}

\revise{For large $N$, this result implies that there is a sharp transition
in the minimax estimation rate near the noise level
$\sigma^2 \asymp K^{1-2\beta} \asymp \|\theta^*\|^2$, which separates the two
signal-to-noise regimes of the problem. Such a transition may be anticipated by
the results of \cite{bandeira2017estimation}, where $\sigma^2 \gtrsim
\|\theta^*\|^2$ is the condition required to carry out the high-noise Taylor
expansion of the chi-squared divergence, and of \cite{romanov2021multi} which provided
a sharp analysis of the sample complexity in the low-noise regime for an
analogous discrete MRA model (see below). As $\sigma^2$ varies in
the small parameter window from $K^{1-2\beta}/\log K$ to $K^{1-2\beta}$ between
these ``low-noise'' and ``high-noise'' regimes, our
result confirms that there must be a rapid increase in the minimax risk, from
roughly the order $K^{2-2\beta}/N$ to $K^{3-2\beta}/N$.}

In the high-noise regime where $\sigma^2\gtrsim \|\theta^*\|^2$, we show
that the minimax rate is achieved by a variant of a third-order
method-of-moments (MoM) procedure. The scaling with $\sigma^6$ matches previous
results of \cite{perry2019sample}, and a notable new feature of the rate
is its scaling with the dimension $K$---for example, when $\beta=0$,
the rate has no explicit dependence on $K$.
In the MRA model, for functions having the
Fourier coefficients (\ref{eq:polarcoordinates}),
second-order moments correspond to the power spectrum
\[\Big\{r_k^2:k=1,\ldots,K\Big\}\]
and third-order moments to the Fourier bispectrum
\[\Big\{\phi_{k+l}-\phi_k-\phi_l:k,l \in \{1,\ldots,K\} \text{ and } k+l \leq K\Big\}\]
Method-of-moments in this context is also
known as bispectrum inversion \citep{sadler1992shift,bendory2017bispectrum},
which aims to estimate the
Fourier phases $\{\phi_k\}$ from an estimate of the bispectrum. Results of
\cite{bendory2017bispectrum,perry2019sample} imply
that for signals where $r_k \neq 0$ for every $k=1,\ldots,K$,
these phases are uniquely determined by the bispectrum. \revise{Our analyses
quantify the conditioning of the linear system relating the bispectrum to the
Fourier phases, which gives rise to the quantitative dependence of the estimation
rate on $K$. To resolve phase ambiguities
before solving this linear system, we prove also an important $\ell_\infty$
stability property of bispectrum inversion
(c.f.\ Lemma \ref{lemma:phasestability}), which is of independent interest.}

Our definition of the low-noise regime $\sigma^2 \lesssim
K^{1-2\beta}/\log K \asymp \|\theta^*\|^2/\log K$ and minimax rate in this
regime are related to the work of \cite{romanov2021multi},
which studied instead the discrete MRA model in the asymptotic limit
$K \to \infty$ and $(\sigma^2 \log K)/K \to 1/\alpha \in (0,\infty)$,
for a Bayesian setting where $\theta^*$ has a standard Gaussian prior.
This work showed a transition in the Bayes risk and associated
sample complexity 
at the sharp threshold $\alpha=2$. The analysis in \cite{romanov2021multi} 
relied on the discreteness of the rotational model, analyzing a template
matching procedure that exactly recovers the latent rotation for each
sample. For continuous MRA, 
this estimation of each rotation is possible only up to a per-sample error 
that is independent of the sample size $N$, and averaging the
correspondingly rotated samples would yield an estimation bias that does not
vanish with $N$. Our analysis shows that direct application of third-order
method-of-moments also does not yield the optimal estimation rate across the
entire low-noise regime.
We instead analyze the maximum-likelihood estimator (MLE) that
marginalizes over latent rotations, to obtain the minimax upper bound in this
regime.

\subsection{Further related literature}

A body of work on MRA and related models focuses on the
synchronization approach, which seeks to first estimate the latent rotation of
each sample based on the relative rotational alignments between pairs of
samples \citep{singer2011angular}. In the context of cryo-EM, this is
known also as the ``common lines'' method \citep{singer2010detecting,singer2011three}. Algorithms developed and
studied for estimating these pairwise alignments include spectral
procedures \citep{singer2011angular,singer2011three,ling2022near}, semidefinite relaxations
\citep{singer2011angular,singer2011three,bandeira2014multireference,bandeira2015non},
and iterative power method or approximate message passing approaches \citep{boumal2016nonconvex,perry2018message}.

In high-noise regimes,
synchronization-based estimation may fail to recover the latent rotations, or
may lead to a biased and inconsistent estimate of the underlying signal. A
separate line of
work has studied alternative method-of-moments or maximum likelihood
procedures for the MRA problem, which marginalize over the latent rotations
\citep{abbe2018multireference,boumal2018heterogeneous,perry2019sample,bandeira2020optimal,ghosh2021multi,bendory2022dihedral}.
These papers relate the rate of estimation in high noise
to the order of moments needed to identify the true signal, which may differ
depending on the sparsity pattern of its Fourier coefficients
and the distribution of the latent random rotations.

Related analyses have been
performed for three-dimensional rotational actions, as arising in Procrustes
alignment problems \citep{pumir2021generalized} and cryo-EM
\citep{sharon2020method}. For cryo-EM, these methods
encompass invariant-features approaches \citep{kam1980reconstruction} and expectation-maximization
algorithms \citep{sigworth1998maximum,scheres2005maximum,scheres2012relion}.
The works \cite{bandeira2017estimation,abbe2018estimation} studied
method-of-moments estimators in problems with general rotational groups,
where \cite{bandeira2017estimation} related the rates of estimation and numbers of moments needed to identify the true signal to the structure of the
invariant polynomial algebra of the group action. In these general settings,
\cite{brunel2019learning,fan2020likelihood,katsevich2020likelihood,fan2021maximum}
studied also properties of the log-likelihood function, its optimization
landscape, and the Fisher information matrix, relating the structure of the
invariant algebra to asymptotic rates of estimation for the MLE.

\subsection{Outline}

Section \ref{sec-2} provides a formal statement of the continuous MRA model and
of our main results. Section \ref{sec-3} provides some preliminaries that relate
the loss function to the Fourier magnitudes and phases. Section \ref{sec-4}
proposes and analyzes a third-order method-of-moments estimator, which
determines the phases
by inverting the Fourier bispectrum. This estimator attains the minimax upper
bound for squared-error risk in the high-noise regime. Section \ref{sec-5}
analyzes the maximum likelihood estimator that attains the minimax upper bound
for squared-error risk in the low-noise regime. Section \ref{sec-6}
gives a minimax lower bound using Assouad's lemma, which matches the upper
bounds of Sections \ref{sec-4} and \ref{sec-5} while also interpolating 
between these two signal-to-noise regimes.

\subsection{Notation}
For a complex number $z=re^{i\theta} \in \C$, $\overline{z}=re^{-i\theta}$ is
its complex conjugate. $\Arg z=\theta$ is its principal argument in the
range $[-\pi,\pi)$.
$\langle u,v \rangle=\sum_k u_k\overline{v_k}$ is the $\ell_2$ inner-product
for real or complex vectors, and $\|u\|=\sqrt{\langle u,u \rangle}$
is the $\ell_2$ norm.
$I_K \in \R^{K \times K}$ is the identity matrix in dimension $K$.
$\Normal_\C(0,\sigma^2)$ is the complex
mean-zero Gaussian distribution, with independent real and imaginary parts
having real Gaussian distribution $\Normal(0,\frac{\sigma^2}{2})$.
We write $a \wedge b=\min(a,b)$.
For a function $F:\R^k \to \R$, we denote its gradient and Hessian by
$\nabla F \in \R^k$ and $\nabla^2 F \in \R^{k \times k}$.
For two distributions $P$ and $Q$, $D_{\KL}(P\|Q)=\int \log(\frac{P}{Q})dP$ is
their Kullback-Leibler (KL) divergence.

\section{Model and main results}
\label{sec-2}
\revise{Let $\cS^1=[-\pi,\pi)$ be identified with the unit circle, with addition modulo
$2\pi$. Let $f:\cS^1 \to \R$ be a smooth periodic function on $\cS^1$. We
represent rotations of the circle by angles $\alpha \in \cA=[-\pi,\pi)$,}
and denote the function $f$ with domain rotated by $\alpha$ as
\[f_{\alpha}(t)=f(t-\alpha \bmod 2\pi).\]
We study estimation of $f$ from $N$ i.i.d.\ samples of the form
\[f_{\alpha}(t)\,\der t+\sigma\,\der W(t), \qquad
\alpha \sim \Unif([-\pi,\pi)).\]
In each sample, $\alpha$ represents a different latent and uniformly random
rotation of the domain of $f$, and the entire rotated function $f_\alpha$ is
observed with additive continuous white noise $\sigma\,\der W(t)$ on the circle.
\revise{An equivalent Gaussian sequence formulation of the model is discussed
below.} We assume that $\sigma>0$ is a fixed and known noise level. As $f$ is
identifiable only up to rotation, we consider the rotation-invariant loss
(\ref{eq:functionloss}).

Note that we may
alternatively study a model where each rotated function $f_\alpha(t)$ is
observed with Gaussian noise only at a discrete set of points $t \in
\cS^1$ that are
fixed or randomly sampled \citep{bandeira2020optimal,bendory2020super}.
We study the above continuous observation model so as to abstract away aspects
of the problem that are related to this discrete sampling.

The mean value of $f$ over the circle is invariant to rotations, and is
easily estimated by averaging across samples. Thus, let us assume for simplicity
and without loss of generality that $f$ has known mean 0.
Passing to the Fourier domain,
we assume that $f$ is bandlimited to $K$ Fourier frequencies, i.e.\ $f$
admits the Fourier sequence representation
\[f(t)=\sum_{k=1}^K \theta_{k,1} f_{k,1}(t)+\theta_{k,2} f_{k,2}(t),
\qquad f_{k,1}(t)=\frac{1}{\sqrt{\pi}} \cos kt,
\quad f_{k,2}(t)=\frac{1}{\sqrt{\pi}} \sin kt,\]
where $\{f_{k,1},f_{k,2}:k=1,\ldots,K\}$ are orthonormal Fourier basis
functions over $[-\pi,\pi)$, and
\[\theta=(\theta_{1,1},\theta_{1,2},\ldots,\theta_{K,1},\theta_{K,2}) \in
\R^{2K}\]
are the Fourier coefficients of $f$. We assume implicitly throughout the paper
that $K \geq 2$, and we are interested in applications with potentially large
values of this bandlimit $K$.

Importantly, due to the choice of Fourier
basis, the $2K$-dimensional space of such bandlimited functions is closed under
rotations of the circle.
The rotation $f \mapsto f_\alpha$ induces a map from the Fourier
coefficients of $f$ to those of $f_\alpha$,
which we denote as $\theta \mapsto g(\alpha) \cdot \theta$
for an orthogonal matrix $g(\alpha) \in \R^{2K \times 2K}$. Explicitly, 
this map $\theta \mapsto g(\alpha) \cdot \theta$ is given separately for
each Fourier frequency $k=1,\ldots,K$ by
\begin{equation}\label{eq:thetarotation}
\begin{pmatrix} \theta_{k,1} \\ \theta_{k,2} \end{pmatrix} \mapsto
\begin{pmatrix} \cos k\alpha\; & -\sin k\alpha \\
\sin k\alpha & \cos k\alpha \end{pmatrix}
\begin{pmatrix} \theta_{k,1} \\ \theta_{k,2} \end{pmatrix},
\end{equation}
and $g(\alpha)$ is the block-diagonal matrix with these
$2 \times 2$ blocks. Equivalently, writing
\[(\theta_{k,1},\theta_{k,2})=(r_k \cos \phi_k, r_k\sin\phi_k)\]
where $r_k \geq 0$ is the magnitude and $\phi_k \in \cA$ is the
phase (identified modulo $2\pi$), this map is given for each $k=1,\ldots,K$ by
\begin{equation}\label{eq:rphirotation}
(r_k,\phi_k) \mapsto (r_k,\phi_k+k\alpha).
\end{equation}

The samples $f_\alpha(t)\,\der t+\sigma\,\der W(t)$ represented in this
Fourier sequence space take the form
\begin{equation}\label{eq:samples}
y^{(m)}=g(\alpha^{(m)}) \cdot \theta+\sigma \eps^{(m)} \in \R^{2K}
\text{ for } m=1,\ldots,N
\end{equation}
where $\alpha^{(1)},\ldots,\alpha^{(N)} \overset{\iid}{\sim} \Unif([-\pi,\pi))$,
$\eps^{(1)},\ldots,\eps^{(N)} \overset{\iid}{\sim} \Normal(0,I_{2K})$, and these
are independent.
Writing $\hat{\theta} \in \R^{2K}$ for the Fourier coefficients of the
estimated function $\hat{f}$ (which should likewise be bandlimited to $K$
Fourier frequencies), the loss (\ref{eq:functionloss}) is equivalent to
\begin{equation}\label{eqn-loss}
L(\hat{\theta},\theta)=\min_{\alpha \in \cA} \|\hat{\theta}-g(\alpha) \cdot \theta\|^2.
\end{equation}
In the remainder of this paper, we will consider the problem in this 
sequence form.

\revise{We reserve the notation $\theta^*$ for the Fourier coefficients of the
true unknown function. Fixing constants $\beta \in [0,\frac{1}{2})$ and
$\clower,\cupper>0$, we consider a parameter space of ``generic'' Fourier
coefficient vectors with power law decay rate $\beta$, given by
\begin{equation}\label{eq:Thetar}
\Theta_\beta=\Big\{\theta^* \in \R^{2K}:\,\clower k^{-\beta} \leq r_k(\theta^*)
\leq \cupper k^{-\beta} \text{ for all } k=1,\ldots,K\Big\}.
\end{equation}
Here, ``generic'' refers to the quantitative lower bound for
each value $r_k(\theta^*)$ that matches the
assumed upper bound up to a constant factor. This condition may be viewed as an
analogue of the genericity condition in \cite{perry2019sample} that all
Fourier magnitudes are bounded above and below by a constant, in
our high-dimensional setting of interest with potentially large $K$ and
decaying Fourier magnitudes.

Our main results are the following two theorems, which characterize the minimax
rates of estimation over $\Theta_\beta$ in high-noise and low-noise regimes.

\begin{Theorem}[Minimax risk in high noise]\label{thm:highnoise}
Fix any $\beta \in [0,\frac{1}{2})$ and any
constant $c_0>0$. If $\sigma^2 \geq c_0K^{1-2\beta}$, then for a
constant $C_0>0$ depending only on $\beta,\clower,\cupper,c_0$ and
for any $N \geq C_0K^{6\beta}\sigma^6\log K$,
\[\inf_{\hat{\theta}} \sup_{\theta^* \in \Theta_\beta}
\E_{\theta^*}[L(\hat{\theta},\theta^*)] \asymp \frac{K^{4\beta}\sigma^6}{N}.\]
\end{Theorem}
\begin{Theorem}[Minimax risk in low noise]\label{thm:lownoise}
Fix any $\beta \in [0,\frac{1}{2})$. There exist constants $C_0,C_1>0$
depending only on $\beta,\clower,\cupper$ such that if
$\sigma^2 \leq \frac{K^{1-2\beta}}{C_1\log K}$ and $N \geq
C_0K^{1+2\beta}\sigma^2\log K$, then
\[\inf_{\hat{\theta}} \sup_{\theta^* \in \Theta_\beta}
\E_{\theta^*}[L(\hat{\theta},\theta^*)] \asymp \frac{K\sigma^2}{N}.\]
\end{Theorem}}
\noindent In both statements, $\E_{\theta^*}$ is the expectation
over $N$ samples $y^{(1)},\ldots,y^{(N)}$ from the model (\ref{eq:samples})
with true parameter $\theta^*$. The infimum $\inf_{\hat{\theta}}$ is over all
estimators $\hat{\theta}$ based on these samples, and $\asymp$ denotes upper
and lower bounds up to constant multiplicative factors that depend only on
$\beta,\clower,\cupper,c_0$.

\section{Preliminaries}
\label{sec-3}
\subsection{Bounds for the loss}

For $\phi,\phi' \in \cA=[-\pi,\pi)$, we define the circular distance
\begin{equation}\label{eq:dA}
|\phi-\phi'|_\cA=\min_{j \in \Z} |\phi-\phi'+2\pi j|.
\end{equation}
It is direct to check that $(\phi,\phi') \mapsto |\phi-\phi'|_\cA$ is a metric 
on $\cA$, satisfying the triangle inequality and the upper bound
\begin{equation}\label{eq:dAupperbound}
|\phi-\phi'|_{\cA} \leq \min(\pi, |\phi-\phi'|).
\end{equation}

We may express and bound the loss (\ref{eqn-loss}) in terms of the Fourier
magnitudes and phases.

\begin{Proposition}\label{prop-lossgeneral}
Let $\theta=(r_k\cos \phi_k,r_k\sin\phi_k)_{k=1}^K$ and
$\theta'=(r_k'\cos \phi_k',r_k'\sin\phi_k')_{k=1}^K$. Then
\begin{equation}\label{eq:lossrphi}
L(\theta,\theta')=\sum_{k=1}^K (r_k-r_k')^2+\inf_{\alpha \in \R}
\sum_{k=1}^K 2r_kr_k'\Big[1-\cos(\phi_k-\phi_k'+k\alpha)\Big].
\end{equation}
Consequently, for universal constants $C,c>0$,
\begin{align*}
&\sum_{k=1}^K (r_k-r_k')^2+c\inf_{\alpha \in \R}
\sum_{k=1}^K r_kr_k'|\phi_k-\phi_k'+k\alpha|_\cA^2\\
&\hspace{1in}\leq
L(\theta,\theta')\leq \sum_{k=1}^K (r_k-r_k')^2+C\inf_{\alpha \in \R}
\sum_{k=1}^K r_kr_k'|\phi_k-\phi_k'+k\alpha|_\cA^2.
\end{align*}
\end{Proposition}
\begin{proof}
For any $\alpha \in \R$, we have
\begin{align*}
\|\theta'-g(\alpha)\cdot\theta\|^2 &= \sum_{k=1}^{K}\left[(r_k'\cos\phi_k' -
r_k\cos(\phi_k+k\alpha))^2+(r_k'\sin\phi_k' - r_k\sin(\phi_k+k\alpha))^2\right] \\
&= \sum_{k=1}^{K}
(r_k-r_k')^2+2r_kr_k'\left[1-\cos(\phi_k-\phi_k'+k\alpha)\right].
\end{align*}
Taking the infimum over $\alpha$ gives (\ref{eq:lossrphi}). The consequent
inequalities follow from the bounds
$c|t|_{\cA}^2 \leq 1-\cos(t) \leq C|t|_{\cA}^2$ for universal constants $C,c>0$,
applied with $t=\phi_k-\phi_k'+k\alpha$ for each $k$.
\end{proof}

\subsection{Complex representation}\label{sec:complexrepr}

It will be notationally and conceptually
convenient to pass between $\theta \in \R^{2K}$ 
and a complex representation by $\tilde\theta \in \C^K$. We use throughout
\begin{equation}\label{eq:principalArg}
\Arg z \in [-\pi,\pi)
\end{equation}
for the principal complex argument of $z \in \C$.
Recalling the $k^\text{th}$ Fourier coefficient
pair $(\theta_{k,1},\theta_{k,2})=(r_k\cos\phi_k,r_k\sin\phi_k)$, we set
\begin{equation}\label{eq:complextheta}
\tilde\theta_k=\theta_{k,1}+i\theta_{k,2}=r_ke^{i\phi_k} \in \C.
\end{equation}
For $\theta,\theta' \in \R^{2K}$, note then that
\begin{equation}\label{eq:CRisometry}
\langle \theta,\theta' \rangle
=\sum_{k=1}^K \theta_{k,1}\theta_{k,1}'+\theta_{k,2}\theta_{k,2}'
=\sum_{k=1}^K \Re \tilde\theta_k\overline{\tilde\theta_k'}
=\frac{\langle \tilde\theta,\tilde\theta' \rangle+\langle \tilde\theta',
\tilde\theta \rangle}{2}
\end{equation}
where the left side is the real inner-product, and the right side is the complex
inner-product $\langle u,v \rangle=\sum_k u_k\overline{v_k}$.

Similarly, we may represent the sample $y^{(m)} \in \R^{2K}$ from
(\ref{eq:samples}) by $\tilde y^{(m)} \in \C^K$ where
\[\tilde y^{(m)}_k=y^{(m)}_{k,1}+iy^{(m)}_{k,2} \in \C.\]
Then, recalling the form of the rotational action (\ref{eq:rphirotation}), we
have
\begin{equation}\label{eq:tildey}
\tilde y_k^{(m)}=r_k e^{i(\phi_k+k\alpha^{(m)})}
+\sigma \tilde{\eps}_k^{(m)} \in \C
\end{equation}
where $\tilde{\eps}_k^{(m)}=\eps_{k,1}^{(m)}+i\eps_{k,2}^{(m)} \sim
\Normal_\C(0,2)$ is complex Gaussian noise, independent across both frequencies
$k=1,\ldots,K$ and samples $m=1,\ldots,N$.

\section{Method-of-moments estimator}
\label{sec-4}
\revise{In this section, we analyze an estimator based on a
third-order method-of-moments idea. We prove a general risk bound that depends on
the smallest non-zero Fourier magnitude $\rlower=\min_k r_k(\theta^*)$ of the
true signal, valid for any noise level $\sigma^2>0$, and we show in particular
that this achieves the minimax upper bound of Theorem \ref{thm:highnoise} for
signals $\theta^* \in \Theta_\beta$ in the high-noise regime.}

Throughout this section, let us denote the Fourier magnitudes and phases of the
true parameter as
$\theta^*=(r_k\cos\phi_k,r_k\sin\phi_k)_{k=1}^K$ and write $\E$ for
$\E_{\theta^*}$. Observe from (\ref{eq:tildey}) that for every $k=1,\ldots,K$,
\[\E\big[|\tilde y_k^{(m)}|^2\big]=r_k^2+2\sigma^2.\]
Then $N^{-1}\sum_{m=1}^N |\tilde y_k^{(m)}|^2-2\sigma^2$ provides an unbiased
estimate of $r_k^2$. Furthermore, denote
\begin{equation}\label{eq:indexset}
\cI=\Big\{(k,l):k,l \in \{1,\ldots,K\} \text{ and } k+l \leq K\Big\}.
\end{equation}
Applying that $\{\tilde{\eps}_k^{(m)}:k=1,\ldots,K\}$ are independent with mean
0, and also $\E[(\tilde{\eps}_k^{(m)})^2]=0$ (cf.\ Proposition
\ref{lemma-distribution-1} of Appendix \ref{appendix:MoM}),
for any $(k,l) \in \cI$ including the case $k=l$ we have
\begin{equation*}
\begin{aligned}
\E\left[\tilde y_{k+l}^{(m)}\cdot\overline{\tilde
y_k^{(m)}}\cdot\overline{\tilde y_l^{(m)}}\right]&=
\E\left[r_{k+l}e^{i(\phi_{k+l}+(k+l)\alpha^{(m)})} \cdot
r_ke^{i(-\phi_k-k\alpha^{(m)})} \cdot r_l e^{i(-\phi_l-l\alpha^{(m)})}\right]\\
&=r_{k+l}r_kr_l e^{i(\phi_{k+l}-\phi_k-\phi_l)}.
\end{aligned}    
\end{equation*}
Thus the complex argument of $N^{-1} \sum_{m=1}^N \tilde y_{k+l}^{(m)} \cdot
\overline{\tilde y_k^{(m)}}\cdot \overline{\tilde y_l^{(m)}}$
provides an estimate of the Fourier bispectrum component
$\phi_{k+l}-\phi_k-\phi_l$
modulo $2\pi$, from which we may hope to recover the individual phases $\phi_k$.

This motivates the following class of method-of-moments procedures:
\begin{enumerate}
\item[1.] For each $k=1,\ldots,K$, estimate $r_k$ by
\begin{equation}\label{eq:hatrk}
\hat{r}_k=\left(\frac{1}{N}\sum_{m=1}^N |\tilde y_k^{(m)}|^2
-2\sigma^2\right)_+^{1/2}.
\end{equation}
\item[2.] For each $(k,l) \in \cI$, compute
\begin{equation}\label{eq:hatBkl}
\hat{B}_{k,l}=\frac{1}{N}\sum_{m=1}^N
\tilde y_{k+l}^{(m)} \cdot \overline{\tilde y_k^{(m)}} \cdot
\overline{\tilde y_l^{(m)}},
\end{equation}
and choose a version of its complex argument $\hat{\Phi}_{k,l}$ in $\R$
such that $\hat{\Phi}_{k,l}-\Arg \hat{B}_{k,l}=0 \bmod 2\pi$.
\item[3.] Estimate $\phi=(\phi_k:k=1,\ldots,K)$ by the least-squares estimator
\begin{equation}\label{eq:leastsquares}
\hat{\phi}=\argmin_{\phi \in \R^K} \sum_{(k,l) \in \cI} \big(\hat{\Phi}_{k,l}
-(\phi_{k+l}-\phi_k-\phi_l)\big)^2.
\end{equation}
Then estimate $\theta$ by
$\hat{\theta}=(\hat{r}_k\cos\hat{\phi}_k,\hat{r}_k\sin\hat{\phi}_k)_{k=1}^K$.
\end{enumerate}

Here, (\ref{eq:leastsquares}) is defined
using the squared difference over $\R$ rather than over the periodic domain $\cA$.
Hence the final estimate $\hat{\theta}$ depends on the specific choice of
argument $\hat{\Phi}_{k,l}$ in Step 2, which we have left ambiguous above.
We proceed by first studying in Section \ref{sec:MoMoracle} an ``oracle''
version of this estimator, where $\hat{\Phi}_{k,l}$ is chosen in Step 2 
using knowledge of the true phases $\phi_1,\ldots,\phi_K$ as the unique version
of the argument of $\hat{B}_{k,l}$ for which
$\hat{\Phi}_{k,l}-(\phi_{k+l}-\phi_k-\phi_l) \in [-\pi,\pi)$. This choice
satisfies an exact distributional symmetry in sign. We leverage this symmetry
to provide a risk bound for this oracle procedure.

To develop an actual estimator based on this oracle idea, we propose in
Section \ref{sec:MoMopt} a method of mimicking this oracle using a pilot
estimate of $\phi_1,\ldots,\phi_K$ that is obtained by first minimizing an
$\ell_\infty$-type optimization objective. We prove an $\ell_\infty$-stability
bound for bispectrum inversion, which implies that the resulting choice of
$\hat{\Phi}_{k,l}$ coincides with the oracle choice with high probability as
long as $N \gtrsim \frac{\sigma^6}{\rlower^6}\log K$. Consequently, this estimator
attains the same estimation rate without oracle knowledge. \revise{We summarize
these results as the following theorem.}

\revise{
\begin{Theorem}\label{thm:oracleMoM}
Let $\hat{\theta} \in \{\hat{\theta}^\oracle,\hat{\theta}^\opt\}$
be the above method-of-moments estimator, where $\hat{\Phi}_{k,l}$
is chosen either using the oracle of Section \ref{sec:MoMoracle} or the
optimization procedure of Section \ref{sec:MoMopt}.
Suppose $r_k \geq \rlower>0$ for each $k=1,\ldots,K$.
There exist universal constants $C,C_0>0$ such that if
$N \geq C_0(\frac{\sigma^6}{\rlower^6}\log K+\frac{\sigma^3}{\rlower^3}(\log
K)^{3/2})$, then
\begin{equation}\label{eq:oracleMoMbound}
\E[L(\hat{\theta},\theta^*)] \leq CK
\left(\frac{\sigma^2}{N}+\frac{\sigma^4}{N\rlower^2}\right)
+\frac{C\|\theta^*\|^2}{K}
\left(\frac{K\sigma^2}{N\rlower^2}+\frac{\sigma^6}{N\rlower^6}\right)
\end{equation}
\end{Theorem}

\noindent We remark that for signals where $\|\theta^*\|^2/K \asymp \rlower^2$,
as is the case for our signal class $\Theta_\beta$ of interest, this risk
bound reduces to
\[\E[L(\hat{\theta},\theta^*)] \leq \frac{C}{N}
\left(K\sigma^2+\frac{K\sigma^4}{\rlower^2}
+\frac{\sigma^6}{\rlower^4}\right)\]}

\subsection{The oracle procedure}\label{sec:MoMoracle}

Let us identify each entry of the true Fourier phase vector as a real
value $\phi_k \in [-\pi,\pi)$, and set
\begin{equation}\label{eq:Phidef}
\Phi_{k,l}=\phi_{k+l}-\phi_k-\phi_l \in \R.
\end{equation}
We emphasize that this arithmetic is carried out in $\R$, not modulo $2\pi$.
We consider an oracle version of the above method-of-moments procedure, where
$\hat{\Phi}_{k,l}^\oracle \in [\Phi_{k,l}-\pi,\Phi_{k,l}+\pi)$
is chosen in Step 2 as the unique version of the complex argument of
$\hat{B}_{k,l}$ that belongs to this range. Recalling the complex
representation of $\theta$ in (\ref{eq:complextheta}) and defining
\begin{equation}\label{eq:Bkl}
B_{k,l}=\tilde\theta_{k+l} \cdot \overline{\tilde\theta_k}
\cdot \overline{\tilde\theta_l}
=r_{k+l}r_kr_l e^{i(\phi_{k+l}-\phi_k-\phi_l)} \in \C,
\end{equation}
note that this means, for the principal argument specified in
(\ref{eq:principalArg}),
\begin{equation}\label{eq:Argoracle}
\hat{\Phi}_{k,l}^\oracle-\Phi_{k,l}=\Arg (\hat{B}_{k,l}/B_{k,l}) \in
[-\pi,\pi).
\end{equation}
We will write $\hat{\Phi}^\oracle=\hat{\Phi}^\oracle(\phi)$
if we wish to make explicit the dependence of this definition on the phase
vector $\phi$ of the true signal.
We denote by $\hat{\phi}^\oracle$ the resulting least-squares estimate of
$\phi$ in (\ref{eq:leastsquares}),
and by $\hat{\theta}^\oracle$ the corresponding
estimate of $\theta$.

In the remainder of this subsection, we describe an
argument showing that Theorem \ref{thm:oracleMoM} holds for
$\hat{\theta}^\oracle$,
deferring detailed proofs to Appendix \ref{appendix:MoM}.
We divide the argument into the analysis of Step 1
of the MoM procedure for estimating the Fourier magnitudes $\{r_k\}_{k=1}^K$, 
Step 2 for estimating the bispectrum components $\{\Phi_{k,l}\}_{(k,l) \in
\cI}$, and Step 3 for recovering the phases $\{\phi_k\}_{k=1}^K$ from the
bispectrum.\\

{\bf Estimating $r_k$.} Standard Gaussian and chi-squared tail bounds
show the following guarantee for estimating the Fourier magnitudes
$r_k$ via $\hat{r}_k$, defined in (\ref{eq:hatrk}).

\begin{Lemma}\label{lemma:rktail}
For each $k=1,\ldots,K$ and a universal constant $c>0$,
\begin{align}
\P[\hat{r}_k \geq r_k(1+s)]&\leq
2\exp\left(-cNs^2\left(\frac{r_k^2}{\sigma^2} \wedge
\frac{r_k^4}{\sigma^4}\right)\right) \text{ for all } s \geq
0,\label{eq:hatrupper}\\
\P[\hat{r}_k \leq r_k(1-s)]&\leq 
2\exp\left(-cNs^2\left(\frac{r_k^2}{\sigma^2} \wedge
\frac{r_k^4}{\sigma^4}\right)\right) \text{ for all } s \in [0,1).\label{eq:hatrlower}
\end{align}
\end{Lemma}

\noindent Integrating these tail bounds yields the following
immediate corollary.

\begin{Corollary}\label{cor:MoMrriskbound}
For each $k=1,\ldots,K$ and a universal constant $C>0$,
\[\E[(\hat{r}_k-r_k)^2] \leq C\left(\frac{\sigma^2}{N}+
\frac{\sigma^4}{Nr_k^2}\right).\]
\end{Corollary}

{\bf Estimating $\Phi_{k,l}$.} Applying a concentration inequality for cubic
polynomials in independent Gaussian random variables, derived
from \cite{latala2006estimates}, we obtain the following tail bounds for
estimating $B_{k,l}$ by $\hat{B}_{k,l}$ in Step 2, and for estimating
the bispectrum component $\Phi_{k,l}$ by the oracle
estimator $\hat{\Phi}_{k,l}^\text{oracle}$.

\begin{Lemma}\label{lemma:Phikltail}
Consider any $(k,l) \in \cI$ and suppose
$r_{k+l},r_k,r_l \geq \rlower$.
Then for universal constants $C,c>0$ and any $s>0$,
\begin{equation}
\label{eqn-B-final-bound}
\P\Big[|\hat{B}_{k,l}/B_{k,l}-1| \geq s\Big]
\leq C\exp\left(-c\left(\frac{Ns^2\rlower^2}{\sigma^2}
\wedge \frac{Ns^2\rlower^6}{\sigma^6} \wedge \frac{(Ns)^{2/3}\rlower^2}
{\sigma^2}\right)\right).
\end{equation}
Furthermore, for universal constants $C,c>0$ and any $s \in (0,\pi/2)$,
\begin{equation}\label{eqn-Phi-final-bound}
\P\Big[|\hat{\Phi}_{k,l}^\oracle-\Phi_{k,l}| \geq s\Big]
\leq C\exp\left(-c\left(\frac{Ns^2\rlower^2}{\sigma^2}
\wedge \frac{Ns^2\rlower^6}{\sigma^6} \wedge
\frac{(Ns)^{2/3}\rlower^2}{\sigma^2}\right)\right).
\end{equation}
\end{Lemma}

\begin{Corollary}\label{cor:riskPhikl}
Consider any $(k,l) \in \cI$ and suppose $r_{k+l},r_k,r_k \geq \rlower$.
Then for a universal constant $C>0$,
\[\E[(\hat\Phi_{k,l}^\oracle-\Phi_{k,l})^2] \leq C\left(
\frac{\sigma^2}{N\rlower^2}+\frac{\sigma^6}{N\rlower^6}\right)\]
\end{Corollary}

A key property of the oracle estimator $\hat{\Phi}_{k,l}^\text{oracle}$ is an
exact distributional symmetry in sign,
\begin{equation}\label{eq:symmetryinsign}
\hat{\Phi}_{k,l}^\text{oracle}-\Phi_{k,l}
\overset{L}{=}-\hat{\Phi}_{k,l}^\text{oracle}+\Phi_{k,l}.
\end{equation}
\revise{This implies that
$\E[\hat{\Phi}_{k,l}^\text{oracle}-\Phi_{k,l}]=0$, and hence
$\E[(\hat{\Phi}_{k,l}^\text{oracle}-\Phi_{k,l})
(\hat{\Phi}_{x,y}^\text{oracle}-\Phi_{x,y})]=0$ when these bispectral components
do not have any overlapping index, as stated in part (a) of the following
lemma.

For $\Phi_{k,l}$ and $\Phi_{x,y}$ that have an
overlapping index, the corresponding estimates $\hat{\Phi}_{k,l}^\text{oracle}$
and $\hat{\Phi}_{x,y}^\text{oracle}$ are not independent.
Our proof of Theorem \ref{thm:oracleMoM} requires a sharper
bound on the expected product of their errors
than what is naively obtained from the preceding Corollary \ref{cor:riskPhikl}
and Cauchy-Schwarz. Indeed, applying the representation
(\ref{eq:Argoracle}) and a first-order Taylor approximation $\Arg z=\Im \Ln z
\approx \Im (z-1)$ around $z=1$, we obtain
$\E[(\hat\Phi_{k,l}^\oracle-\Phi_{k,l})(\hat\Phi_{x,y}^\oracle-\Phi_{x,y})]
\approx \E[\Im(\hat{B}_{k,l}/B_{k,l}-1)
\Im(\hat{B}_{x,y}/B_{x,y}-1)]$, and it is easily checked that this latter
expectation is of size $O(\sigma^2/N\rlower^2)$, exhibiting a cancellation of
the $O(\sigma^6/N\rlower^6)$ error.
However, a naive bound for the error of this Taylor
approximation remains of size $O(\sigma^6/N\rlower^6)$.
Part (b) of the following lemma establishes a sharp bound for
$\E[(\hat\Phi_{k,l}^\oracle-\Phi_{k,l})(\hat\Phi_{x,y}^\oracle-\Phi_{x,y})]$
by carrying out the Taylor expansion to a higher order
$J \asymp N\rlower^6/\sigma^6$ with a remainder that is exponentially small
in $N\rlower^6/\sigma^6$, and exhibiting a similar cancellation in
expectation for all terms of the Taylor expansion up to this order $J$.}

\begin{Lemma}\label{lemma:Phicovariance}
Let $(k,l),(x,y) \in \cI$, and suppose $r_k,r_l,r_{k+l},r_x,r_y,r_{x+y} \geq
\rlower$. For some universal constants $C,c>0$,
\begin{enumerate}[(a)]
\item If $\{k,l,k+l\}$ is disjoint from $\{x,y,x+y\}$, then
\[\E[(\hat\Phi_{k,l}^\oracle-\Phi_{k,l})(\hat\Phi_{x,y}^\oracle-\Phi_{x,y})]=0.\]
\item If $\{k,l,k+l\} \cap \{x,y,x+y\}$ has cardinality 1, then
\begin{equation}\label{eq:card1bound}
\Big|\E[(\hat\Phi_{k,l}^\oracle-\Phi_{k,l})(\hat\Phi_{x,y}^\oracle-\Phi_{x,y})]\Big| \leq C\left(\frac{\sigma^2}{N\rlower^2}
+e^{-c\left(\frac{N\rlower^6}{\sigma^6}
\wedge \frac{N^{2/3}\rlower^2}{\sigma^2}\right)}\right).
\end{equation}
\item For any $(k,l),(x,y) \in \cI$,
\[\Big|\E[(\hat\Phi_{k,l}^\oracle-\Phi_{k,l})(\hat\Phi_{x,y}^\oracle-\Phi_{x,y})]\Big|
\leq C\left(\frac{\sigma^2}{N\rlower^2}+\frac{\sigma^6}{N\rlower^6}\right).\]
\end{enumerate}
\end{Lemma}

{\bf Estimating $\phi_k$.}
We now translate the preceding bounds for estimating the Fourier
bispectrum $\{\Phi_{k,l}\}$ to estimating the phases $\{\phi_k\}$ 
using the least squares procedure (\ref{eq:leastsquares}).

Define the matrix $M \in \R^{\cI \times K}$ with rows indexed by the bispectrum
index set $\cI$
from (\ref{eq:indexset}), such that the linear system
(\ref{eq:Phidef}) may be expressed as $\Phi=M\phi$.
That is, row $(k,l)$ of $M$ is given by $e_{k+l}-e_k-e_l$ where $e_k \in \R^K$
is the $k^\text{th}$ standard basis vector. Then (\ref{eq:leastsquares})
is given explicitly by
\begin{equation}\label{eq:Mdagger}
\hat{\phi}=M^\dagger \hat{\Phi}
\end{equation}
where $M^\dagger$ is the Moore-Penrose pseudo-inverse.

Recall that a rotation of the circular domain of $f$ induces the
map (\ref{eq:rphirotation}), which does not change the bispectral components
$\Phi_{k,l}$. This is reflected by the property that $(1,2,3,\ldots,K)$
belongs to the kernel of $M$. The following lemma shows that this is the unique
vector in the kernel. Furthermore, $M$ is well-conditioned on the subspace
orthogonal to
this kernel, with all remaining $K-1$ singular values on the same order of
$\sqrt{K}$.

\begin{Lemma}\label{lemma:Mproperties}
$M$ has rank exactly $K-1$, and the kernel of $M$ is the 
span of $(1,2,3,\ldots,K) \in \R^K$.
All $K-1$ non-zero eigenvalues of $M^{\top}M \in \R^{K\times K}$ are integers
in the interval $[K+1,2K+1]$.
\end{Lemma}

This yields the following corollary for estimation of the Fourier phases
$\{\phi_k\}$, up to a global rotation that is represented by an additive shift
in the direction of $(1,2,3,\ldots,K)$.

\begin{Corollary}\label{cor:MoMphiriskbound}
Suppose $r_k \geq \rlower$ for each $k=1,\ldots,K$. Then for universal constants
$C,c>0$,
\revise{
\begin{equation}\label{eq:phiriskbound}
\E\left[\inf_{\alpha \in \R} \sum_{k=1}^K r_k^2
|\hat\phi_k^\oracle-\phi_k+k\alpha|_\cA^2\right] \leq 
\frac{C\|\theta^*\|^2}{K}
\left(\frac{K\sigma^2}{N\rlower^2}+\frac{\sigma^6}{N\rlower^6}
+Ke^{-c(\frac{N\rlower^6}{\sigma^6}
\wedge \frac{N^{2/3}\rlower^2}{\sigma^2})}\right).
\end{equation}}
\end{Corollary}
\begin{proof}
By adding a multiple of $(1,2,3,\ldots,K)$ to $\phi$ and absorbing this shift
into $\alpha$,
we may assume without loss of generality that $\phi$ is orthogonal to
$(1,2,3,\ldots,K)$. Under this assumption, we will then upper-bound the left
side by choosing $\alpha=0$. Since $\Phi=M\phi$,
this implies $M^\dagger \Phi=M^\dagger M\phi=\phi$, the last equality holding
because Lemma \ref{lemma:Mproperties} implies that $M^\dagger M$ is the
projection orthogonal to $(1,2,3,\ldots,K)$. \revise{Set $D=\diag(r_k^2)_{k=1}^K
\in \R^{K \times K}$.} Then applying $\Tr AB \leq \Tr B \cdot
\|A\|_{\mathrm{op}}$ for positive semidefinite $A,B$, where
$\|\cdot\|_{\mathrm{op}}$ is the $\ell_2 \to \ell_2$ operator norm,
\revise{
\begin{align*}
\E\left[\sum_{k=1}^K r_k^2|\hat\phi_k^\oracle-\phi_k|_\cA^2\right] &\leq 
\E[(\hat\phi^\oracle-\phi)^\top D(\hat\phi^\oracle-\phi)]
=\E[(\hat\Phi^\oracle-\Phi)^\top M^{\dagger\top}
DM^\dagger(\hat\Phi^\oracle-\Phi)]\\
&=\Tr M^{\dagger\top} DM^\dagger\E[(\hat\Phi^\oracle-\Phi)
(\hat\Phi^\oracle-\Phi)^{\top}]\\
&\leq
\Tr(M^{\dagger\top}
DM^{\dagger})\cdot\|\E[(\hat\Phi^\oracle-\Phi)(\hat\Phi^\oracle-\Phi)^{\top}]\|_{\mathrm{op}}\\
&\leq \Tr D \cdot \|M^{\dagger}M^{\dagger\top}\|_{\mathrm{op}}
\cdot\|\E[(\hat\Phi^\oracle-\Phi)(\hat\Phi^\oracle-\Phi)^{\top}]\|_{\mathrm{op}}
\end{align*}
Here, $\Tr D=\sum_{k=1}^K r_k^2=\|\theta^*\|^2$, and
Lemma \ref{lemma:Mproperties} implies
$\|M^{\dagger}M^{\dagger\top}\|_{\mathrm{op}}
=\|(M^\top M)^\dagger\|_{\mathrm{op}} \leq 1/(K+1)$.}

We have $\|A\|_{\mathrm{op}} \leq \|A\|_\infty$ for positive semidefinite $A$,
where $\|A\|_\infty$ is the $\ell_\infty \to \ell_\infty$ operator norm given by
the maximum absolute row sum.
For a universal constant $C>0$ and each $(k,l) \in \cI$, there are at most $C$
pairs $(x,y) \in \cI$ for which $\{k,l,k+l\} \cap \{x,y,x+y\}$ has cardinality 2
or 3, and at most $CK$ pairs $(x,y) \in \cI$ for which
$\{k,l,k+l\} \cap \{x,y,x+y\}$ has cardinality 1. Applying
Lemma \ref{lemma:Phicovariance}(b) for those pairs for which this cardinality is
1, Lemma \ref{lemma:Phicovariance}(c) for those pairs for which this cardinality
is 2 or 3, and Lemma \ref{lemma:Phicovariance}(a) for all remaining pairs,
we obtain for different universal constants $C,c>0$ that
\[\|\E[(\hat\Phi^\oracle-\Phi)(\hat\Phi^\oracle-\Phi)^{\top}]\|_\infty
\leq C\left(\frac{K\sigma^2}{N\rlower^2}+\frac{\sigma^6}{N\rlower^6}
+Ke^{-c(\frac{N\rlower^6}{\sigma^6}
\wedge \frac{N^{2/3}\rlower^2}{\sigma^2})}\right).\]
Combining the above concludes the proof.
\end{proof}

Let us remark that using Lemma \ref{lemma:Phicovariance}(b) in place of
Lemma \ref{lemma:Phicovariance}(c) for the pairs where $\{k,l,k+l\}$ and
$\{x,y,x+y\}$ overlap in one index is important for removing a factor of $K$ in
the $\sigma^6/(N\rlower^6)$ component of the error, which will be the leading
contribution to the overall estimation error in the high-noise regime.

Theorem \ref{thm:oracleMoM} for $\hat\theta^\oracle$
now follows from the loss upper bound in Proposition
\ref{prop-lossgeneral} in terms of the separate estimation errors for magnitude
and phase, together with
Corollaries \ref{cor:MoMrriskbound} and \ref{cor:MoMphiriskbound}.

\subsection{Mimicking the oracle}\label{sec:MoMopt}

We now consider the method-of-moments procedure where the choice of
$\hat{\Phi}_{k,l}$ in Step 2 is determined instead by the following method: 
Compute a ``pilot'' estimate of $\phi$ as any minimizer of
the $\ell_\infty$-type objective
\begin{equation}\label{eq:tildephidef}
\tilde{\phi}=\argmin_{\phi \in \cA^K} \max_{(k,l) \in \cI}
|\Arg \hat{B}_{k,l}-(\phi_{k+l}-\phi_k-\phi_l)|_\cA,
\end{equation}
where a minimizer exists because $\cA$ is compact under $|\cdot|_{\cA}$.
Identify each entry $\tilde{\phi}_k \in [-\pi,\pi)$ of this estimate as a real
value, and set
$\tilde{\Phi}_{k,l}=\tilde{\phi}_{k+l}-\tilde{\phi}_k-\tilde{\phi}_l$
where arithmetic is again carried out in $\R$, not modulo $2\pi$.
Then choose
$\hat{\Phi}_{k,l}^\opt \in [\tilde{\Phi}_{k,l}-\pi,\tilde{\Phi}_{k,l}+\pi)$
as the unique version of the complex argument of $\hat{B}_{k,l}$ belonging to this range.
Let $\hat\phi^\opt$ be the resulting least-squares estimate of $\phi$ in
(\ref{eq:leastsquares}), and let $\hat\theta^\opt$ be
the corresponding estimate of $\theta$.

\revise{We prove Theorem \ref{thm:oracleMoM} for $\hat\theta^\opt$
by showing that, with high
probability, $\hat{\Phi}^\opt=\hat{\Phi}^\oracle(\phi')$ for some
phase vector $\phi'$ that is equivalent to $\phi$.
By ``equivalent'', we mean that $\phi$ and $\phi'$ represent the same
Fourier phases up to rotation of the circular domain, i.e.\ there
exists $\alpha \in \R$ for which
\begin{equation}\label{eq:phiequivalence}
|\phi_k'-\phi_k+k\alpha|_\cA=0 \text{ for each } k=1,\ldots,K.
\end{equation}
Then using $\hat{\Phi}^\opt$ achieves the same loss as using
$\hat{\Phi}^\oracle(\phi)$.}
The main additional ingredient in the proof is a deterministic
$\ell_\infty$-stability bound for recovery of the Fourier
phases from the bispectrum, stated in the following result.

\begin{Lemma}\label{lemma:phasestability}
Fix any $\delta \in (0,\pi/3)$ and $\phi,\phi' \in \R^K$. Denote
$\Phi_{k,l}=\phi_{k+l}-\phi_k-\phi_l$ and
$\Phi_{k,l}'=\phi_{k+l}'-\phi_k'-\phi_l'$. If
\[|\Phi_{k,l}-\Phi_{k,l}'|_\cA \leq \delta \text{ for all } (k,l) \in \cI,\]
then there exists some $\alpha \in \R$ such that
\[|\phi_k-\phi_k'-k\alpha|_\cA \leq \delta \text{ for all } k=1,\ldots,K.\]
\end{Lemma}

\noindent This guarantees that, if $\tilde{\phi}$ yields a bispectrum
$\tilde{\Phi}$ which is elementwise close to the true bispectrum $\Phi$ in the
circular distance modulo $2\pi$, then $\tilde{\phi}$ must also be elementwise
close to $\phi$ up to a rotation of the circular domain. \revise{In other words,
this is an $\ell_\infty \to \ell_\infty$ operator-norm bound
for the matrix $M^\dagger$ from (\ref{eq:Mdagger}), where the $\ell_\infty$
norms are defined using the circular distance per coordinate
and modulo the equivalence relation (\ref{eq:phiequivalence}).}

The above guarantee is sufficient to show that
if each quantity $\Arg \hat{B}_{k,l}$ estimates the true bispectral component
$\Phi_{k,l}$ up to a small constant error in the circular distance
$|\cdot|_\cA$, then its version $\hat{\Phi}_{k,l}^\text{opt}$ that is
chosen using $\tilde{\phi}$ must coincide exactly with the oracle choice
$\hat{\Phi}_{k,l}^\text{oracle}(\phi')$, based on a phase vector $\phi'$ that is
equivalent to the true phase vector $\phi$.

\begin{Corollary}\label{cor:optmimicsoracle}
Let $\hat{B}_{k,l}$ be as defined in (\ref{eq:hatBkl}), and suppose $\phi \in
\R^K$ is such that
\begin{equation}\label{eq:linftyPhibound}
|\Arg \hat{B}_{k,l}-(\phi_{k+l}-\phi_k-\phi_l)|_\cA
<\pi/12 \text{ for every } (k,l) \in \cI.
\end{equation}
Then there exists $\phi'$ equivalent to $\phi$ such that
$\hat\Phi^\opt=\hat\Phi^\oracle(\phi')$.
\end{Corollary}
\begin{proof}
By the definition of the optimization procedure which defines
$\tilde{\phi}$ in (\ref{eq:tildephidef}),
\begin{equation}\label{eq:tildePhibound}
\max_{(k,l) \in \cI}
|\Arg \hat{B}_{k,l}-(\tilde{\phi}_{k+l}-\tilde{\phi}_k-\tilde{\phi}_l)|_\cA
\leq \max_{(k,l) \in \cI}
|\Arg \hat{B}_{k,l}-(\phi_{k+l}-\phi_k-\phi_l)|_\cA.
\end{equation}
By assumption, the right side is at most $\pi/12$.
Then by the triangle inequality for $|\cdot|_\cA$, for every $(k,l) \in \cI$,
we have $|(\tilde{\phi}_{k+l}-\tilde{\phi}_k-\tilde{\phi}_l)
-(\phi_{k+l}-\phi_k-\phi_l)|_\cA<\pi/6$.
Applying Lemma \ref{lemma:phasestability},
we obtain for some $\alpha \in \R$ and all $k=1,\ldots,K$
that $|\tilde{\phi}_k-\phi_k-k\alpha|_\cA<\pi/6$.
This means that there exists $\phi'$ equivalent to $\phi$ for which, for the
usual absolute value,
\[|\tilde{\phi}_k-\phi_k'|<\pi/6 \text{ for all } k=1,\ldots,K.\]
Then denoting $\Phi_{k,l}'=\phi_{k+l}'-\phi_k'-\phi_l'$, by the triangle
inequality, $|\tilde\Phi_{k,l}-\Phi_{k,l}'|<\pi/2$ for all $(k,l) \in \cI$.
Since $\phi'$ is equivalent to $\phi$, also
\[|\Arg \hat{B}_{k,l}-(\phi_{k+l}'-\phi_k'-\phi_l')|_\cA
=|\Arg \hat{B}_{k,l}-(\phi_{k+l}-\phi_k-\phi_l)|_\cA<\pi/12.\]
So by the definition of $\hat{\Phi}^\oracle(\phi')$,
we have $|\hat{\Phi}_{k,l}^\oracle(\phi')-\Phi_{k,l}'|<\pi/12$ for the usual
absolute value. Then $|\hat{\Phi}_{k,l}^\oracle(\phi')-\tilde\Phi_{k,l}|
<\pi/2+\pi/12<\pi$ for all $(k,l) \in \cI$, meaning that
$\hat{\Phi}^\oracle(\phi')=\hat{\Phi}^\opt$.
\end{proof}

The tail bounds of Lemma \ref{lemma:Phikltail} may be used to show that
the event (\ref{eq:linftyPhibound}) holds with high probability.
On this event, the loss of $\hat{\theta}^\opt$ matches exactly that of
$\hat{\theta}^\oracle$.
Combining with a crude bound for the loss on the complementary event, which has
exponentially small probability in $N$, we obtain
Theorem \ref{thm:oracleMoM} for $\hat\theta^\opt$.

\revise{
\begin{Remark}
We study this two-stage estimation procedure primarily to
enable a theoretical analysis of its risk. One may alternatively consider a
more direct procedure where the least-squares objective (\ref{eq:leastsquares})
is defined using the squared distance
$|\hat{\Phi}_{k,l}-(\phi_{k+l}-\phi_k-\phi_l)|_{\cA}^2$ over the periodic domain
$\cA$, which would avoid the
need to identify a version of $\hat{\Phi}_{k,l}$. However, analyzing the
risk of such a procedure may require an $\ell_2$-analogue
of the stability guarantee of Lemma \ref{lemma:phasestability}, which seems
more challenging to obtain. Here, stability in the $\ell_\infty$ sense
allows us to circumvent this issue by first estimating the oracle choices of
$\hat{\Phi}_{k,l}$ using the $\ell_\infty$-objective (\ref{eq:tildephidef}).
\end{Remark}
}

\revise{Finally, let us check that this estimation guarantee in Theorem
\ref{thm:oracleMoM} coincides with our stated minimax rate in
Theorem \ref{thm:highnoise} when restricted to parameters $\theta^* \in
\Theta_\beta$ and to the high-noise regime.

\begin{proof}[Proof of Theorem \ref{thm:highnoise}, upper bound]
For $\theta^* \in \Theta_\beta$, we have $\rlower^2 \geq cK^{-2\beta}$
and $\|\theta^*\|^2 \leq CK^{1-2\beta}$, for ($\beta$-dependent) constants
$C,c>0$. Thus the risk bound of Theorem \ref{thm:oracleMoM} reduces to
\[\E[L(\hat{\theta}^\opt,\theta^*)] \leq \frac{C}{N}\left(K\sigma^2
+K^{1+2\beta}\sigma^4+K^{4\beta}\sigma^6\right)
\leq \frac{C'K^{4\beta}\sigma^6}{N}\]
for constants $C,C'>0$, the last inequality holding in the high-noise setting
$\sigma^2 \geq c_0K^{1-2\beta}$. In this setting, there is a constant $c>0$ for
which
\[\frac{\sigma^6}{\rlower^6}\log K \geq \frac{c\sigma^3}{\rlower^3}(\log
K)^{3/2}.\]
Then the required condition for $N$ in Theorem \ref{thm:oracleMoM}
is implied by $N \geq C_0'K^{6\beta}\sigma^6\log K$ for a sufficiently 
large constant $C_0'>0$, and this yields the minimax upper bound
of Theorem \ref{thm:highnoise}.
\end{proof}}

We remark that Theorem \ref{thm:oracleMoM} gives an estimation
guarantee not just in the high-noise regime, but for any noise level $\sigma^2$.
In a regime of \emph{very} low noise \revise{$\sigma^2 \lesssim K^{-2\beta}$},
it also implies the upper bound of Theorem \ref{thm:lownoise}.

\revise{
\begin{proof}[Proof of Theorem \ref{thm:lownoise}, upper bound, for
$\sigma^2 \leq K^{-2\beta}$]
For $\sigma^2 \leq K^{-2\beta}$, the risk bound of
Theorem \ref{thm:oracleMoM} reduces instead to
\[\E[L(\hat{\theta}^\opt,\theta^*)] \leq \frac{C}{N}\left(K\sigma^2
+K^{1+2\beta}\sigma^4+K^{4\beta}\sigma^6\right)
\leq \frac{C'K\sigma^2}{N}\]
The required condition for $N$
is implied by $N \geq C_0'K^{1+2\beta}\sigma^2\log K$ for a
sufficiently large constant $C_0'>0$, and this yields the minimax upper bound
of Theorem \ref{thm:lownoise}.
\end{proof}}

In high dimensions $K$ and the noise regime \revise{$K^{-2\beta} \ll \sigma^2
\ll K^{1-2\beta}/\log K$,
(\ref{eq:oracleMoMbound}) exhibits the rate $K^{1+2\beta}\sigma^4/N$} which is
larger than the minimax rate $K\sigma^2/N$. This arises from estimating the
Fourier magnitudes $\{r_k\}$ without using phase information. In this regime,
the above method-of-moments procedure becomes suboptimal. We will instead analyze in Section \ref{sec-5} the maximum likelihood estimator, to establish the
minimax rate over the entire low-noise regime described by
Theorem \ref{thm:lownoise}.

\begin{Remark}
This proof of the minimax upper bound is information-theoretic in nature, in
that the pilot estimate used to mimic the oracle may require exponential time
in $K$ to compute.
We describe in Appendix \ref{appendix:freqmarching} an alternative ``frequency
marching'' method, as discussed also
in \cite[Section IV]{bendory2017bispectrum}, which provides a
computationally efficient alternative to mimic the oracle at the expense of a
larger requirement for the sample size $N$.

This method sets $\tilde{\phi}_1=0$ and, for each $k=2,\ldots,K$, sets
\[\tilde{\phi}_k=\Arg \hat{B}_{1,k-1}+\tilde{\phi}_{k-1} \bmod 2\pi\]
to define a pilot estimator $\tilde{\phi}$ for $\phi$. We show that, resolving
the phase ambiguity of $\hat{\Phi}$ using this pilot estimate and then
re-estimating $\hat{\phi}$ by least squares, the resulting procedure
achieves the same risk as described in Theorem
\ref{thm:oracleMoM} under a requirement
for $N$ that is larger by a factor of $K^2$.
\end{Remark}

\section{Maximum likelihood estimator}\label{sec-5}

The method-of-moments procedure analyzed in the preceding section is not
rate-optimal over the full low-noise regime described by
Theorem \ref{thm:lownoise}. Motivated by this observation, and by the more
common use of likelihood-based approaches in practice
\citep{sigworth1998maximum,scheres2012relion}, in this section we analyze
the maximum likelihood estimator (MLE) in the setting of Theorem
\ref{thm:lownoise}.

Define the log-likelihood function
\begin{equation}\label{eq:ll}
l(\theta,y)=\log p_\theta(y):=\log\left[\frac{1}{2\pi} \int_{-\pi}^\pi
\left(\frac{1}{\sqrt{2\pi\sigma^2}}\right)^{2K}\exp\left(-\frac{\|y-g(\alpha)\cdot\theta\|^2}{2\sigma^2}\right)d\alpha\right]
\end{equation}
where $p_\theta(y)$ denotes the Gaussian mixture density 
that marginalizes over the unknown rotation. 
Then the MLE is given by
\[\hat{\theta}^{\mathrm{MLE}}=\arg\min_{\theta \in \R^{2K}} R_N(\theta),
\qquad R_N(\theta)=-\frac{1}{N}\sum_{m=1}^{N}l(\theta, y^{(m)}),\]
where $R_N(\theta)$ denotes the negative empirical log-likelihood.

\revise{For the results of this section, we isolate the following general
condition for the Fourier magnitudes of $\theta^*$.
\begin{Assumption}\label{assump:gen}
There exists a constant $c_\gen>0$ such that for any $B \subseteq
\{1,\ldots,K\}$ with $|B| \geq K/2$
\[\sum_{k \in B} r_k(\theta^*)^2 \geq c_\gen \|\theta^*\|^2\]
\end{Assumption}
\noindent It is clear that this condition holds for our signal class
$\Theta_\beta$ of interest.
Our main result is then the following general risk bound
for $\hat{\theta}^{\mathrm{MLE}}$ in the
low-noise setting of Theorem \ref{thm:lownoise}.}

\begin{Theorem}\label{thm-mle}
\revise{Suppose Assumption \ref{assump:gen} holds. Then
there exist constants $C,C_0,C_1>0$ depending only on $c_\gen$ such that if
$\sigma^2 \leq \frac{K}{C_1\log K}$ and $N \geq
C_0K(1+\frac{K\sigma^2}{\|\theta^*\|^2})
\log (K+\frac{\|\theta^*\|^2}{\sigma^2})$,} then
\[\mathbb{E}_{\theta^*}[L(\hat{\theta}^{\mathrm{MLE}},\theta^*)]
\leq \frac{CK\sigma^2}{N}.\]
\end{Theorem}
\noindent For \revise{$\sigma^2 \geq K^{-2\beta}$}, this requirement for $N$ reduces to that of Theorem \ref{thm:lownoise}, up to a modified constant $C_0>0$.
Combined with the argument for \revise{$\sigma^2 \leq K^{-2\beta}$}
in Section \ref{sec:MoMopt}, this immediately implies the minimax upper bound
of Theorem \ref{thm:lownoise}.

In the remainder of this section, we prove Theorem \ref{thm-mle}.
The proof applies a classical idea of second-order
Taylor expansion for the log-likelihood function. Observe first that the
negative log-likelihood
$R_N(\theta)$ satisfies the rotational invariance
$R_N(\theta)=R_N(g(\alpha) \cdot \theta)$ for all $\alpha \in \cA$. Thus
$\hat{\theta}^{\mathrm{MLE}}$ is defined only up to rotation, and all
rotations of $\hat{\theta}^{\mathrm{MLE}}$ incur the same loss.
To fix this rotation and ease notation in the analysis, let us denote by
$\hat{\theta}^{\mathrm{MLE}}$ the rotation of the MLE such that
\begin{equation}\label{eq:MLEversion}
\|\hat{\theta}^\mathrm{MLE}-\theta^*\|^2=
\min_{\alpha \in \cA} \|g(\alpha) \cdot \hat{\theta}^\mathrm{MLE}-\theta^*\|^2
=L(\hat{\theta}^\mathrm{MLE},\theta^*),
\end{equation}
where $\theta^*$ is the true parameter.
Since $\hat{\theta}^{\mathrm{MLE}}$ minimizes $R_N(\theta)$,
we have $0 \geq R_N(\hat{\theta}^{\mathrm{MLE}})-R_N(\theta^*)$.
Then Taylor expansion (for this rotation of $\hat{\theta}^\mathrm{MLE}$
that satisfies (\ref{eq:MLEversion})) gives
\begin{align}
0 &\geq R_N(\hat{\theta}^{\mathrm{MLE}})-R_N(\theta^*)\nonumber\\
&=\nabla R_N(\theta^*)^\top(\hat{\theta}^{\mathrm{MLE}}-\theta^*)
+\frac{1}{2}(\hat{\theta}^{\mathrm{MLE}}-\theta^*)^\top
\nabla^2 R_N(\tilde{\theta})(\hat{\theta}^{\mathrm{MLE}}-\theta^*)
\label{eq:taylor}
\end{align}
where $\tilde{\theta} \in \R^{2K}$ is on the line segment between $\theta^*$ and
$\hat{\theta}^{\mathrm{MLE}}$. Heuristically, Theorem \ref{thm-mle} will
follow from the bounds
\begin{align}
\Big|\nabla R_N(\theta^*)^\top(\hat{\theta}^{\mathrm{MLE}}-\theta^*)\Big|
&\lesssim \sqrt{\frac{K}{N\sigma^2}} \cdot
\|\hat{\theta}^{\mathrm{MLE}}-\theta^*\|,
\label{eq:heuristicgrad}\\
(\hat{\theta}^{\mathrm{MLE}}-\theta^*)^\top
\nabla^2 R_N(\tilde{\theta})(\hat{\theta}^{\mathrm{MLE}}-\theta^*)
&\gtrsim \frac{1}{\sigma^2} \cdot \|\hat{\theta}^{\mathrm{MLE}}-\theta^*\|^2.
\label{eq:heuristichess}
\end{align}
Applying these to (\ref{eq:taylor}) and rearranging yields the desired result
$\|\hat{\theta}^{\mathrm{MLE}}-\theta^*\|^2 \lesssim K\sigma^2/N$.

The bulk of the proof lies in establishing an appropriate version of
(\ref{eq:heuristichess}). This requires a delicate argument for large
$K$, as naive uniform concentration and Lipschitz bounds for
$\nabla^2 R_N(\theta) \in \R^{2K \times 2K}$ fail to establish
(\ref{eq:heuristichess}) in the full ranges of $\sigma^2$ and
$N$ that are specified by Theorem \ref{thm-mle}.
In the remainder of this section, we describe the components of this
argument, deferring detailed proofs to Appendix \ref{appendix:MLE}.

\subsection{Gradient and Hessian of the log-likelihood}

To simplify the model, observe that each sample $y^{(m)}$
satisfies the equality in law
\[y^{(m)}=g(\alpha^{(m)}) \cdot \theta^*+\sigma \eps^{(m)}
\overset{L}{=}g(\alpha^{(m)}) \cdot (\theta^*+\sigma \eps^{(m)}).\]
Furthermore, $g(\alpha^{(m)})^{-1} g(\alpha)=g(\alpha-\alpha^{(m)})$
where, if $\alpha \sim \Unif([-\pi,\pi))$ is a uniformly random rotation, then
$\alpha-\alpha^{(m)}$ is also uniformly random for any fixed $\alpha^{(m)}$.
Applying these observations to the form (\ref{eq:ll})
of the log-likelihood function, we obtain the equality in law for
the negative log-likelihood process
\begin{equation}\label{eq:RNequalinlaw}
\Big\{R_N(\theta):\theta \in \R^{2K}\Big\}
\overset{L}{=} \left\{-\frac{1}{N}\sum_{m=1}^N l(\theta,\,\theta^*+\sigma
\eps^{(m)}):\theta \in \R^{2K}\right\}.
\end{equation}
That is to say, having defined the log-likelihood function to marginalize over
a uniformly
random latent rotation, the distribution of $\{R_N(\theta):\theta \in \R^{2K}\}$
is the same under the model
$y^{(m)}=g(\alpha^{(m)}) \cdot \theta^*+\sigma \eps^{(m)}
\sim p_{\theta^*}$ as under a model
$y^{(m)}=\theta^*+\sigma \eps^{(m)}$ without latent rotations.
Thus, in the analysis, we will henceforth assume the simpler model
\begin{equation}\label{eq:simplifyy}
y^{(m)}=\theta^*+\sigma \eps^{(m)} \text{ for } m=1,\ldots,N,
\qquad \eps^{(1)},\ldots,\eps^{(N)} \overset{\iid}{\sim} \mathcal{N}(0,I_{2K}).
\end{equation}

Under this model (\ref{eq:simplifyy}), expanding the square in the exponent of
(\ref{eq:ll}), $R_N(\theta)$ may be written as
\begin{align}
R_N(\theta)&=\frac{1}{N}\sum_{m=1}^N K\log 2\pi \sigma^2
+\frac{\|\theta\|^2}{2\sigma^2}+\frac{\|\theta^*+\sigma \eps^{(m)}\|^2}
{2\sigma^2}\nonumber\\
&\hspace{1in}
-\log\left[\frac{1}{2\pi}\int_{-\pi}^\pi
\exp\left(\frac{\langle \theta^*+\sigma \eps^{(m)},
g(\alpha)\cdot\theta\rangle}{\sigma^2}\right)d\alpha\right].\label{eq:llexpanded}
\end{align}
Given $\theta,\eps \in \R^{2K}$, define $\cP_{\theta,\eps}$ to be the tilted
probability law over angles $\alpha \in \cA$ with density
\begin{equation}\label{eq:Pthetaeps}
\frac{d\cP_{\theta,\eps}(\alpha)}{d\alpha}
=\exp\left(\frac{\langle \theta^*+\sigma \eps,
\,g(\alpha) \cdot \theta \rangle}{\sigma^2}\right)\Bigg/
\int_{-\pi}^\pi \exp\left(\frac{\langle \theta^*+\sigma \eps,
\,g(\alpha) \cdot \theta \rangle}{\sigma^2}\right)\,d\alpha.
\end{equation}
Then direct computation shows that the gradient and Hessian of $R_N(\theta)$
take the forms
\begin{align}
\nabla R_N(\theta)&=\frac{\theta}{\sigma^2}-\frac{1}{N}\sum_{m=1}^N
\frac{1}{\sigma^2}
\E_{\alpha \sim \cP_{\theta,\eps^{(m)}}}\Big[g(\alpha)^{-1}(\theta^*+\sigma
\eps^{(m)})\Big]\label{eq:grad}\\
\nabla^2 R_N(\theta)&=\frac{1}{\sigma^2}I-\frac{1}{N}\sum_{m=1}^N
\frac{1}{\sigma^4}\Cov_{\alpha \sim \cP_{\theta,\eps^{(m)}}}
\Big[g(\alpha)^{-1}(\theta^*+\sigma \eps^{(m)})\Big]\label{eq:hess}
\end{align}
where the expectation and covariance are over the random rotation $\alpha \sim
\cP_{\theta,\eps^{(m)}}$ (conditional on $\eps^{(m)}$) following the above law.

\subsection{Tail bound}
\label{section-5.2}

As a first step of the proof, we fix a small constant $\delta_1 \in (0,1)$ to be
determined, and define the domain
\begin{equation}\label{eq:Bdelta}
\cB(\delta_1)=\left\{\theta:\|\theta-\theta^*\| \leq \delta_1 
\|\theta^*\|\right\} \subset \R^{2K}.
\end{equation}
We first establish the following lemma, which shows that $\hat{\theta}^\MLE$
belongs to this domain $\cB(\delta_1)$ with high probability, and provides also
an upper bound for the fourth moment of $\hat{\theta}^{\text{MLE}}$.

\begin{Lemma}
\label{lemma-highprob-bound}
\revise{Suppose that Assumption \ref{assump:gen} holds.}
Fix any constant $\delta_1>0$,
and define $\cB(\delta_1)$ by (\ref{eq:Bdelta}).
Then there exist constants $C_0,C_1,C',c'>0$ depending only on
$c_\gen,\delta_1$ such that if \revise{$\sigma^2 \leq
\frac{\|\theta^*\|^2}{C_1\log K}$} and $N \geq C_0K$, then 
\begin{align}
\P\left[\hat{\theta}^{\mathrm{MLE}} \in \cB(\delta_1)\right]
&\geq 1-e^{-c'N(\log K)^2/K},\label{eq:MLElocalization}\\
\E[\|\hat{\theta}^{\mathrm{MLE}}\|^4] &\leq C'\|\theta^*\|^4.
\label{eq:MLE4thmoment}
\end{align}
\end{Lemma}

To show this lemma, define the population negative log-likelihood
$R(\theta)=\E_{\theta^*}[R_N(\theta)]$,
where the equality in law (\ref{eq:RNequalinlaw})
allows us to evaluate the expectation under the simplified model
(\ref{eq:simplifyy}). Then the KL-divergence
between $p_{\theta^*}$ and $p_\theta$ is given by
\begin{equation}\label{eq:DKL}
D_{\KL}(p_{\theta^*}\|p_\theta)=R(\theta)-R(\theta^*)
=\E_{\theta^*}[R_N(\theta)]-\E_{\theta^*}[R_N(\theta^*)].
\end{equation}
Recalling the form (\ref{eq:llexpanded}) for the negative log-likelihood
$R_N(\theta)$, we have
\begin{equation}\label{eqn-KL-lower-1} 
D_{\mathrm{KL}}(p_{\theta^*}\|p_{\theta})=
\frac{\|\theta\|^2-\|\theta^*\|^2}{2\sigma^2}+\mathrm{I}-\mathrm{II}
\end{equation}
where
\begin{align*}
\mathrm{I}&=\mathbb{E}\log\frac{1}{2\pi}\int_{-\pi}^\pi
\exp\left(\frac{\langle \theta^*+\sigma \eps,
g(\alpha)\cdot\theta^*\rangle}{\sigma^2}\right)d\alpha\\
\mathrm{II}&=\mathbb{E}\log\frac{1}{2\pi}\int_{-\pi}^\pi
\exp\left(\frac{\langle \theta^*+\sigma \eps,
g(\alpha)\cdot\theta\rangle}{\sigma^2}\right)d\alpha
\end{align*}
and both expectations are over $\eps \sim \mathcal{N}(0,I_{2K})$.

\revise{For sufficiently small $|\alpha|$,}
we may apply a quadratic Taylor expansion of
$\langle \theta^*,\,g(\alpha) \cdot \theta^* \rangle=\sum_k r_k(\theta^*)^2 \cos
k\alpha$ around $\alpha=0$, to write
\begin{equation}\label{eq:alpha0Taylor}
\langle \theta^*,\,g(\alpha) \cdot \theta^* \rangle
-\|\theta^*\|^2 \approx -\sum_{k=1}^K r_k(\theta^*)^2 \cdot
\frac{k^2\alpha^2}{2} \asymp \revise{-K^2\|\theta^*\|^2\alpha^2}
\end{equation}
\revise{where this last approximation holds under Assumption
\ref{assump:gen}.}
Then $\int \exp(\theta^*,g(\alpha) \cdot \theta^*/\sigma^2)\,d\alpha$ 
in $\mathrm{I}$ may be approximated by a Gaussian integral over $\alpha \in \R$.
Upper bounding $\mathrm{II}$
by the supremum over $\alpha$, and applying a standard covering net argument to
control the suprema of the Gaussian processes $\langle \eps,\,g(\alpha) \cdot
\theta^* \rangle$ and $\langle \eps,\,g(\alpha) \cdot \theta \rangle$,
we obtain the following lower bound on the KL-divergence.

\begin{Lemma}\label{lemma-KL-lower}
\revise{Suppose Assumption \ref{assump:gen} holds, and $\sigma^2 \leq
\|\theta^*\|^2$.} Then there are constants $C_2,C_3>0$ depending only
on $c_\gen$ such that for any $\theta\in\mathbb{R}^{2K}$,
\[D_{\mathrm{KL}}(p_{\theta^*}\|p_{\theta}) \geq \frac{\min_{\alpha \in \cA}
\|\theta^*-g(\alpha)\cdot\theta\|^2}{2\sigma^2}
-\frac{1}{2}\log\revise{\left(\frac{C_2K^2\|\theta^*\|^2}{\sigma^2}\right)}
-\frac{C_3(\|\theta^*\|+\|\theta\|)}{\sigma}\cdot \sqrt{\log K}.\]
\end{Lemma}

Comparing this with the rate of uniform concentration of the
negative log-likelihood $R_N(\theta)$ around its mean $R(\theta)$
(cf.\ Lemma \ref{lemma-bounded-norm-start}), we obtain an exponential
tail bound for the probability of the event
\[\|\theta^*-\hat{\theta}^\MLE\| \in
\big[n\delta_1\|\theta^*\|,(n+1)\delta_1\|\theta^*\|\big]\]
for each integer $n \geq 1$. Summing this bound over all $n \geq 1$ yields Lemma
\ref{lemma-highprob-bound}.

\subsection{Lower bound for the information matrix}\label{sec:localanalysis}

In light of Lemma \ref{lemma-highprob-bound}, to show (\ref{eq:heuristichess})
with high probability, it suffices to establish a version of the lower bound
\begin{equation}\label{eq:heuristichesslower}
\nabla^2 R_N(\theta) \gtrsim \frac{1}{\sigma^2} \cdot I
\qquad \text{ uniformly over } \theta \in \cB(\delta_1).
\end{equation}
Denote the tangent vector to the rotational orbit
$\{g(\alpha) \cdot \theta^*: \alpha \in \cA\}$ at $\theta^*$ by
\begin{equation}\label{eq:ustar}
u^*=\frac{d}{d\alpha} g(\alpha) \cdot \theta^*\bigg|_{\alpha=0}
=g'(0) \cdot \theta^*.
\end{equation}
From the rotational invariance of $R(\theta)$, it
is easy to see that the expected (Fisher) information matrix
$\E[\nabla^2 R_N(\theta^*)]=\nabla^2 R(\theta^*)$
must be singular, with $u^*$ belonging to its kernel. Thus we
cannot expect the bound (\ref{eq:heuristichesslower}) to hold in all directions
of $\R^{2K}$, but only in those directions orthogonal to $u^*$. This will
suffice to show (\ref{eq:heuristichess}), because we will check that
choosing $\hat{\theta}^\MLE$ to satisfy (\ref{eq:MLEversion})
also ensures $\hat{\theta}^\MLE-\theta^*$ is orthogonal to $u^*$.
The statement (\ref{eq:heuristichesslower}) restricted to directions orthogonal
to $u^*$ is formalized in the following lemma.

\begin{Lemma}\label{lemma-delta-A}
\revise{Suppose Assumption \ref{assump:gen} holds.}
Fix any constant $\eta>0$.
There exist constants $C_0,C_1,\delta_1,c>0$ depending only on
$c_\gen,\eta$ such that if \revise{$\sigma^2 \leq \frac{\|\theta^*\|^2}{C_1\log
K}$ and $N \geq C_0K(1+\frac{K\sigma^2}{\|\theta^*\|^2})\log
(K+\frac{\|\theta^*\|^2}{\sigma^2})$},
then with probability at
least $1-e^{-\frac{cN}{(1+K\sigma^2/\|\theta^*\|^2)^2}}$, the following holds:
For every $\theta \in \cB(\delta_1)$ and every
unit vector $v \in \R^{2K}$ satisfying $\langle u^*,v \rangle=0$,
\[v^\top \nabla^2 R_N(\theta) v \geq \frac{1-\eta}{\sigma^2}.\]
\end{Lemma}

From the form of $\nabla^2 R_N(\theta)$ in (\ref{eq:hess}), observe that
\begin{equation}\label{eq:hessexpansion}
v^\top \nabla^2 R_N(\theta)v=\frac{1}{\sigma^2}
-\frac{1}{N\sigma^4}\sum_{m=1}^N \Var_{\alpha \sim \cP_{\theta,\eps^{(m)}}}\Big[
v^\top g(\alpha)^{-1} (\theta^*+\sigma\eps^{(m)}) \Big].
\end{equation}
The proof of Lemma \ref{lemma-delta-A} is based on a refinement of the
argument in the preceding section, to approximate the distribution
$\cP_{\theta,\eps}$ in the above variance by a Gaussian law over
$\alpha$. Here, applying a separate bound to control the Gaussian process
$\sup_\alpha \langle \eps,g(\alpha) \cdot \theta \rangle$ will be too loose
to obtain the lemma. We instead perform a Taylor
expansion of $\langle \theta^*+\sigma \eps,\,g(\alpha) \cdot \theta \rangle$
around its (random, $\eps$-dependent) mode
\[\alpha_0=\argmax_\alpha 
\langle \theta^*+\sigma \eps,\,g(\alpha) \cdot \theta \rangle,\]
and combine this with the condition $\theta \in \cB(\delta_1)$
to obtain a quadratic approximation
\[\frac{\langle \theta^*+\sigma \eps,\,g(\alpha) \cdot \theta \rangle}{\sigma^2}
-\text{constant} \asymp
\revise{-\frac{K^2\|\theta^*\|^2}{\sigma^2}(\alpha-\alpha_0)^2}\]
where the constant is independent of $\alpha$. Thus, $\cP_{\theta,\eps}$ for any
$\theta \in \cB(\delta_1)$ may be
approximated by a Gaussian law with mean $\alpha_0$ and variance on the order of
$\frac{\sigma^2}{K^2\|\theta^*\|^2}$. Applying a Taylor expansion also of
$v^\top g(\alpha)^{-1}(\theta^*+\sigma \eps)$ around $\alpha=\alpha_0$,
and approximating the variance over $\alpha \sim \cP_{\theta,\eps}$ by the
variance with respect to this Gaussian law, we obtain a bound
\[\Var_{\alpha \sim \cP_{\theta,\eps}}\Big[v^\top g(\alpha)^{-1}
(\theta^*+\sigma \eps)\Big] \leq \eta \sigma^2\]
for a small constant $\eta>0$, which is sufficient to show Lemma
\ref{lemma-delta-A}.

These Taylor expansion arguments may be formalized on a high-probability
event for $\eps$, where this event is dependent on $\theta$ and $v$.
More precisely, let
\[\tilde{\theta}=(\theta_1,\ldots,\theta_K) \in \C^K,
\qquad \tilde{v}=(v_1,\ldots,v_K) \in \C^K,
\qquad \tilde{\eps}=(\eps_1,\ldots,\eps_K) \in \C^K\]
denote the complex representations of $\theta,v,\eps$ as defined in
Section \ref{sec:complexrepr}.
For each $\theta \in \cB(\delta_1)$ and unit test vector
$v \in \R^{2K}$ with $\langle u^*,v \rangle=0$, we define a
$(\theta,v)$-dependent domain $\cE(\theta,v,\delta_1) \subset \R^{2K}$ by
the four conditions
\revise{
\begin{align*}
\sup_{\alpha \in \cA} |\langle \eps,g(\alpha) \cdot \theta \rangle|
&\leq \frac{\delta_1\|\theta^*\|^2}{\sigma}\\
\sup_{\alpha \in \cA} \Big|\langle \eps,g(\alpha) \cdot v \rangle\Big|
& \leq \frac{\|\theta^*\|}{\sigma}\\
\sup_{\alpha,\alpha' \in [-\pi,\pi)}
\frac{1}{\alpha^2}
\left|\Re \sum_{k=1}^K \overline{\eps_k}e^{ik\alpha'}
\Big(e^{ik\alpha}-1-ik\alpha\Big)\theta_k\right|
&\leq \frac{\delta_1 K^2 \|\theta^*\|^2}{\sigma}\\
\sup_{\alpha,\alpha' \in [-\pi,\pi)}
\frac{1}{|\alpha-\alpha'|}
\left|\Re \sum_{k=1}^K \overline{\eps_k}\Big(e^{ik\alpha}-e^{ik\alpha'}\Big)
v_k\right| &\leq \frac{\delta_1K \|\theta^*\|}{\sigma}
\end{align*}}
The following deterministic lemma holds on the event
that $\eps \in \cE(\theta,v,\delta_1)$.

\begin{Lemma}\label{lemma-infor-3}
\revise{Suppose Assumption \ref{assump:gen} holds.}
Fix any $\eta>0$. There exist constants $C_1,\delta_1>0$ depending only on
$c_\gen,\eta$ such that if \revise{$\sigma^2 \leq \frac{\|\theta^*\|^2}{C_1\log
K}$,}
then the following holds: For any $\theta \in \cB(\delta_1)$, any unit vector
$v \in \R^{2K}$ satisfying $\langle u^*,v \rangle=0$, and any (deterministic)
$\eps \in \cE(\theta,v,\delta_1)$,
\begin{align}
\Var_{\alpha \sim \cP_{\theta,\eps}}\Big[v^\top
g(\alpha)^{-1}(\theta^*+\sigma \eps)\Big] &\leq \eta\sigma^2.
\label{eq:varbound}
\end{align}
\end{Lemma}

Each of the four conditions defining $\cE(\theta,v,\delta_1)$ involves the
supremum of a Gaussian process, which may be bounded using a standard covering
net argument. We remark that each of these conditions is defined
with the right side being a factor $\|\theta^*\|/\sigma$ larger than the
mean value of the left side, so that their failure
probabilities are exponentially small in $\|\theta^*\|^2/\sigma^2$. This is summarized in
the following result.

\begin{Lemma}\label{lemma:epsgoodprob}
\revise{Suppose Assumption \ref{assump:gen} holds.}
Fix any constant $\delta_1>0$, any $\theta \in \cB(\delta_1)$, and any unit
vector $v$ satisfying $\langle u^*,v \rangle=0$.
For some constants $C_1,c>0$ depending only on $c_\gen,\delta_1$,
if \revise{$\sigma^2 \leq \frac{\|\theta^*\|^2}{C_1\log K}$, then
\[\P_{\eps \sim \mathcal{N}(0,I)}\Big[\eps \notin \cE(\theta,v,\delta_1)
\Big] \leq e^{-c\|\theta^*\|^2/\sigma^2}.\]}
\end{Lemma}

Finally, we combine Lemmas \ref{lemma-infor-3} and \ref{lemma:epsgoodprob}
to conclude the proof of Lemma \ref{lemma-delta-A}: We may write the second
term of (\ref{eq:hessexpansion}) as
\begin{align*}
&\frac{1}{N\sigma^4} \sum_{m=1}^N \Var_{\alpha \sim \cP_{\theta,\eps^{(m)}}}\Big[
v^\top g(\alpha)^{-1} (\theta^*+\sigma\eps^{(m)}) \Big]
\cdot \1\{\eps^{(m)} \in \cE(\theta,v,\delta_1)\}\\
&\hspace{1in}+\frac{1}{N\sigma^4}
\sum_{m=1}^N \Var_{\alpha \sim \cP_{\theta,\eps^{(m)}}}\Big[
v^\top g(\alpha)^{-1} (\theta^*+\sigma\eps^{(m)}) \Big]
\cdot \1\{\eps^{(m)} \notin \cE(\theta,v,\delta_1)\}.
\end{align*}
The first sum is bounded by Lemma \ref{lemma-infor-3}, while the second sum
is sparse by Lemma \ref{lemma:epsgoodprob} and may be controlled using
a Chernoff bound for binomial random variables. Taking a union bound over a
covering net of pairs $(\theta,v)$ shows Lemma \ref{lemma-delta-A}.

\subsection{Proof of Theorem \ref{thm-mle}}

We now combine the preceding lemmas to conclude the proof of
Theorem \ref{thm-mle}. Let $C_0,C_1,\delta_1>0$ be such that the conclusions
of Lemma \ref{lemma-delta-A} hold for $\eta=1/2$. Define the event
\[\cE=\left\{\hat{\theta}^\MLE \in \cB(\delta_1) \text{ and }
\sup_{\theta \in \cB(\delta_1)}\,\sup_{v:\|v\|=1,\langle u^*,v \rangle=0}\,
v^\top \nabla^2 R_N(\theta) v \geq \frac{1}{2\sigma^2}\right\}.\]

When $\cE$ holds, we have also $\tilde{\theta} \in \cB(\delta_1)$ in the
Taylor expansion (\ref{eq:taylor}). Recall our choice of rotation
(\ref{eq:MLEversion}) for $\hat{\theta}^{\text{MLE}}$. Then the
first-order condition for (\ref{eq:MLEversion}) gives
\[0=\frac{d}{d\alpha} \|\hat{\theta}^\mathrm{MLE}-
g(\alpha) \cdot \theta^*\|^2\bigg|_{\alpha=0}
=-2\langle u^*,\,\hat{\theta}^\mathrm{MLE}-\theta^* \rangle,\]
so that $\langle u^*,\,\hat{\theta}^{\mathrm{MLE}}-\theta^* \rangle=0$.
Then (\ref{eq:taylor}) and the definition of $\cE$ imply
\[0 \geq \1\{\cE\}\left(\nabla R_N(\theta^*)^\top (\hat{\theta}^\MLE-\theta^*)
+\frac{1}{4\sigma^2}\|\hat{\theta}^\MLE-\theta^*\|^2\right).\]
Rearranging, we get
\[\1\{\cE\}\|\hat{\theta}^\MLE-\theta^*\|^2
\leq -\1\{\cE\} \cdot 4\sigma^2 \cdot
\nabla R_N(\theta^*)^\top (\hat{\theta}^\MLE-\theta^*)
\leq 4\sigma^2 \cdot
\|\nabla R_N(\theta^*)\| \cdot \|\hat{\theta}^\MLE-\theta^*\|.\]
Dividing by $\|\hat{\theta}^\MLE-\theta^*\|$, squaring both sides,
and taking expectation yields
\begin{equation}\label{eq:MLEriskonE}
\E\Big[\1\{\cE\}\|\hat{\theta}^\MLE-\theta^*\|^2\Big]
\leq 16\sigma^4\E\Big[\|\nabla R_N(\theta^*)\|^2\Big].
\end{equation}

From (\ref{eq:grad}), we have
\[\nabla R_N(\theta^*)=\frac{1}{N}\sum_{m=1}^N \left(\frac{\theta^*}{\sigma^2}
-\frac{1}{\sigma^2}\E_{\alpha \sim \cP_{\theta^*,\eps^{(m)}}}\Big[
g(\alpha)^{-1}(\theta^*+\sigma \eps^{(m)})\Big]\right).\]
These summands (the per-sample score vectors)
are independent random vectors with mean 0, by the first-order
condition for $\theta^*$ minimizing $R(\theta)$. So
\begin{align*}
\E\Big[\|\nabla R_N(\theta^*)\|^2\Big]
&=\frac{1}{N}\E_{\eps \sim \mathcal{N}(0,I)}
\left[\left\|\frac{\theta^*}{\sigma^2}
-\frac{1}{\sigma^2}\E_{\alpha \sim \cP_{\theta,\eps}}\Big[
g(\alpha)^{-1}(\theta^*+\sigma \eps)\Big]\right\|^2\right]\\
&=\frac{1}{N\sigma^4} \E_{\eps \sim \mathcal{N}(0,I)}
\left[\left\|\E_{\alpha \sim \cP_{\theta,\eps}}\Big[
g(\alpha)^{-1}(\theta^*+\sigma \eps)\Big]\right\|^2-\|\theta^*\|^2\right]\\
&\leq \frac{1}{N\sigma^4}\E_{\eps \sim \mathcal{N}(0,I)}
\left[\|\theta^*+\sigma\eps\|^2-\|\theta^*\|^2\right]=\frac{2K}{N\sigma^2}.
\end{align*}
Combining with (\ref{eq:MLEriskonE}),
\[\E\Big[\1\{\cE\}\|\hat{\theta}^\MLE-\theta^*\|^2\Big]
\leq \frac{32K\sigma^2}{N}.\]

By Lemmas \ref{lemma-highprob-bound} and \ref{lemma-delta-A},
\revise{$\P[\cE^c] \leq e^{-\frac{cN}{(1+K\sigma^2/\|\theta^*\|^2)^2}}$}
for some constant $c>0$.
Then applying also (\ref{eq:MLE4thmoment}), for some constant $C>0$,
\revise{
\[\E\Big[\1\{\cE^c\}\|\hat{\theta}^\MLE-\theta^*\|^2\Big]
\leq \sqrt{\E[\|\hat{\theta}^\MLE-\theta^*\|^4]}
\cdot \sqrt{\P[\cE^c]} \leq C\|\theta^*\|^2 \cdot
e^{-\frac{cN}{2(1+K\sigma^2/\|\theta^*\|^2)^2}}.\]
Under the given assumption
$N \geq C_0K(1+\frac{K\sigma^2}{\|\theta^*\|^2})\log(K+\frac{\|\theta^*\|^2}{\sigma^2})$ for
sufficiently large $C_0>0$, this implies also $N \geq
C_0'K(1+\frac{K\sigma^2}{\|\theta^*\|^2})\log N$ for a large constant $C_0'>0$. (This is verified in the proof of Lemma \ref{lemma-delta-A}, cf.\
(\ref{eq:logKlogN}) of Appendix \ref{appendix:MLE}.) Then
\[\E\Big[\1\{\cE^c\}\|\hat{\theta}^\MLE-\theta^*\|^2\Big]
\leq C\|\theta^*\|^2 \cdot e^{-\frac{cN}{2(1+K\sigma^2/\|\theta^*\|^2)^2}}
\leq \frac{C'\sigma^2}{N}.\]}
Combining the above two risk bounds on $\cE$ and $\cE^c$
yields Theorem \ref{thm-mle}.

\section{Minimax lower bounds}
\label{sec-6}

In this section, we show the minimax lower bounds of Theorems
\ref{thm:highnoise} and \ref{thm:lownoise}. The lower bounds will be implied by
estimation of the Fourier phases $\phi_k(\theta^*)$ only, even when the
Fourier magnitudes $r_k(\theta^*)$ are known. \revise{Fix any
$\beta \in [0,\frac{1}{2})$, and consider the parameter space
\[\cP_\beta=\Big\{\theta^* \in \R^{2K}: r_k(\theta^*)=k^{-\beta}
\text{ for all } k=1,\ldots,K \Big\}.\]}
The main result of this section is the following minimax lower bound
over $\cP_\beta$, which is valid for any noise level $\sigma^2>0$ and
interpolates between the low-noise and high-noise regimes.

\revise{
\begin{Lemma}\label{lemma:minimaxlowerP}
Fix any $\beta \in [0,\frac{1}{2})$. Then for some $\beta$-dependent
constants $C,c>0$ and any $\sigma^2>0$,
\begin{equation}\label{eq:minimaxlowerlownoise}
\inf_{\hat{\theta}} \sup_{\theta^* \in \cP_\beta}
\E_{\theta^*}[L(\theta^*,\hat{\theta})] \geq c \cdot
\min\left(\frac{1}{N} \cdot \max\left(K\sigma^2,\;
\frac{K^{4\beta}\sigma^6}
{e^{CK^{1-2\beta}/\sigma^2}}\right),\,K^{1-2\beta}\right).
\end{equation}
\end{Lemma}}

Let us check that this implies the minimax lower bounds of 
Theorems \ref{thm:highnoise} and \ref{thm:lownoise}.

\revise{
\begin{proof}[Proof of Theorems \ref{thm:highnoise} and \ref{thm:lownoise},
lower bounds]

By rescaling, we may assume without loss of generality that $\clower \leq 1 \leq
\cupper$, and hence $\cP_\beta \subset \Theta_\beta$.
Assuming $\sigma^2 \geq c_0K^{1-2\beta}$, choosing the second argument of $\max(\cdot)$
in (\ref{eq:minimaxlowerlownoise}) gives
\[\inf_{\hat{\theta}} \sup_{\theta^* \in \Theta_\beta}
\E_{\theta^*}[L(\theta^*,\hat{\theta})] \geq
\inf_{\hat{\theta}} \sup_{\theta^* \in \cP_\beta}
\E_{\theta^*}[L(\theta^*,\hat{\theta})] \geq c \cdot
\min\left(\frac{K^{4\beta}\sigma^6}{N},\,K^{1-2\beta}\right)\]
for a constant $c>0$ depending on $c_0$.
When $N \geq C_0K^{6\beta}\sigma^6\log K$ for sufficiently large $C_0>0$, we
have $K^{4\beta}\sigma^6/N<K^{1-2\beta}$, so this gives the lower bound of 
Theorem \ref{thm:highnoise}.
For any $\sigma^2>0$, choosing the first argument of $\max(\cdot)$
in (\ref{eq:minimaxlowerlownoise}) also gives
\[\inf_{\hat{\theta}} \sup_{\theta^* \in \Theta_\beta}
\E_{\theta^*}[L(\theta^*,\hat{\theta})] \geq
\inf_{\hat{\theta}} \sup_{\theta^* \in \cP_\beta}
\E_{\theta^*}[L(\theta^*,\hat{\theta})] \geq c \cdot
\min\left(\frac{K\sigma^2}{N},\,K^{1-2\beta}\right).\]
When $N \geq C_0K^{1+2\beta}\sigma^2 \log K$ for sufficiently
large $C_0>0$, we have $K\sigma^2/N<K^{1-2\beta}$, so this gives the lower bound of
Theorem \ref{thm:lownoise}.
\end{proof}}

Finally, we describe the arguments that show Lemma \ref{lemma:minimaxlowerP},
deferring detailed proofs to Appendix \ref{appendix:lower}.
Denote $p_\theta(y)$ as the Gaussian mixture density of $y$, as in
(\ref{eq:ll}).
The proof will apply Assouad's hypercube construction together with an upper
bound on the KL-divergence $D_{\KL}(p_\theta\|p_{\theta'})$. For the low-noise
regime of Theorem \ref{thm:lownoise}, a tight upper bound is provided by
(\ref{eq:KLupperlownoise}) below, which is immediate from
the data processing inequality. For the high-noise regime of Theorem
\ref{thm:highnoise}, we apply an argument from \cite{bandeira2020optimal}
for bounding the $\chi^2$-divergence, and track carefully the dependence of
this argument on the dimension $K$.

\begin{Lemma}\label{lemma:KLupperbound}
For any $\theta,\theta' \in \R^{2K}$,
\begin{equation}\label{eq:KLupperlownoise}
D_{\KL}(p_\theta\|p_{\theta'}) \leq \frac{\|\theta-\theta'\|^2}{2\sigma^2}.
\end{equation}
Furthermore, let $\theta=(r_k\cos\phi_k,r_k\sin\phi_k)_{k=1}^K$
and $\theta'=(r_k'\cos\phi_k',r_k'\sin\phi_k')_{k=1}^K$. Denote
$R^2=\max(\sum_{k=1}^K r_k^2$, $\sum_{k=1}^K {r_k'}^2)$ and
$\rupper=\max(\max_{k=1}^K r_k,\max_{k=1}^K r_k')$. Then also
\begin{align}
D_{\KL}(p_{\theta}\|p_{\theta'}) &\leq
\frac{e^{R^2/2\sigma^2}}{4\sigma^4}\sum_{k=1}^K (r_k^2-{r_k'}^2)^2\nonumber\\
&\hspace{0.2in}+\frac{3\rupper^2R^2e^{3R^2/2\sigma^2}}{2\sigma^6} \cdot \inf_{\alpha \in \R}
\sum_{k=1}^K \Big[(r_k-r_k')^2 +r_kr_k'(\phi_k-\phi_k'+k\alpha)^2\Big].
\label{eq:KLupperhighnoise}
\end{align}
\end{Lemma}

\revise{The upper bound (\ref{eq:KLupperhighnoise}) is sufficient to prove Lemma
\ref{lemma:minimaxlowerP} in the setting $\beta=0$, where the argument is as
follows:} We restrict attention to a discrete
space of $2^K$ parameters $\theta^\tau \in \cP_0$, indexed by the hypercube
$\tau \in \{0,1\}^K$, where all Fourier magnitudes are equal to 1 and
the Fourier phases $\phi^\tau=(\phi_1^\tau,\ldots,\phi_K^\tau)$ are given by
\[\phi^\tau_k=\tau_k \cdot \phi.\]
Here, the value $\phi \in \R$ is chosen maximally while ensuring that
$D_{\KL}(p_{\theta^\tau}\|p_{\theta^{\tau'}}) \leq H(\tau,\tau')/N$
by the bounds of Lemma \ref{lemma:KLupperbound}, where
$H(\tau,\tau')$ is the Hamming distance on the hypercube. Applying 
Proposition \ref{prop-lossgeneral}, we may show that the loss between such
parameters is also lower bounded in terms of Hamming distance as
$L(\theta^\tau,\theta^{\tau'}) \gtrsim r^2\phi^2
\cdot H(\tau,\tau')$. Assouad's lemma, see e.g.\ \citep[Lemma
2]{cai2012optimal}, then implies a minimax lower bound over the discrete
parameter space $\{\theta^\tau:\tau \in \{0,1\}^K\}$, which in turn implies the
lower bound of Lemma \ref{lemma:minimaxlowerP} over $\cP_0$.
\revise{For more general decay parameters $\beta \in [0,\frac{1}{2})$, we apply
a variation of this argument where the parameters $\theta^\tau$ are defined such
that only the Fourier phases $\phi_k^\tau$ for $k>K/2$ are non-zero. We establish
a modified version of (\ref{eq:KLupperhighnoise}) for the corresponding
vectors $\theta^\tau$, where $\rupper$ may be replaced by the maximum of
$(r_k,r_k')$ over $k>K/2$. The remainder of the proof is then similar to the
$\beta=0$ setting.}

\appendix

\section{Proofs for method-of-moments estimation}\label{appendix:MoM}

We prove the results of Section \ref{sec-4} on the method-of-moments estimator.

\begin{Proposition}\label{lemma-distribution-1}
Let $\eta \sim \Normal_\C(0,2)$. Then we have the equalities in law
$\eta \overset{L}{=} \overline{\eta}$ and
$\eta \overset{L}{=} e^{i\phi} \eta$ for any $\phi \in \R$.
Furthermore,
\begin{align}
\E[\eta^j\overline{\eta}^k]&=0 \text{ for all integers } j \neq k,
\label{eq:gaussiansymmetry}\\
\E[|\eta|^{2j}] &\leq 4^j j! \text{ for all integers } j \geq 1.\label{eq:gaussianmoments}
\end{align}
\end{Proposition}
\begin{proof}
We may represent $\eta=Re^{i\alpha}$ where $R^2 \sim \chi_2^2$ is independent of
$\alpha \sim \Unif([-\pi,\pi))$. Then $\bar{\eta}=Re^{-i\alpha}$,
$e^{i\phi}\eta=Re^{i(\phi+\alpha)}$, and $\eta^j\bar{\eta}^k=R^{j+k}
e^{i(j-k)\alpha}$, so $\eta \overset{L}{=} \overline{\eta}$,
$\eta \overset{L}{=} e^{i\phi} \eta$, and (\ref{eq:gaussiansymmetry}) follow.
For (\ref{eq:gaussianmoments}), write $|\eta|^2=R^2=Z^2+{Z'}^2$ where $Z,Z'
\overset{\iid}{\sim} \Normal(0,1)$. Then
$\E[|\eta|^{2j}]=\E[(Z^2+{Z'}^2)^j] \leq 2^j\E[Z^{2j}+{Z'}^{2j}]$.
We have $\E[Z^{2j}]=(2j-1)!! \leq 2^{j-1} \cdot j!$, showing
(\ref{eq:gaussianmoments}).
\end{proof}

\subsection{Estimation of $r_k$}

\begin{proof}[Proof of Lemma \ref{lemma:rktail}]
Write $\theta_k=(\theta_{k,1},\theta_{k,2}) \in \R^2$ and
$\eps_k^{(m)}=(\eps_{k,1}^{(m)},\eps_{k,2}^{(m)}) \in \R^2$.
Since $|\tilde y_k^{(m)}|^2=\|\theta_k+\sigma\eps_k^{(m)}\|^2$
and $\|\theta_k\|^2=r_k^2$, we have
\[\frac{1}{N}\sum_{m=1}^N |\tilde y_k^{(m)}|^2-2\sigma^2
=r_k^2+\frac{1}{N}\sum_{m=1}^N 2 \sigma
\langle \eps_k^{(m)},\theta_k \rangle+\frac{\sigma^2}{N}
\sum_{m=1}^N (\|\eps_k^{(m)}\|^2-2).\]
Applying $N^{-1}\sum_m 2\sigma \langle \eps_k^{(m)},\theta_k \rangle
\sim \Normal(0,4\sigma^2r_k^2/N)$,
$\sum_m \|\eps_k^{(m)}\|^2 \sim \chi^2_{2N}$, and standard Gaussian and
chi-squared tail bounds, for a universal constant $c>0$ and any $t>0$ we have
\[\P\left[\frac{1}{N}\sum_{m=1}^N 2 \sigma \langle \eps_k^{(m)},\theta_k
\rangle \geq t\right] \leq e^{-cNt^2/\sigma^2r_k^2},
\quad \P\left[\frac{1}{N}\sum_{m=1}^N (\|\eps_k^{(m)}\|^2-2)
\geq t \right] \leq e^{-cN(t \wedge t^2)}.\]
Then for a universal constant $c'>0$,
\begin{align*}
\P\left[\frac{1}{N}\sum_{m=1}^N |\tilde
y_k^{(m)}|^2-2\sigma^2 \geq (1+t)r_k^2\right] &\leq 
\P\left[\frac{1}{N}\sum_{m=1}^N 2 \sigma \langle \eps_k^{(m)},\theta_k \rangle
\geq \frac{tr_k^2}{2}\right]
+\P\left[\frac{\sigma^2}{N}\sum_{m=1}^N (\|\eps_k^{(m)}\|^2-2)
\geq \frac{tr_k^2}{2}\right]\\
&\leq 2\exp\left(-c'N\left(\frac{t^2r_k^2}{\sigma^2} \wedge
\frac{tr_k^2}{\sigma^2} \wedge \frac{t^2r_k^4}{\sigma^4}\right)\right).
\end{align*}
Applying this with $t=2s+s^2$ and recalling the definition of $\hat{r}_k$ from
(\ref{eq:hatrk}), the left side is exactly $\P[\hat{r}_k \geq r_k(1+s)]$.
Then, considering separately the cases $s \geq 1$ and $s \leq 1$,
the right side reduces to the upper bound (\ref{eq:hatrupper}).
For the lower bound, similarly for any $t>0$,
a lower chi-squared tail bound gives
\[\P\left[\frac{1}{N}\sum_{m=1}^N (\|\eps_k^{(m)}\|^2-2) \leq -t\right]
\leq e^{-cNt^2}.\]
Then we obtain analogously
\begin{align*}
\P\left[\frac{1}{N}\sum_{m=1}^N |\tilde y_k^{(m)}|^2-2\sigma^2 \leq
r_k^2(1-t)\right] & \leq 2\exp\left(-cN\left(\frac{t^2r_k^2}{\sigma^2} \wedge
\frac{t^2r_k^4}{\sigma^4}\right)\right).
\end{align*}
Applying this with $t=2s-s^2 \geq s$ for $s \in [0,1)$, we obtain
(\ref{eq:hatrlower}).
\end{proof}

\begin{proof}[Proof of Corollary \ref{cor:MoMrriskbound}]
We apply $\E[X^2]=\E[\int_0^\infty \1\{|X| \geq s\}\cdot 2s\,ds]=\int_0^\infty
\P[|X| \geq s] \cdot 2s\,ds$ with $X=\hat{r}_k/r_k-1$, and
$\int_0^\infty s\,e^{-\alpha s^2}ds=\alpha^{-1} \int_0^\infty t\,e^{-t^2}dt
\leq C/\alpha$. Then Lemma \ref{lemma:rktail} gives
\[\E[(\hat{r}_k-r_k)^2]=r_k^2 \cdot \E[(\hat{r}_k/r_k-1)^2]
\leq r_k^2 \cdot \int_0^\infty 8s\left(e^{-cNs^2r_k^2/\sigma^2}
+e^{-cNs^2r_k^4/\sigma^4}\right) ds
\leq C\left(\frac{\sigma^2}{N}+
\frac{\sigma^4}{Nr_k^2}\right).\]
\end{proof}

\subsection{Oracle estimation of $\Phi_{k,l}$}

\begin{proof}[Proof of Lemma \ref{lemma:Phikltail}]
Recall $B_{k,l}$ from (\ref{eq:Bkl}) and $\hat{B}_{k,l}$ from (\ref{eq:hatBkl}).
We first show concentration of $\hat{B}_{k,l}$ around $B_{k,l}$. Let us write
\[\tilde y_k^{(m)}=r_ke^{i(\phi_k+k\alpha^{(m)})}+\sigma \tilde{\eps}_k^{(m)}
=e^{ik\alpha^{(m)}}\tilde\theta_k(1+(\sigma/r_k) \eta_k^{(m)})\]
where $\tilde\theta_k=r_k e^{i\phi_k}$ is the complex representation of
$(\theta_{k,1},\theta_{k,2})$,
and $\eta_k^{(m)}=e^{-ik\alpha^{(m)}}(r_k/\tilde{\theta}_k)
\tilde{\eps}_k^{(m)}$ is a
rotation of the Gaussian noise. By Proposition \ref{lemma-distribution-1},
we still have $\eta_k^{(m)} \sim \Normal_\C(0,2)$ where these
remain independent across all $k=1,\ldots,K$ and $m=1,\ldots,N$.
Applying this to (\ref{eq:hatBkl}),
the factors $e^{ik\alpha^{(m)}},e^{il\alpha^{(m)}},e^{i(k+l)\alpha^{(m)}}$
cancel to yield
\begin{equation}\label{eq:hatBexpansion}
\hat{B}_{k,l}=\frac{1}{N}\sum_{m=1}^N
\tilde\theta_{k+l}\overline{\tilde\theta_k\tilde\theta_l}
\left(1+(\sigma/r_{k+l}) \eta_{k+l}^{(m)}\right)
\left(1+(\sigma/r_k) \overline{\eta_k^{(m)}}\right)
\left(1+(\sigma/r_l) \overline{\eta_l^{(m)}}\right)
=B_{k,l}(1+\mathrm{I}+\mathrm{II}+\mathrm{III})
\end{equation}
where
\begin{align*}
\mathrm{I}&=\frac{\sigma}{N}\sum_{m=1}^N
\frac{\overline{\eta_k^{(m)}}}{r_k}+\frac{\overline{\eta_l^{(m)}}}{r_l}
+\frac{\eta_{k+l}^{(m)}}{r_{k+l}}\\
\mathrm{II}&=\frac{\sigma^2}{N}\sum_{m=1}^N 
\frac{\overline{\eta_k^{(m)}}\overline{\eta_l^{(m)}}}{r_kr_l}+
\frac{\overline{\eta_k^{(m)}}\eta_{k+l}^{(m)}}{r_kr_{k+l}}
+\frac{\overline{\eta_l^{(m)}} \eta_{k+l}^{(m)}}{r_lr_{k+l}}\\
\mathrm{III}&=\frac{\sigma^3}{N}\sum_{m=1}^N
\frac{\overline{\eta_k^{(m)}}\overline{\eta_l^{(m)}}\eta_{k+l}^{(m)}}{r_kr_lr_{k+l}}.
\end{align*}

To bound $\mathrm{I}$, observe that
\begin{equation}\label{eq:deg1term}
\frac{\sigma}{N} \sum_{m=1}^N \frac{\Re \eta_k^{(m)}}{r_k}
\sim \Normal\left(0,\frac{\sigma^2}{Nr_k^2}\right),
\end{equation}
and similarly for the imaginary part and for the other two terms of $\mathrm{I}$. Then by a Gaussian tail bound,
\begin{equation}\label{eq:deg1finalbound}
\P[|\mathrm{I}| \geq t] \leq C\exp(-cNt^2\rlower^2/\sigma^2).
\end{equation}

To bound $\mathrm{II}$, consider first $k \neq l$ and
$\sum_m \Re \overline{\eta_k^{(m)}} \cdot \Re
\overline{\eta_l^{(m)}}$.
Each term $\Re \overline{\eta_k^{(m)}} \cdot \Re
\overline{\eta_l^{(m)}}$ is the product of two independent
standard Gaussian variables.
Then applying \cite[Corollary 1]{latala2006estimates} with $d=2$, $A=I$,
$\|A\|_{\{1,2\}}=\sqrt{N}$, and $\|A\|_{\{1\},\{2\}}=1$, we have
\[\P\left[\left|\frac{1}{N}\sum_{m=1}^N \Re \overline{\eta_k^{(m)}} \cdot \Re
\overline{\eta_l^{(m)}}\right| \geq t\right]
\leq Ce^{-cN(t \wedge t^2)}.\]
So
\begin{equation}\label{eq:degree2bound}
\P\left[\left|\frac{\sigma^2}{N}\sum_{m=1}^N
\frac{\Re \overline{\eta_k^{(m)}}}{r_k} \cdot \frac{\Re
\overline{\eta_l^{(m)}}}{r_l}\right| \geq t\right]
\leq C\exp\left(-cN\left(\frac{t\rlower^2}{\sigma^2} \wedge
\frac{t^2\rlower^4}{\sigma^4}\right)\right).
\end{equation}
The same bound holds for all products of real and imaginary parts of
$\eta_k^{(m)}$ and $\eta_l^{(m)}$, except for 
$\Re \eta_k^{(m)} \cdot \Re \eta_l^{(m)}$ and
$\Im \eta_k^{(m)} \cdot \Im \eta_l^{(m)}$ when $k=l$. For these products, we may
consider them together and apply
\[\frac{\sigma^2}{N}\sum_{m=1}^N \frac{(\Re \eta_k^{(m)})^2}{r_k^2}
-\frac{(\Im \eta_k^{(m)})^2}{r_k^2}
=\frac{2\sigma^2}{Nr_k^2}
\sum_{m=1}^N \frac{\Re \eta_k^{(m)}-\Im \eta_k^{(m)}}{\sqrt{2}} \cdot
\frac{\Re \eta_k^{(m)}+\Im \eta_k^{(m)}}{\sqrt{2}}\]
where now $(\Re \eta_k^{(m)}-\Im \eta_k^{(m)})/\sqrt{2}$ and
$(\Re \eta_k^{(m)}+\Im \eta_k^{(m)})/\sqrt{2}$ are independent standard Gaussian
variables. The bound (\ref{eq:degree2bound}) then holds for this sum, and this
shows
\[\P\left[\left|\frac{\sigma^2}{N}\sum_{m=1}^N \frac{\overline{\eta_k^{(m)}}
\overline{\eta_l^{(m)}}}{r_kr_l}\right|>t\right]
\leq C\exp\left(-cN\left(\frac{t\rlower^2}{\sigma^2} \wedge
\frac{t^2\rlower^4}{\sigma^4}\right)\right)\]
for the first term of $\mathrm{II}$.
Applying the same argument for the remaining two terms of $\mathrm{II}$,
\begin{equation}\label{eq:deg2finalbound}
\P[|\mathrm{II}| \geq t]
\leq C\exp\left(-cN\left(\frac{t\rlower^2}{\sigma^2} \wedge
\frac{t^2\rlower^4}{\sigma^4}\right)\right).
\end{equation}

We apply a similar argument to bound $\mathrm{III}$.
Consider first $k \neq l$ and
$\sum_m \Re \eta_k^{(m)} \cdot \Re \eta_l^{(m)} \cdot \Re
\eta_{k+l}^{(m)}$.
Each term $\Re \eta_k^{(m)} \cdot \Re \eta_l^{(m)} \cdot \Re \eta_{k+l}^{(m)}$
is the product of three independent standard Gaussian variables. Then
applying \cite[Corollary 1]{latala2006estimates} with $d=3$, $A=\sum_{m=1}^N e_m
\otimes e_m \otimes e_m$, $\|A\|_{\{1,2,3\}}=\sqrt{N}$,
$\|A\|_{\{1,2\},\{3\}}=1$, and $\|A\|_{\{1\},\{2\},\{3\}}=1$,
\[\P\left[\left|\frac{1}{N}\sum_{m=1}^N \Re \eta_k^{(m)} \cdot \Re \eta_l^{(m)}
\cdot \Re \eta_{k+l}^{(m)}\right| \geq t\right] \leq 
Ce^{-c(Nt^2 \wedge Nt \wedge (Nt)^{2/3})}
\leq Ce^{-cN(t^2 \wedge \frac{t^{2/3}}{N^{1/3}})}.\]
(The second inequality applies $t \geq t^2 \wedge
\frac{t^{2/3}}{N^{1/3}}$ for any $N \geq 1$ and $t \geq 0$.)
The same bound holds for all combinations of real and imaginary parts of
$\eta_k^{(m)},\eta_l^{(m)},\eta_{k+l}^{(m)}$, except again for products having
$\Re \eta_k^{(m)} \cdot \Re \eta_l^{(m)}$ or
$\Im \eta_k^{(m)} \cdot \Im \eta_l^{(m)}$ when $k=l$. These products may be
bounded by applying
\[\frac{1}{2} \Re \eta_{2k}^{(m)} \cdot \left((\Re \eta_k^{(m)})^2
-(\Im \eta_k^{(m)})^2\right)
=\Re \eta_{2k}^{(m)} \cdot
\frac{\Re \eta_k^{(m)}-\Im \eta_k^{(m)}}{\sqrt{2}} \cdot
\frac{\Re \eta_k^{(m)}+\Im \eta_k^{(m)}}{\sqrt{2}}\]
and similarly for $\Im \eta_{2k}^{(m)}\cdot ((\Re \eta_k^{(m)})^2
-(\Im \eta_k^{(m)})^2)$, where  $\Re \eta_{2k}^{(m)}$,
$\Im \eta_{2k}^{(m)}$, $(\Re \eta_k^{(m)}-\Im \eta_k^{(m)})/\sqrt{2}$, and
$(\Re \eta_k^{(m)}+\Im \eta_k^{(m)})/\sqrt{2}$ are independent standard Gaussian
variables. Thus
\begin{equation}\label{eq:deg3finalbound}
\P\left[|\mathrm{III}| \geq t\right] \leq 
C\exp\left(-cN\left(\frac{t^2\rlower^6}{\sigma^6} \wedge 
\frac{t^{2/3}\rlower^2}{N^{1/3}\sigma^2}\right)\right).
\end{equation}

Combining (\ref{eq:deg1finalbound}), (\ref{eq:deg2finalbound}), and
(\ref{eq:deg3finalbound}), for any $s>0$ we obtain
\[\P\Big[|\hat{B}_{k,l}/B_{k,l}-1| \geq s\Big]
\leq C\exp\left(-cN\left(\frac{s^2\rlower^2}{\sigma^2}
\wedge \frac{s\rlower^2}{\sigma^2}
\wedge \frac{s^2\rlower^4}{\sigma^4}
\wedge \frac{s^2\rlower^6}{\sigma^6}
\wedge \frac{s^{2/3}\rlower^2}{N^{1/3}\sigma^2}\right)\right).\]
We have
\[\frac{s\rlower^2}{\sigma^2}
\geq \frac{s^2\rlower^2}{\sigma^2} \wedge
\frac{s^{2/3}\rlower^2}{N^{1/3}\sigma^2}, \qquad
\frac{s^2\rlower^4}{\sigma^4}
\geq \frac{s^2\rlower^2}{\sigma^2} \wedge \frac{s^2\rlower^6}{\sigma^6},\]
so this simplifies to (\ref{eqn-B-final-bound}).

Finally, for any $z \in \C$ and any $s \in (0,1)$, observe
that $|z-1|<s$ implies $|\Arg z|<\arcsin s<\pi s/2$ for the principal
argument (\ref{eq:principalArg}). Then, recalling that
$\hat{\Phi}_{k,l}^\oracle-\Phi_{k,l}=\Arg (\hat{B}_{k,l}/B_{k,l})$ from
(\ref{eq:Argoracle}), we obtain for any $s \in (0,\pi/2)$ that
\[\P\Big[|\hat{\Phi}_{k,l}^\oracle-\Phi_{k,l}| \geq s\Big] \leq
\P\Big[|\hat{B}_{k,l}/B_{k,l}-1| \geq \frac{2s}{\pi}\Big]\]
and (\ref{eqn-Phi-final-bound}) follows.
\end{proof}

\begin{proof}[Proof of Corollary \ref{cor:riskPhikl}]
We apply $\E[X^2]=\int_0^\infty \P[|X| \geq s] \cdot 2s\,ds$
and Lemma \ref{lemma:Phikltail} to obtain, for universal constants $C,c>0$,
\begin{align*}
\E[(\hat\Phi_{k,l}^\oracle-\Phi_{k,l})^2] &\leq 
\int_0^{\pi/2} Cs\left(e^{-cNs^2\rlower^2/\sigma^2}+
e^{-cNs^2\rlower^6/\sigma^6}+e^{-c(Ns)^{2/3}\rlower^2/\sigma^2}\right)ds\\
&\hspace{2in}+C\left(e^{-cN\rlower^2/\sigma^2}+
e^{-cN\rlower^6/\sigma^6}+e^{-cN^{2/3}\rlower^2/\sigma^2}\right),
\end{align*}
where the second term bounds the integral from $s=\pi/2$ to $s=\pi$.
The result then follows from applying
$\int_0^\infty se^{-\alpha s^2}ds=\alpha^{-1}\int_0^\infty te^{-t^2}dt
\leq C/\alpha$ and $\int_0^\infty se^{-\alpha s^{2/3}}ds
=\alpha^{-3}\int_0^\infty t^3e^{-t^2} \cdot 3t^2\,dt
\leq C/\alpha^3$ for the first term,
$e^{-cx} \leq C/x$ and $e^{-cx} \leq C/x^3$ for the second term, and
$\sigma^6/N^2\rlower^6 \leq \sigma^6/N\rlower^6$.
\end{proof}

\begin{proof}[Proof of Lemma \ref{lemma:Phicovariance}]
Part (c) follows from Corollary \ref{cor:riskPhikl} and Cauchy-Schwarz. For part
(a), recall from (\ref{eq:Argoracle}) and the expression for $\hat{B}_{k,l}$ in
(\ref{eq:hatBexpansion}) that
\[\hat\Phi_{k,l}^\oracle-\Phi_{k,l}=\Arg(\hat{B}_{k,l}/B_{k,l})
=\Arg \frac{1}{N}\sum_{m=1}^N \left(1+(\sigma/r_{k+l}) \eta_{k+l}^{(m)}
\right)\left(1+(\sigma/r_k) \overline{\eta_k^{(m)}}\right)
\left(1+(\sigma/r_l) \overline{\eta_l^{(m)}}\right)\]
Since $\eta_k^{(m)}$ are independent across $k=1,\ldots,K$ and $m=1,\ldots,N$,
we obtain in the setting of part (a) that $\hat\Phi_{k,l}^\oracle-\Phi_{k,l}$ is
independent of $\hat\Phi_{x,y}^\oracle-\Phi_{x,y}$.
Furthermore, applying the conjugation symmetry of
Proposition \ref{lemma-distribution-1} to the variables $\eta_k^{(m)}$,
we have the equality in law
$\hat{B}_{k,l}/B_{k,l} \overset{L}{=} \overline{\hat{B}_{k,l}/B_{k,l}}$ for the
quantity inside $\Arg(\cdot)$. Since $\Arg z=-\Arg \overline{z}$ whenever
$\Arg z \neq -\pi$, and the probability is 0 that
$\Arg \hat{B}_{k,l}/B_{k,l}=-\pi$ exactly, this equality in law implies
the sign symmetry (\ref{eq:symmetryinsign}). Hence
$\E[\hat{\Phi}_{k,l}^\oracle-\Phi_{k,l}]=0$. This shows part (a).

It remains to show part (b). Let $\Ln$ denote the principal value of the
complex logarithm with branch cut on the negative real line, so that
$\Arg z=\Im \Ln z$ whenever $\Arg z \neq -\pi$. Denote
$\delta_{k,l}=\hat B_{k,l}/B_{k,l}-1$. Then (with probability 1)
\begin{align*}
(\hat\Phi_{k,l}^\oracle-\Phi_{k,l})(\hat\Phi_{x,y}^\oracle-\Phi_{x,y})
&=\Arg \frac{\hat B_{k,l}}{B_{k,l}} \cdot \Arg \frac{\hat B_{x,y}}{B_{x,y}}
=\Im \Ln\left(1+\delta_{k,l}\right) \cdot \Im\Ln\left(1+\delta_{x,y}\right).
\end{align*}
Let us fix an integer $J=J(N,\rlower^2,\sigma^2) \geq 1$ to be determined, and
apply a Taylor expansion of $t \mapsto \Ln(1+t\delta)$ around $t=0$ to write
\[\Im \Ln(1+\delta)=q(\delta)+r(\delta), 
\qquad q(\delta)=\Im \sum_{j=1}^J \frac{(-1)^{j-1}}{j} \delta^j,
\qquad r(\delta)=\Im \int_0^1 \delta^{J+1} \cdot
\frac{(-1)^J(1-t)^J}{(1+t\delta)^{J+1}}dt.\]
Define the event
$\cE=\{|\delta_{k,l}|<1/2 \text{ and } |\delta_{x,y}|<1/2\}$. We may then
apply the approximation
\[\E\left[(\hat\Phi_{k,l}^\oracle-\Phi_{k,l})(\hat\Phi_{x,y}^\oracle-\Phi_{x,y})\right]=\E\Big[q(\delta_{k,l}) \cdot q(\delta_{x,y})\Big]
+\mathrm{I}+\mathrm{II}+\mathrm{III}\]
where we define the three error terms
\begin{align*}
\mathrm{I}&=-\E\Big[\1\{\cE^c\} \cdot q(\delta_{k,l}) \cdot
q(\delta_{x,y})\Big]\\
\mathrm{II}&=\E\left[\1\{\cE\} \Big(q(\delta_{k,l}) \cdot r(\delta_{x,y})
+r(\delta_{k,l}) \cdot q(\delta_{x,y})
+r(\delta_{k,l}) \cdot r(\delta_{x,y})\Big)\right]\\
\mathrm{III}&=\E\left[\1\{\cE^c\} \cdot
(\hat\Phi_{k,l}^\oracle-\Phi_{k,l})(\hat\Phi_{x,y}^\oracle-\Phi_{x,y})\right]
\end{align*}

To bound these errors, let $C,C',c,c'>0$ denote
universal constants changing from instance to instance. Recall
from (\ref{eqn-B-final-bound}) that
\begin{equation}\label{eq:Tayloreventbound}
\P[\cE^c] \leq \P[|\delta_{k,l}| \geq 1/2]+\P[|\delta_{x,y}| \geq 1/2]
\leq Ce^{-c(\frac{N\rlower^6}{\sigma^6}
\wedge \frac{N^{2/3}\rlower^2}{\sigma^2})}.
\end{equation}
Also, for any $j \geq 1$,
applying $\E[|X|^j]=\int_0^\infty \P[|X| \geq s] \cdot js^{j-1}\,ds$ 
with $X=2\delta_{k,l}$, and applying also for $Z \sim \Normal(0,1)$ that
$\E[|Z|^j] \leq \E[Z^{2j}]^{1/2}=[(2j-1)!!]^{1/2}
\leq (2j)^{(j-1)/2}$, we have from (\ref{eqn-B-final-bound}) that
\begin{align*}
\E[|2\delta_{k,l}|^j] &\leq 
\int_0^\infty js^{j-1} \cdot C(e^{-cNs^2\rlower^2/\sigma^2}
+e^{-cNs^2\rlower^6/\sigma^6}+e^{-c(Ns)^{2/3}\rlower^2/\sigma^2})\,ds\\
&=Cj \left(\left(\frac{\sigma^j}{\rlower^j N^{j/2}}
+\frac{\sigma^{3j}}{\rlower^{3j} N^{j/2}} \right)
\cdot \int_0^\infty t^{j-1}e^{-ct^2}dt
+\left(\frac{\sigma^{3j}}{\rlower^{3j} N^j}\right)
\cdot \int_0^\infty t^{3j-3}e^{-ct^2}\cdot 3t^2\,dt\right)\\
&\leq (C_0j)^{\frac{j}{2}}\left(\frac{\sigma^j}{\rlower^j N^{j/2}}
+\frac{\sigma^{3j}}{\rlower^{3j} N^{j/2}} \right)
+(C_0j)^{\frac{3j}{2}}\left(\frac{\sigma^{3j}}{\rlower^{3j} N^j}\right)
\end{align*}
where $C_0$ in the last line is a universal constant, which we will later
assume satisfies $C_0 \geq 3$. Let us set
\begin{equation}\label{eq:Jchoice}
J=\left\lfloor \frac{1}{4C_0e}
\left(\frac{N\rlower^6}{\sigma^6}
\wedge \frac{N^{2/3}\rlower^2}{\sigma^2}\right)\right\rfloor
\end{equation}
for this constant $C_0>0$. Note that if the quantity inside $\lfloor \cdot
\rfloor$ is less than 1, then the statement of part (b) holds since the left
side of (\ref{eq:card1bound}) is at most $\pi^2$, and the right side is an
arbitrarily large constant. Thus, we may assume henceforth that $J \geq 1$.
The above gives
\[\E[|2\delta_{k,l}|^j] \leq 2\left(\frac{j}{4eJ}\right)^{j/2}
+\left(\frac{j}{4eJ}\right)^{3j/2}.\]
Applying $|q(\delta)| \leq J\max(|\delta|,|\delta|^J)$, and also
$|r(\delta)| \leq |2\delta|^{J+1}$ for $|\delta|<1/2$, this shows
\begin{align*}
\E[|q(\delta_{k,l})|^2] &\leq J^2\E[|\delta_{k,l}|^2]
+J^2\E[|\delta|^{2J}] \leq CJ+J^2e^{-cJ},\\
\E[|q(\delta_{k,l})|^3] &\leq J^3\E[|\delta_{k,l}|^3]
+J^3\E[|\delta|^{3J}] \leq CJ^{3/2}+J^3e^{-cJ},\\
\E[\1\{\cE\}|r(\delta_{k,l})|^2] &\leq \E[|\delta_{k,l}|^{2J+2}] \leq Ce^{-cJ}.
\end{align*}
Then, applying these bounds together with (\ref{eq:Tayloreventbound}),
H\"older's inequality, and Cauchy-Schwarz,
\[|\mathrm{I}|+|\mathrm{II}|+|\mathrm{III}| \leq CJ^2 e^{-cJ} \leq C'e^{-c'J}.\]
This gives the second term on the right side of (\ref{eq:card1bound}).

Finally, let us bound the dominant term $\E[q(\delta_{k,l})q(\delta_{x,y})]$
using the condition that $\{k,l,k+l\}
\cap \{x,y,x+y\}$ has cardinality 1. Applying $2\Im u \cdot \Im v=\Re
u\bar{v}-\Re uv$, we have
\[\E[q(\delta_{k,l})q(\delta_{x,y})]=\sum_{i,j=1}^J \frac{(-1)^{i+j}}{ij}
\E[\Im \delta_{k,l}^i \cdot \Im \delta_{x,y}^j]
=\sum_{i,j=1}^J \frac{(-1)^{i+j}}{2ij}
\Big(\Re \E[\delta_{k,l}^i\overline{\delta_{x,y}^j}]
-\Re \E[\delta_{k,l}^i\delta_{x,y}^j]\Big).\]
From the expression for $\hat{B}_{k,l}$ in (\ref{eq:hatBexpansion}), observe
that
\[\delta_{k,l}=\frac{\hat{B}_{k,l}}{B_{k,l}}-1=\frac{1}{N}\sum_{m=1}^N
\left(1+\frac{\sigma}{r_{k+l}}\eta_{k+l}^{(m)}\right)\left(1+\frac{\sigma}{r_k}\overline{\eta_k^{(m)}}\right)\left(1+\frac{\sigma}{r_l}\overline{\eta_l^{(m)}}\right)-1.\]
We view this as a polynomial in the variables
$\{\eta_{k+l}^{(m)},
\overline{\eta_k^{(m)}},\overline{\eta_l^{(m)}}:m=1,\ldots,N\}$ where,
after canceling $+1$ with $-1$, each monomial has total degree at least 1
in these variables. We consider three cases.

{\bf Case 1:} $k+l=x+y$. This allows possibly $k=l$ and/or $x=y$,
but ensures $\{k,l\} \cap \{x,y\}=\emptyset$ since $\{k,l,k+l\} \cap
\{x,y,x+y\}$ has cardinality 1. We may
expand $\delta_{k,l}^i\delta_{x,y}^j$ as
a sum of monomials in $\eta_{k+l}^{(m)},\overline{\eta_k^{(m)}},
\overline{\eta_l^{(m)}},\overline{\eta_x^{(m)}},\overline{\eta_y^{(m)}}$ with
degree at least 1, and observe that $k+l=x+y$ is distinct from $\{k,l,x,y\}$
because it is strictly greater in value.
Then (\ref{eq:gaussiansymmetry}) from Proposition
\ref{lemma-distribution-1} implies $\E[\delta_{k,l}^i \delta_{x,y}^j]=0$.
We may also expand $\delta_{k,l}^i \overline{\delta_{x,y}^j}$ as a sum of
monomials in $\eta_{k+l}^{(m)}, \overline{\eta_k^{(m)}},
\overline{\eta_l^{(m)}},\overline{\eta_{k+l}^{(m)}},\eta_x^{(m)},
\eta_y^{(m)}$. Since $\{k,l\}$ are distinct from $\{x,y,k+l\}$, any monomial
involving $\overline{\eta_k^{(m)}},\overline{\eta_l^{(m)}}$ has vanishing
expectation. Similarly, any monomial involving
$\eta_x^{(m)},\eta_y^{(m)}$ has vanishing expectation. Thus
the only non-vanishing terms are
\begin{equation}\label{eq:deltaprodexpansion1}
\E[\delta_{k,l}^i\overline{\delta_{x,y}^j}]
=\E\left[\left(\frac{1}{N}\sum_{m=1}^N
\frac{\sigma}{r_{k+l}}\eta_{k+l}^{(m)}\right)^i
\left(\frac{1}{N}\sum_{m=1}^N
\frac{\sigma}{r_{k+l}}\overline{\eta_{k+l}^{(m)}}\right)^j\right].
\end{equation}
Then applying the equality in law
$N^{-1}\sum_{m=1}^N (\sigma/r_{k+l})\eta_{k+l}^{(m)} \overset{L}{=}
\eta \cdot \sigma/(r_{k+l}\sqrt{N})$ where $\eta \sim \Normal_\C(0,2)$,
together with (\ref{eq:gaussiansymmetry}) and (\ref{eq:gaussianmoments})
and the bound $j! \leq j^j$,
\[\Big|\E[\delta_{k,l}^i\overline{\delta_{x,y}^j}]\Big|
=\left(\frac{\sigma}{r_{k+l}\sqrt{N}}\right)^{i+j}\E[\eta^i\overline{\eta^j}]
\leq \1\{i=j\}\left(\frac{4j\sigma^2}{N\rlower^2}\right)^j.\]
So, recalling the definition of $J$ from (\ref{eq:Jchoice}) where $C_0 \geq 3$,
\[\Big|\E[q(\delta_{k,l})q(\delta_{x,y})]\Big|
\leq \sum_{j=1}^J \frac{1}{2j^2}
\left(\frac{4j\sigma^2}{N\rlower^2}\right)^j
\leq \left(\frac{2\sigma^2}{N\rlower^2}\right)
\sum_{j=1}^J \left(\frac{4J\sigma^2}{N\rlower^2}\right)^{j-1}
\leq \left(\frac{2\sigma^2}{N\rlower^2}\right)\sum_{j=1}^\infty
\left(\frac{1}{C_0e}\right)^{j-1}.\]
Thus we obtain, for a universal constant $C>0$,
\begin{equation}\label{eq:IVbound}
\Big|\E[q(\delta_{k,l})q(\delta_{x,y})]\Big|
\leq \frac{C\sigma^2}{N\rlower^2}.
\end{equation}
This concludes the proof in Case 1.

{\bf Case 2:} $k=x+y$. (By symmetry, this addresses also $l=x+y$,
$x=k+l$, and $y=k+l$.) Then $\{k,l\} \cap \{x,y,k+l\}=\emptyset$
and $\{k+l\} \cap \{x,y\}=\emptyset$, because $k+l$ is greater than $\{k,l\}$,
$k$ is greater than $\{x,y\}$, and $l=x$ or $l=y$
would imply that $\{k,l,k+l\} \cap \{x,y,x+y\}$ has cardinality 2.
We may expand $\delta_{k,l}^i\overline{\delta_{x,y}^j}$
as a sum of monomials in $\eta_{k+l}^{(m)},\overline{\eta_k^{(m)}},
\overline{\eta_l^{(m)}},\eta_x^{(m)},\eta_y^{(m)}$. Since $\{k,l\}$ are distinct
from $\{x,y,k+l\}$, (\ref{eq:gaussiansymmetry}) implies
$\E[\delta_{k,l}^i\overline{\delta_{x,y}^j}]=0$. We may also expand
$\delta_{k,l}^i\delta_{x,y}^j$ as a sum of monomials in
$\eta_{k+l}^{(m)},\overline{\eta_k^{(m)}},\overline{\eta_l^{(m)}},
\eta_k^{(m)},\overline{\eta_x^{(m)}},\overline{\eta_y^{(m)}}$. Here, $k+l$ is
distinct from $\{k,l,x,y\}$, and $\{x,y\}$ are distinct from $\{k,k+l\}$,
so any monomial involving
$\eta_{k+l}^{(m)},\overline{\eta_x^{(m)}},\overline{\eta_y^{(m)}}$ has vanishing
expectation. If $l \neq k$, then also $l$ is distinct from $\{k,k+l\}$ so
monomials involving $\overline{\eta_l^{(m)}}$ have vanishing expectation,
yielding
\[\E[\delta_{k,l}^i\delta_{x,y}^j]
=\E\left[\left(\frac{1}{N}\sum_{m=1}^N
\frac{\sigma}{r_k}\overline{\eta_k^{(m)}}\right)^i
\left(\frac{1}{N}\sum_{m=1}^N \frac{\sigma}{r_k}\eta_k^{(m)}\right)^j\right].\]
This is analogous to (\ref{eq:deltaprodexpansion1}), and the same argument as
above gives (\ref{eq:IVbound}).

If instead, $l=k$, then we obtain that the only non-zero terms of
$\E[\delta_{k,l}^i\delta_{x,y}^j]$ are
\begin{equation}\label{eq:deltaprodexpansion2}
\E\left[\delta_{k,l}^i\delta_{x,y}^j\right]
=\E\left[\left(\frac{1}{N}\sum_{m=1}^N
\frac{2\sigma}{r_k}\overline{\eta_k^{(m)}}+\frac{\sigma^2}{r_k^2}
(\overline{\eta_k^{(m)}})^2\right)^i \left(\frac{1}{N}\sum_{m=1}^N
\frac{\sigma}{r_k}\eta_k^{(m)}\right)^j\right]
\end{equation}
Let us distribute this product and then factor the expectations of the
resulting terms, using independence of
$\{\eta_k^{(m)}:m=1,\ldots,N\}$.
We write $\sum_{(a_1,\ldots,a_N)|a}$ for the sum over all
tuples of nonnegative integers $(a_1,\ldots,a_N)$ that sum to $a$. Then the
above may be rewritten as
\begin{align*}
\E\left[\delta_{k,l}^i\delta_{x,y}^j\right]
&=\mathop{\sum_{a,b \geq 0}}_{a+b=i}
\binom{i}{a} \E\left[\left(\frac{1}{N}\sum_{m=1}^N
\frac{2\sigma}{r_k}\overline{\eta_k^{(m)}}\right)^a
\left(\frac{1}{N}\sum_{m=1}^N
\frac{\sigma^2}{r_k^2}(\overline{\eta_k^{(m)}})^2\right)^b
\left(\frac{1}{N}\sum_{m=1}^N \frac{\sigma}{r_k}\eta_k^{(m)}\right)^j\right]\\
&=\mathop{\sum_{a,b \geq 0}}_{a+b=i} \binom{i}{a}2^aN^b \E\left[
\left(\sum_{m=1}^N \frac{\sigma}{Nr_k}\overline{\eta_k^{(m)}}\right)^a
\left(\sum_{m=1}^N
\left(\frac{\sigma}{Nr_k}\overline{\eta_k^{(m)}}\right)^2\right)^b
\left(\sum_{m=1}^N \frac{\sigma}{Nr_k}\eta_k^{(m)}\right)^j\right]\\
&=\mathop{\sum_{a,b \geq 0}}_{a+b=i} \binom{i}{a} 2^aN^b
\sum_{(a_1,\ldots,a_N)|a}\,\sum_{(b_1,\ldots,b_N)|b}\,
\sum_{(j_1,\ldots,j_N)|j}\\
&\hspace{0.2in} \binom{a}{a_1,\ldots,a_N}\binom{b}{b_1,\ldots,b_N}
\binom{j}{j_1,\ldots,j_N} \prod_{m=1}^N
\E\left[\left(\frac{\sigma \overline{\eta_k^{(m)}}}{Nr_k}\right)^{a_m+2b_m}
\left(\frac{\sigma \eta_k^{(m)}}{Nr_k}\right)^{j_m}\right]
\end{align*}
where the last line uses that there are $\binom{a}{a_1,\ldots,a_N}$ ways to
choose $a_1$ of the factors $(\sum_{m=1}^N
\frac{\sigma}{Nr_k}\overline{\eta_k^{(m)}})^a$ to correspond to $m=1$, $a_2$ to
correspond to $m=2$, etc., and similarly for $b$ and $j$.

Then, applying (\ref{eq:gaussiansymmetry}) and (\ref{eq:gaussianmoments}),
\begin{align*}
\left|\E\left[\delta_{k,l}^i\delta_{x,y}^j\right]\right|
&\leq \mathop{\sum_{a,b \geq 0}}_{a+b=i} \binom{i}{a} 2^aN^b
\sum_{(a_1,\ldots,a_N)|a}\,\sum_{(b_1,\ldots,b_N)|b}\,
\sum_{(j_1,\ldots,j_N)|j} \\
&\hspace{0.5in} \binom{a}{a_1,\ldots,a_N}\binom{b}{b_1,\ldots,b_N}
\binom{j}{j_1,\ldots,j_N} \prod_{m=1}^N \1\{a_m+2b_m=j_m\}
\left(\frac{4\sigma^2}{N^2r_k^2}\right)^{j_m} j_m!
\end{align*}
Observe that $\binom{j}{j_1,\ldots,j_N} \cdot \prod_{m=1}^N j_m!=j! \leq j^j$.
The condition $a_m+2b_m=j_m$ for every $m=1,\ldots,N$ requires
$a+2b=\sum_m a_m+2b_m=\sum_m j_m=j$. For $(a,b,j)$ satisfying this requirement,
fixing any partitions $(a_1,\ldots,a_N)|a$ and $(b_1,\ldots,b_N)|b$, there is
exactly one partition $(j_1,\ldots,j_N)|j$ for which $a_m+2b_m=j_m$ holds
for every $m=1,\ldots,N$. Thus, the above gives
\begin{align*}
\left|\E\left[\delta_{k,l}^i\delta_{x,y}^j\right]\right|
\leq \left(\frac{4j\sigma^2}{N^2\rlower^2}\right)^j
\mathop{\sum_{a,b \geq 0}}_{a+b=i,\,a+2b=j} \binom{i}{a} 2^aN^b
\sum_{(a_1,\ldots,a_N)|a}\,\sum_{(b_1,\ldots,b_N)|b}\,
\binom{a}{a_1,\ldots,a_N}\binom{b}{b_1,\ldots,b_N}.
\end{align*}
Observe now that $\sum_{(a_1,\ldots,a_N)|a} \binom{a}{a_1,\ldots,a_N}$ counts
exactly the number of assignments of each of $a$ labeled objects to $N$ bins,
by first determining the number of objects in each bin, followed by their
identities. So $\sum_{(a_1,\ldots,a_N)|a} \binom{a}{a_1,\ldots,a_N}=N^a$.
Applying the similar identity for $b$ and $N^{a+2b}=N^j$ above,
\[\left|\E\left[\delta_{k,l}^i\delta_{x,y}^j\right]\right|
\leq \left(\frac{4j\sigma^2}{N\rlower^2}\right)^j
\mathop{\sum_{a,b \geq 0}}_{a+b=i,\,a+2b=j} \binom{i}{a} 2^a.\]
Then, recalling that $\E\left[\delta_{k,l}^i\overline{\delta_{x,y}^j}\right]=0$,
\[\Big|\E[q(\delta_{k,l})q(\delta_{x,y})]\Big|
\leq \sum_{i,j=1}^J \frac{1}{2ij}
\left|\E\left[\delta_{k,l}^i\delta_{x,y}^j\right]\right|
\leq \sum_{j=1}^J \frac{1}{2j}
\left(\frac{4j\sigma^2}{N\rlower^2}\right)^j
\mathop{\sum_{a,b \geq 0}}_{a+2b=j} \frac{1}{a+b}\binom{a+b}{a}2^a.\]
We may apply
\[\mathop{\sum_{a,b \geq 0}}_{a+2b=j} \frac{1}{a+b}\binom{a+b}{a}2^a
\leq \sum_{a=0}^j \binom{j}{a}2^a=3^j.\]
Then, recalling the definition of $J$ from (\ref{eq:Jchoice}) where
$C_0 \geq 3$,
\[\Big|\E[q(\delta_{k,l})q(\delta_{x,y})]\Big|
\leq \sum_{j=1}^J \frac{1}{2j}
\left(\frac{12j\sigma^2}{N\rlower^2}\right)^j
\leq \frac{6\sigma^2}{N\rlower^2}\sum_{j=1}^J 
\left(\frac{12J\sigma^2}{N\rlower^2}\right)^{j-1}
\leq \frac{6\sigma^2}{N\rlower^2}\sum_{j=1}^\infty
\left(\frac{3}{C_0e}\right)^{j-1}.\]
This again yields (\ref{eq:IVbound}), and concludes the proof in Case 2.

{\bf Case 3:} $k=x$. (By symmetry, this addresses also $k=y$, $l=x$, and $l=y$.)
This ensures that $\{k+l,k+y\} \cap \{k,l,y\}=\emptyset$ and $k+l \neq k+y$,
because $k+l$ is greater than $\{k,l\}$, $k+y$ is greater than $\{k,y\}$,
and $k+l=y$ or $k+y=l$ or $k+l=k+y$
would lead to $\{k,l,k+l\} \cap \{x,y,x+y\}$ having cardinality 2.
We may expand $\delta_{k,l}^i\delta_{x,y}^j$ as monomials in
$\eta_{k+l}^{(m)},\overline{\eta_k^{(m)}},\overline{\eta_l^{(m)}},
\eta_{k+y}^{(m)},\overline{\eta_y^{(m)}}$. Since $\{k+l,k+y\}$ are distinct from
$\{k,l,y\}$, (\ref{eq:gaussiansymmetry}) implies
$\E[\delta_{k,l}^i\delta_{x,y}^j]=0$. We may also expand
$\delta_{k,l}^i\overline{\delta_{x,y}^j}$ as monomials in
$\eta_{k+l}^{(m)},\overline{\eta_k^{(m)}},\overline{\eta_l^{(m)}},
\overline{\eta_{k+y}^{(m)}},\eta_k^{(m)},\eta_y^{(m)}$. Since $k+l$ is distinct
from $\{k,l,k+y\}$ and $k+y$ is distinct from $\{k,y,k+l\}$, 
(\ref{eq:gaussiansymmetry}) implies that any monomials
involving $\eta_{k+l}^{(m)}$ or $\overline{\eta_{x+y}^{(m)}}$ have vanishing
expectation. Note that since $k+l \neq k+y$, also $l \neq y$.
If $k$ is distinct from $\{l,y\}$, then 
monomials involving $\overline{\eta_l^{(m)}}$ or $\eta_y^{(m)}$ also have
vanishing expectation, so
\[\E[\delta_{k,l}^i\overline{\delta_{x,y}^j}]
=\E\left[\left(\frac{1}{N}\sum_{m=1}^N
\frac{\sigma}{r_k}\overline{\eta_k^{(m)}}\right)^i
\left(\frac{1}{N}\sum_{m=1}^N
\frac{\sigma}{r_k}\eta_k^{(m)}\right)^j\right].\]
This is analogous to (\ref{eq:deltaprodexpansion1}), and the same argument as in
Case 1 leads to (\ref{eq:IVbound}). If instead $k=l$ (which by symmetry
addresses also $k=y$), then
\[\E[\delta_{k,l}^i\overline{\delta_{x,y}^j}]
=\E\left[\left(\frac{1}{N}\sum_{m=1}^N
\frac{2\sigma}{r_k}\overline{\eta_k^{(m)}}
+\frac{\sigma^2}{r_k^2}(\overline{\eta_k^{(m)}})^2\right)^i
\left(\frac{1}{N}\sum_{m=1}^N
\frac{\sigma}{r_k}\eta_k^{(m)}\right)^j\right].\]
This is analogous to (\ref{eq:deltaprodexpansion2}), and the same argument as in
Case 2 leads to (\ref{eq:IVbound}). This concludes the proof in Case 3.
Combining these three cases shows (\ref{eq:card1bound}).
\end{proof}

\subsection{Oracle estimation of $\phi_k$}

\begin{proof}[Proof of Lemma \ref{lemma:Mproperties}]
Suppose $M\phi=0$. Then for all $(k,l) \in \cI$,
$\phi_{k+l}=\phi_k+\phi_l$.
Then $\phi_2=\phi_1+\phi_1=2\phi_1$, $\phi_3=\phi_2+\phi_1=3\phi_1$, etc., and
$\phi_K=\phi_{K-1}+\phi_1=K\phi_1$, so $\phi$ is a multiple of
$(1,2,3,\ldots,K)$. Conversely, any multiple of $(1,2,3,\ldots,K)$ satisfies
$M\phi=0$ by the definition of $M$. This shows the first statement.

Denote $T=M^\top M$. We explicitly compute $T$: 
Let $M_k \in \R^{\cI}$ be the $k^\text{th}$ column of $M$. The diagonal entries of $T$ are
$T_{kk}=\|M_k\|^2$. If $2k>K$, then the non-zero entries of $M_k$
correspond to the $(i,j)$ pairs
\begin{align*}
(i,j)=(1,k-1),\ldots,(k-1,1)&: M_{(i,j),k}=1\\
(i,j)=(1,k),\ldots,(K-k,k)&: M_{(i,j),k}=-1\\
(i,j)=(k,1),\ldots,(k,K-k)&: M_{(i,j),k}=-1
\end{align*}
So $T_{kk}=(k-1)+2(K-k)=2K-1-k$. If $2k \leq K$, then the non-zero entries of
$M_k$ correspond to the $(i,j)$ pairs
\begin{align*}
(i,j)=(1,k-1),\ldots,(k-1,1)&: M_{(i,j),k}=1\\
(i,j)=(1,k),\ldots,(k-1,k),(k+1,k),\ldots,(K-k,k)&: M_{(i,j),k}=-1\\
(i,j)=(k,1),\ldots,(k,k-1),(k,k+1),\ldots,(k,K-k)&: M_{(i,j),k}=-1\\
(i,j)=(k,k)&: M_{(i,j),k}=-2
\end{align*}
So $T_{kk}=(k-1)+2(K-k-1)+4=2K+1-k$. Thus, for all $k=1,\ldots,K$,
\[T_{kk}=2K+1-k-2\cdot \1\{2k>K\}.\]
For $1\leq j<k \leq K$, when $j+k>K$, we have $T_{jk}=M_j^\top M_k=-2$
where the only non-zero contributions to this inner-product come from rows
$(k-j,j),(j,k-j) \in \cI$. When $j+k \leq K$, the non-zero contributions come
from rows $(k-j,j),(j,k-j),(j,k),(k,j) \in \cI$, and these cancel
exactly to yield $T_{jk}=M_j^\top M_k=0$. Combining these diagonal and
off-diagonal components, $T$ has the form
\[T=\left(\begin{matrix}2K&~&~&~\\~&2K-1&~&~\\~&~&\ddots&~\\~&~&~&K+1\end{matrix}\right)-\left(\begin{matrix}~&~&~&2\\~&~&2&2\\~&\begin{rotate}{90}$\ddots$\end{rotate}& \vdots & \vdots \\2&\ldots&2&2\end{matrix}\right)\]
where the second matrix accounts also for the term $-2 \cdot \1\{2k>K\}$ of the
diagonal entries.

Now let $\lambda>0$ be a positive eigenvalue of $T$, with non-zero eigenvector
$x=(x_1,x_2,\ldots,x_K)$. This must be orthogonal to the null vector
$(1,2,\ldots,K)$ so $x_1+2x_2+\ldots+Kx_K=0$.
From the above form of $T$, the equation $Tx=\lambda x$ may be arranged as the
linear system
\begin{align*}
(2K-\lambda)x_1&= 2x_K\\
(2K-1-\lambda)x_2&= 2(x_{K-1}+x_K)\\
&\vdots\\
(K+2-\lambda)x_{K-1}&=2(x_2+\ldots+x_K)\\
(K+1-\lambda)x_K &= 2(x_1+x_2+\ldots+x_K).
\end{align*}    
Summing these equations and adding $x_1+2x_2+\ldots+Kx_K$ to both sides, we
obtain
\[(2K+1-\lambda)(x_1+x_2+\ldots+x_K)=
2(x_1+2x_2+\ldots+Kx_K)+(x_1+2x_2+\ldots+Kx_K)=0.\]
If $x_1+x_2+\ldots+x_K \neq 0$, then this implies $\lambda=2K+1$.
If $x_1+x_2+\ldots+x_K=0$, but $\lambda\notin\{K+1,K+2,\ldots,2K\}$, then
from the above linear system, we have the implications
\begin{align*}
(K+1-\lambda)x_K = 2(x_1+x_2+\ldots+x_K)=0 &\Rightarrow x_K=0\\
(2K-\lambda)x_1 = 2x_K=0 &\Rightarrow x_1=0\\
(K+2-\lambda)x_{K-1}=2(x_2+x_3+\ldots+x_K)=2(0-x_1)=0 &\Rightarrow x_{K-1}=0\\
(2K-1-\lambda)x_2= 2(x_{K-1}+x_K)=0 &\Rightarrow x_2=0
\end{align*}
and so forth. Then $x_1=x_2=\ldots=x_K=0$, which contradicts $x\neq 0$. Thus
any positive eigenvalue $\lambda$ of $T$ is one of the values
$\{K+1,K+2,\ldots,2K+1\}$.
\end{proof}

\begin{proof}[\revise{Proof of Theorem \ref{thm:oracleMoM},
$\hat{\theta}=\hat{\theta}^\oracle$}]
Recall the loss upper bound from Proposition \ref{prop-lossgeneral}.
For constants $C,C'>0$, applying $ab \leq a^2+b^2$,
\begin{align*}
L(\hat\theta^\oracle,\theta^*) &\leq \sum_{k=1}^K (\hat{r}_k-r_k)^2
+C\inf_{\alpha \in \R} \sum_{k=1}^K \hat{r}_kr_k
\big|\hat\phi_k^\oracle-\phi_k+k\alpha\big|_{\cA}^2\\
&=\sum_{k=1}^K (\hat{r}_k-r_k)^2
+C\inf_{\alpha \in \R} \sum_{k=1}^K [r_k^2+r_k(\hat{r}_k-r_k)]
\big|\hat\phi_k^\oracle-\phi_k+k\alpha\big|_{\cA}^2\\
&\leq (C+1)\sum_{k=1}^K (\hat{r}_k-r_k)^2
+C\inf_{\alpha \in \R} \sum_{k=1}^K r_k^2\left(
\big|\hat\phi_k^\oracle-\phi_k+k\alpha\big|_{\cA}^2
+\big|\hat\phi_k^\oracle-\phi_k+k\alpha\big|_{\cA}^4\right)\\
&\leq C'\left(\sum_{k=1}^K (\hat{r}_k-r_k)^2
+\inf_{\alpha \in \R} \sum_{k=1}^K r_k^2
\big|\hat\phi_k^\oracle-\phi_k+k\alpha\big|_\cA^2\right).
\end{align*}

The expectation may be bounded using Corollaries \ref{cor:MoMrriskbound}
and \ref{cor:MoMphiriskbound}. For the bound (\ref{eq:phiriskbound}) of 
Corollary \ref{cor:MoMphiriskbound}, fixing $c>0$ be the constant in the
exponent, observe that the given condition for $N$ with $C_0>0$ large enough
implies
\[c\left(\frac{N\rlower^6}{\sigma^6} \wedge \frac{N^{2/3}\rlower^2}{\sigma^2}\right)
\geq \frac{c}{2}\left(\frac{N\rlower^6}{\sigma^6} \wedge \frac{N^{2/3}\rlower^2}{\sigma^2}\right)
+3\log K.\]
We may then apply $e^{-(c/2)x},e^{-(c/2)x^{2/3}} \leq C/x$
for a constant $C>0$ to obtain
\begin{equation}\label{eq:exptopoly}
e^{-c(\frac{N\rlower^6}{\sigma^6} \wedge \frac{N^{2/3}\rlower^2}{\sigma^2})}
\leq \frac{C}{K^3}\left(
\frac{\sigma^6}{N\rlower^6}+\frac{\sigma^3}{N\rlower^3}\right)
\leq \frac{C'}{K^3}\left(\frac{\sigma^2}{N\rlower^2}
+\frac{\sigma^6}{N\rlower^6}\right).
\end{equation}
Then Corollary \ref{cor:MoMphiriskbound} gives simply
\revise{
\[\E\left[\inf_{\alpha \in \R} \sum_{k=1}^K
r_k^2|\hat\phi_k^\oracle-\phi_k+k\alpha|_\cA^2\right] \leq 
\frac{C'\|\theta^*\|^2}{K}
\left(\frac{K\sigma^2}{N\rlower^2}+\frac{\sigma^6}{N\rlower^6}\right)\]}
and combining this with Corollary \ref{cor:MoMrriskbound} yields the lemma.
\end{proof}

\subsection{Estimation of $\Phi_{k,l}$ by optimization}

\begin{proof}[Proof of Lemma \ref{lemma:phasestability}]
Set $v=\phi-\phi'$. We must show: If $v \in \R^K$ is such that
\begin{equation}\label{eq:Mphismall}
|v_{k+l}-v_k-v_l|_{\cA} \leq \delta \text{ for all } (k,l) \in \cI
\end{equation}
then there exists $\alpha \in \R$ with $|v_k-k\alpha|_\cA \leq \delta$ for all
$k=1,\ldots,K$.

We induct on $K$. For $K=1$ the result holds trivially by setting
$\alpha=v_1$. Suppose the result holds for $K-1$. Consider $v \in \R^K$
that satisfies
(\ref{eq:Mphismall}). By the induction hypothesis, there exists $\alpha \in \R$
such that $|v_k-k\alpha|_{\cA} \leq \delta$ for $k=1,\ldots,K-1$. Then by
the triangle inequality, it is immediate to see that
\[|v_K-K\alpha|_\cA
\leq |v_K-v_1-v_{K-1}|_\cA
+|v_1-\alpha|_\cA+|v_{K-1}-(K-1)\alpha|_\cA \leq 3\delta.\]
To complete the induction, we must show the stronger bound of $\delta$ instead
of $3\delta$.

For this, let
\[\alpha_*=\argmin_{\alpha \in \R} \Big(\max_{k=1}^{K-1} |v_k-k\alpha|_\cA\Big),
\qquad \eps_k=|v_k-k\alpha_*|_\cA \text{ for } k=1,\ldots,K-1,
\qquad \eps=\max_{k=1}^{K-1} \eps_k.\]
Note that a minimizing $\alpha_*$ exists because the minimum may equivalently
be restricted to the compact domain $[-\pi,\pi]$.
The induction hypothesis implies $\eps \leq \delta$.
By definition of $|\cdot|_\cA$, there exists $j_k \in \Z$
for each $k=1,\ldots,K-1$ such that
\[\eps_k=|v_k-k\alpha_*+2\pi j_k|.\]
Furthermore, we claim that there must exist two indices
$k,l \in \{1,\ldots,K-1\}$ for which
\[\eps=v_k-k\alpha_*+2\pi j_k \quad \text{ and } \quad
{-}\eps=v_l-l\alpha_*+2\pi j_l.\]
This is because 
\[\Big\{k \in \{1,\ldots,K-1\}:\eps_k=\eps\Big\}=\Big\{k \in
\{1,\ldots,K-1\}:|v_k-k\alpha_*+2\pi j_k|=\eps\Big\}\]
is non-empty by definition of $\eps$. If $v_k-k\alpha_*+2\pi j_k=\eps$
for every $k$ belonging to this set, then we may
decrease the value of $\max_{k=1}^{K-1} |v_k-k\alpha_*|_\cA$ by slightly
increasing $\alpha_*$, which contradicts the optimality of $\alpha_*$.
Similarly if $v_k-k\alpha_*+2\pi j_k=-\eps$ for all $k$ in this set, then
we may decrease $\max_{k=1}^{K-1} |v_k-k\alpha_*|_\cA$ by slightly
decreasing $\alpha_*$, again contradicting the optimality of $\alpha_*$.
Thus the claimed indices $k,l$ exist.

Then, for this index $k \in \{1,\ldots,K-1\}$, we have
\[v_k-k\alpha_* \in 2\pi \Z+\eps,
\qquad v_{K-k}-(K-k)\alpha_* \in 2\pi \Z+[-\eps,\eps],
\qquad v_K-v_k-v_{K-k} \in 2\pi \Z+[-\delta,\delta].\]
Adding these three conditions and applying $\eps \leq \delta$,
\[v_K-K\alpha_* \in 2\pi \Z+[-\delta,2\eps+\delta]
\subseteq 2\pi \Z+[-\delta,3\delta].\]
Similarly, for this index $l \in \{1,\ldots,K-1\}$, we have
\[v_l-l\alpha_* \in 2\pi \Z-\eps,
\qquad v_{K-l}-(K-l)\alpha_* \in 2\pi \Z+[-\eps,\eps],
\qquad v_K-v_l-v_{K-l} \in 2\pi \Z+[-\delta,\delta].\]
Then adding these conditions, also
\[v_K-K\alpha_* \in 2\pi \Z+[-2\eps-\delta,\delta]
\subseteq 2\pi \Z+[-3\delta,\delta].\]
Since $3\delta<\pi$ strictly, the above two conditions combine to show that
$v_K-K\alpha_* \in 2\pi \Z+[-\delta,\delta]$,
i.e.\ $|v_K-K\alpha_*|_\cA \leq \delta$. This completes the induction.
\end{proof}

\begin{proof}[\revise{Proof of Theorem \ref{thm:oracleMoM},
$\hat{\theta}=\hat{\theta}^\opt$}]
Let $\cE$ be the event where (\ref{eq:linftyPhibound}) holds. Note that this
is exactly the event where $|\hat{\Phi}_{k,l}^\oracle(\phi)-\Phi_{k,l}|<\pi/12$ for all
$(k,l) \in \cI$. Then by Lemma \ref{lemma:Phikltail} and a union bound,
\[\P[\cE^c] \leq CK^2\,e^{-c(\frac{N\rlower^6}{\sigma^6} \wedge
\frac{N^{2/3}\rlower^2}{\sigma^2})} \leq \frac{C'}{K}
\left(\frac{\sigma^2}{N\rlower^2}+\frac{\sigma^6}{N\rlower^6}\right)\]
where the second inequality holds under the given condition for $N$
as argued in (\ref{eq:exptopoly}). Applying
Corollaries \ref{cor:optmimicsoracle} and \ref{cor:MoMphiriskbound},
\revise{
\begin{align*}
&\E\left[\inf_{\alpha \in \R}
\sum_{k=1}^K r_k^2 |\hat\phi_k^\opt-\phi_k+k\alpha|_\cA^2\right]\\
&=\E\left[\1\{\cE\}\inf_{\alpha \in \R}
\sum_{k=1}^K r_k^2 |\hat\phi_k^\opt-\phi_k+k\alpha|_\cA^2\right]
+\E\left[\1\{\cE^c\}\inf_{\alpha \in \R}
\sum_{k=1}^K r_k^2 |\hat\phi_k^\opt-\phi_k+k\alpha|_\cA^2\right]\\
&\leq \E\left[\inf_{\alpha \in \R}
\sum_{k=1}^K r_k^2|\hat\phi_k^\oracle(\phi')-\phi_k'+k\alpha|_\cA^2\right]
+C\|\theta^*\|^2 \cdot \P[\cE^c]
\leq \frac{C'\|\theta^*\|^2}{K}\left(\frac{K\sigma^2}{N\rlower^2}
+\frac{\sigma^6}{N\rlower^6}\right).
\end{align*}}
This is (up to a universal constant) the same risk bound as established for
the oracle estimator itself in Corollary \ref{cor:MoMphiriskbound}. The
remainder of the proof is then the same as that of Theorem \ref{thm:oracleMoM}
for $\hat{\theta}=\hat{\theta}^\oracle$.
\end{proof}

\subsection{Estimation of $\Phi_{k,l}$ by frequency
marching}\label{appendix:freqmarching}

We describe in this section 
an alternative frequency marching method for mimicking the oracle
estimator, which is more explicit and computationally efficient but requires a
larger sample size $N$ to succeed. Let
\[\tilde{\phi}_1=0\]
and, for each $k=2,\ldots,K$, set
\[\tilde{\phi}_k=\Arg \hat{B}_{1,k-1}+\tilde{\phi}_{k-1} \bmod 2\pi.\]
This defines a vector $\tilde\phi \in [-\pi,\pi)^K$, which we use in place of
(\ref{eq:tildephidef}). Then, as in Section \ref{sec:MoMopt}, define
$\tilde\Phi_{k,l}=\tilde\phi_{k+l}-\tilde\phi_k-\tilde\phi_l$ with arithmetic
carried out over $\R$, and choose
$\hat{\Phi}_{k,l}^\fm \in [\tilde\Phi_{k,l}-\pi,\tilde\Phi_{k,l}+\pi)$ as the
unique version of the phase of $\hat{B}_{k,l}$ belonging to this range. Finally,
let $\hat\phi^\fm$ be the resulting least-squares estimate of $\phi$ in
(\ref{eq:leastsquares}), and let $\hat\theta^\fm$ be the resulting estimate of
$\theta$. Again, this procedure uses the frequency-marching estimate
$\tilde{\phi}$ only as a pilot estimate to resolve the phase ambiguity of the
estimated bispectrum, which is then inverted using a least-squares approach.

The following lemma is analogous to Corollary \ref{cor:optmimicsoracle}, but
requires an improvement for the error of $\Arg \hat{B}_{k,l}$ by a factor of
$1/K$.

\begin{Lemma}\label{lemma:fmmimicsoracle}
Suppose
\begin{equation}\label{eq:linftyfm}
|\Arg \hat{B}_{k,l}-(\phi_{k+l}-\phi_k-\phi_l)|_\cA<\pi/(6K) \text{ for
every } (k,l) \in \cI.
\end{equation}
Then there exists $\phi'$ equivalent to $\phi$ such that
$\hat\Phi^\fm=\hat\Phi^\oracle(\phi')$.
\end{Lemma}
\begin{proof}
For each $k=2,\ldots,K$, by the definition of $\tilde\phi_k$ and
the triangle inequality,
\[\big|\tilde\phi_k-\phi_k+k\phi_1\big|_\cA
\leq \big|\Arg \hat{B}_{1,k-1}-(\phi_k-\phi_1-\phi_{k-1})\big|_\cA
+\big|\tilde\phi_{k-1}-\phi_{k-1}+(k-1)\phi_1\big|_\cA.\]
Under the given condition, recursively applying this bound and using
$\tilde\phi_1-\phi_1+\phi_1=0$ for $k=1$,
\[\big|\tilde\phi_k-\phi_k+k\alpha\big|_\cA \leq
\frac{\pi(k-1)}{6K}<\frac{\pi}{6} \text{ for } \alpha=\phi_1 \text{ and all }
k=1,\ldots,K.\]
This means there exists $\phi'$ equivalent to $\phi$ for which
$|\tilde\phi_k-\phi_k'|<\pi/6$ for all
$k=1,\ldots,K$, and the remainder of the argument is the same as in Corollary
\ref{cor:optmimicsoracle}.
\end{proof}

The following guarantee is then analogous to Theorem \ref{thm:oracleMoM},
now describing the estimator $\hat\theta^\fm$ under a requirement for $N$ that
is larger by a factor of $K^2$.

\begin{Proposition}
Suppose $r_k \geq \rlower$ for each $k=1,\ldots,K$.
There exist universal constants $C,C_0>0$ such that if
\revise{$N \geq C_0(\frac{K^2\sigma^6}{\rlower^6}\log K+\frac{K\sigma^3}{\rlower^3}
(\log K)^{3/2})$}, then the guarantee (\ref{eq:oracleMoMbound}) holds also
for $\hat{\theta}^\fm$.
\end{Proposition}
\begin{proof}
Let $\cE$ be the event where (\ref{eq:linftyfm}) holds. This is exactly the
event where
$|\hat{\Phi}_{k,l}^\oracle(\phi)-\Phi_{k,l}|<\pi/(6K)$ for all $(k,l) \in
\cI$. Then by Lemma \ref{lemma:Phikltail} and a union bound,
\[\P[\cE^c] \leq
CK^2\,e^{-c(\frac{N\rlower^2}{K^2\sigma^2} \wedge
\frac{N\rlower^6}{K^2\sigma^6}\wedge \frac{N^{2/3}\rlower^2}{K^{2/3}\sigma^2})}
.\]
Under the given condition for $N$, an argument similar to (\ref{eq:exptopoly})
shows that this implies
\[\P[\cE^c] \leq \frac{C'}{K}\left(\frac{\sigma^2}{N\rlower^2}
+\frac{\sigma^6}{N\rlower^6}\right),\]
and the remainder of the proof is the same as that of Theorem \ref{thm:oracleMoM}.
\end{proof}

\section{Proofs for maximum likelihood estimation in low
noise}\label{appendix:MLE}

\subsection{KL divergence and tail bound for the MLE}

The following lemma bounds the Gaussian process $\langle \eps,\,g(\alpha) \cdot
\theta \rangle$ which appears in (\ref{eq:llexpanded}).
\begin{Lemma}
\label{lemma-infor-1}
Let $\eps \sim \mathcal{N}(0,I_{2K})$. For a universal constant $C>0$,
any $\theta \in \R^{2K}$, and any $s,t>0$,
\begin{align}
\P\left[\sup_{\alpha \in \cA} |\langle \eps, g(\alpha) \cdot \theta
\rangle|>t \text{ and } \|\eps\| \leq s\right] &\leq
\frac{8\pi\|\theta\|Ks}{t}\cdot e^{-\frac{t^2}{8\|\theta\|^2}}
\label{eq:GPtail}\\
\E\left[\sup_{\alpha \in \cA} |\langle \eps, g(\alpha) \cdot \theta
\rangle| \right] &\leq C\|\theta\|\sqrt{\log K}.\label{eq:GPexpectation}
\end{align}
\end{Lemma}
\begin{proof}
For each fixed $\alpha\in \cA$, we have $\langle \varepsilon,
g(\alpha)\cdot\theta\rangle \sim \mathcal N(0, \|\theta\|^2)$. Thus by a
Gaussian tail bound,
\[\P[|\langle \eps,g(\alpha) \cdot \theta \rangle|>t/2]
\leq 2e^{-\frac{t^2}{8\|\theta\|^2}}.\]
We set $\delta=t/(2\|\theta\|Ks)$ and
take $N_\delta \subset \cA$ as a $\delta$-net of $\cA=[-\pi,\pi)$ in the
metric $|\cdot|_\cA$, having cardinality
\[|N_\delta|=\frac{2\pi}{\delta}=\frac{4\pi\|\theta\|Ks}{t}.\]
For any $\alpha, \alpha'\in \cA$ such that $|\alpha-\alpha'|_\cA \leq \delta$,
from the definition (\ref{eq:thetarotation}) of the diagonal blocks of
$g(\alpha)$, we have
$\|g(\alpha)-g(\alpha')\|_{\text{op}} \leq K\delta$. Thus, on the event
$\{\|\eps\| \leq s\}$,
\[\big|\langle \eps,g(\alpha) \cdot \theta \rangle
-\langle \eps,g(\alpha') \cdot \theta \rangle\big|
\leq K\delta s\|\theta\|=t/2.\]
So
\[\P\left[\sup_{\alpha \in \cA} |\langle \eps,g(\alpha) \cdot \theta \rangle|
>t \text{ and } \|\eps\| \leq s \right]
\leq \P\left[\sup_{\alpha \in N_\delta}
|\langle \eps,g(\alpha) \cdot \theta \rangle|>t/2\right]
\leq |N_\delta| \cdot 2e^{-\frac{t^2}{8\|\theta\|^2}}\]
which yields (\ref{eq:GPtail}). Applying (\ref{eq:GPtail}) with $s=\sqrt{4K}$
and integrating from $t=4\|\theta\|\sqrt{\log K}$ to $t=\infty$,
\begin{align*}
&\E\left[\sup_{\alpha \in \cA} |\langle\varepsilon,
g(\alpha)\cdot\theta\rangle|\cdot \mathbf{1}\{\|\eps\| \leq \sqrt{4K}\}\right]\\
&\leq\E\left[4\|\theta\|\sqrt{\log K}
+\int_{4\|\theta\|\sqrt{\log K}}^{\infty}\mathbf{1}
\left\{\sup_{\alpha \in \cA} |\langle\varepsilon, g(\alpha)\cdot\theta\rangle|
>t \text{ and } \|\eps\| \leq \sqrt{4K} \right\}dt\right]\\
&\leq 4\|\theta\|\sqrt{\log K}+\int_{4\|\theta\|\sqrt{\log K}}^\infty
\frac{16\pi \|\theta\| K^{3/2}}{t} e^{-t^2/8\|\theta\|^2}dt\\
&=4\|\theta\|\sqrt{\log K}+16\pi\|\theta\|K^{3/2}\int_{4\sqrt{\log K}}^\infty
e^{-t^2/8}dt \leq C\|\theta\|\sqrt{\log K}
\end{align*}
for a universal constant $C>0$ and any $K \geq 2$. Applying a chi-squared tail
bound, we have also
\begin{align*}
\E\left[\sup_{\alpha \in \cA} |\langle \eps,g(\alpha) \cdot \theta \rangle|
\cdot \mathbf{1}\{\|\eps\| \geq \sqrt{4K}\}\right]
&\leq \|\theta\| \cdot \E\Big[\|\eps\| \cdot
\mathbf{1}\{\|\eps\| \geq \sqrt{4K}\} \Big]\\
&\leq \|\theta\| \cdot \E[\|\eps\|^2]^{1/2}
\P[\|\eps\| \geq \sqrt{4K}]^{1/2}
\leq \|\theta\| \cdot \sqrt{2K} \cdot e^{-cK}
\end{align*}
for a universal constant $c>0$.
Combining the above gives (\ref{eq:GPexpectation}).
\end{proof}

The next lemma formalizes the statement (\ref{eq:alpha0Taylor}) obtained by a
Taylor expansion around $\alpha=0$.

\begin{Lemma}
\label{lemma-infor-2}
\revise{Suppose Assumption \ref{assump:gen} holds.
Fix any constant $\delta_0 \in [0,3c_\gen/8]$.}
Then there are constants $C,c>0$ depending only on $c_\gen$ (and
independent of $\delta_0$) such that for all
$\alpha \in [-\frac{\delta_0}{K},\frac{\delta_0}{K}]$,
\begin{equation}\label{eq:maintermlower}
\revise{cK^2\|\theta^*\|^2\alpha^2 \leq
\|\theta^*\|^2-\langle \theta^*,g(\alpha) \cdot \theta^* \rangle
\leq CK^2\|\theta^*\|^2\alpha^2.}
\end{equation}
Furthermore, there is a constant $\iota>0$
depending only on $c_\gen,\delta_0$ such that for all
$\alpha \in [-\pi,\pi) \setminus [-\frac{\delta_0}{K},\frac{\delta_0}{K}]$,
\begin{equation}\label{eq:maintermupper}
\langle \theta^*,g(\alpha) \cdot \theta^* \rangle
\leq (1-\iota)\|\theta^*\|^2.
\end{equation}
\end{Lemma}
\begin{proof}
We write as shorthand
$r_k=r_k(\theta^*)$. From (\ref{eq:rphirotation}), observe that
\begin{equation}\label{eq:maintermtmp}
\langle \theta^*,\,g(\alpha) \cdot \theta^* \rangle
=\sum_{k=1}^K r_k^2\cos k\alpha
=\|\theta^*\|^2-\sum_{k=1}^K r_k^2(1-\cos k\alpha).
\end{equation}
This is an even function of $\alpha$, so it suffices
to consider $\alpha \in [0,\pi]$.
Suppose first that $0\leq \alpha\leq \delta_0/K$. By Taylor expansion
around $\alpha=0$,
\[\sum_{k=1}^{K}r_k^2(1-\cos k\alpha)=\sum_{k=1}^{K}
r_k^2 \cdot \frac{k^2\alpha^2}{2}+r(\alpha),\]
where $|r(\alpha)| \leq \sum_{k=1}^{K} r_k^2(k\alpha)^3/6$. \revise{Observe that
\begin{equation}\label{eq:krsumbounds}
\sum_{k=1}^K k^3r_k^2 \leq K^3\|\theta^*\|^2,
\qquad \sum_{k=1}^K k^2r_k^2 \leq K^2\|\theta^*\|^2,
\qquad \sum_{k=1}^K k^2r_k^2 \geq
(K/2)^2 \sum_{k=\lceil K/2 \rceil}^K r_k^2
\geq \frac{c_\gen K^2\|\theta^*\|^2}{4},
\end{equation}
the last inequality applying Assumption \ref{assump:gen}.
Then for $0 \leq \alpha \leq \frac{\delta_0}{K} \leq
\frac{3c_\gen}{8K}$, applying the first and third
of these bounds,} we have $|r(\alpha)| \leq \sum_{k=1}^K
r_k^2 (k\alpha)^2/4$. Applying this and the above Taylor expansion
to (\ref{eq:maintermtmp}) gives
\[\sum_{k=1}^K r_k^2 \cdot \frac{k^2\alpha^2}{4} \leq
\|\theta^*\|^2-\langle \theta^*,g(\alpha) \cdot \theta^* \rangle
\leq \sum_{k=1}^K r_k^2 \cdot \frac{3k^2\alpha^2}{4},\]
which implies (\ref{eq:maintermlower}) \revise{by the second and third bounds of
(\ref{eq:krsumbounds}).}

Now consider $\delta_0/K<\alpha \leq \pi$.
In the sequence $(\cos\alpha, \cos2\alpha, \ldots, \cos K\alpha)$, we claim that
there are at most $\lceil K/2 \rceil$
items belonging to the interval $[\cos L, 1]$, where $L=\min(\delta_0/2,\pi/8)$:
\begin{itemize}
\item If $\frac{\delta_0}{K}<\alpha<\frac{\pi}{K}$,
then $\alpha, 2\alpha, \ldots,K\alpha\in (0,\pi)$. So $\cos(k\alpha) \in
[\cos L,1]$ implies that $k\alpha\in [0,L]$, and the number of such items is
at most $L/\alpha \leq (\delta_0/2)/(\delta_0/K)=K/2$.
\item If $\frac{t\pi}{K} \leq \alpha < \frac{(t+1)\pi}{K}$ for some $1 \leq
t\leq \frac{K}{4}-1$, then $\alpha, 2\alpha, \ldots,K\alpha\in(0,(t+1)\pi)$. So
$\cos(k\alpha) \in [\cos L,1]$ implies that $k\alpha$ falls into one of
$t+1$ closed intervals of width $L$, and the number of such items is at most
\[(t+1)\cdot \left\lceil\frac{L}{\alpha}\right\rceil \leq
(t+1)\cdot\left(\frac{\pi/8}{t\pi/K}+1\right)=\frac{K}{8}
+\frac{K}{8t}+t+1 \leq \frac{K}{8}+\frac{K}{8}+\frac{K}{4}=\frac{K}{2}.\]
\item If $\frac{\pi}{4}<\alpha \leq \pi$, then any two consecutive items $\cos
k\alpha$ and $\cos (k+1)\alpha$ cannot both belong to $[\cos L, 1]$,
since $\alpha>\frac{\pi}{4} \geq 2L$. Therefore, the number of items would not
exceed $\lceil K/2 \rceil$.
\end{itemize}
Denoting $B=\{k:\cos k\alpha\notin[\cos L, 1]\}$, we then have
$|B| \geq \lfloor K/2 \rfloor$ and $1-\cos L \geq c$ a small constant depending
on $\delta_0$, so
\[\sum_{k=1}^K r_k^2(1-\cos k\alpha) \geq \sum_{k \in B} r_k^2(1-\cos L)
\geq \revise{c \cdot c_\gen\|\theta^*\|^2} \geq \iota\|\theta^*\|^2\]
for a constant $\iota>0$ depending only on $c_\gen,\delta_0$.
Applying this to (\ref{eq:maintermtmp}) gives (\ref{eq:maintermupper}).
\end{proof}

\begin{proof}[Proof of Lemma \ref{lemma-KL-lower}]
Recall the form (\ref{eqn-KL-lower-1}) for the KL divergence.
For $\mathrm{II}$, upper bounding the average over $\alpha$ by the maximum,
\begin{align}
\mathrm{II} \leq \E\log \sup_{\alpha \in \cA}
\exp\left(\frac{\langle \theta^*+\sigma \eps,
g(\alpha)\cdot\theta\rangle}{\sigma^2}\right)
&\leq \frac{\sup_{\alpha \in \cA} \langle\theta^*,
g(\alpha)\cdot\theta\rangle}{\sigma^2} + \mathbb{E}
\left[ \frac{\sup_{\alpha \in \cA}
\langle\varepsilon, g(\alpha)\cdot\theta\rangle}{\sigma}\right]\nonumber\\
&\leq \frac{\sup_{\alpha \in \cA} \langle \theta^*, g(\alpha) \cdot \theta
\rangle}{\sigma^2}+\frac{C\|\theta\|}{\sigma}\cdot\sqrt{\log K},\label{eqn-KL-lower-2-1}
\end{align}
where the last inequality applies (\ref{eq:GPexpectation}) from Lemma
\ref{lemma-infor-1}. Similarly, to lower bound $\mathrm{I}$, let us set
\revise{$\delta_0=3c_\gen/8$} and apply
\begin{align*}
\mathrm{I}& \geq \E \left[\log \frac{1}{2\pi}\int_{-\pi}^\pi
\exp\left(\frac{\langle \theta^*,g(\alpha) \cdot \theta^*}{\sigma^2}\right)
d\alpha \cdot \inf_{\alpha} \exp\left(\frac{\eps,g(\alpha) \cdot
\theta^*}{\sigma}\right)\right]\\
&= \log\frac{1}{2\pi}\int_{-\pi}^{\pi} \exp\left(\frac{\langle \theta^*, g(\alpha)\cdot\theta^*\rangle}{\sigma^2}\right)d\alpha
-\mathbb{E}\left[\frac{\sup_{\alpha \in \cA}
\langle\varepsilon, g(\alpha)\cdot\theta^*\rangle}{\sigma}\right]\\
&\geq \log\frac{1}{2\pi}\int_{-\delta_0/K}^{\delta_0/K} \exp\left(\frac{\langle \theta^*, g(\alpha)\cdot\theta^*\rangle}{\sigma^2}\right)d\alpha
-\frac{C\|\theta^*\|}{\sigma} \cdot \sqrt{\log K}
\end{align*}
Applying the upper bound of (\ref{eq:maintermlower}) from
Lemma \ref{lemma-infor-2}, we have for a constant $C>0$ that
\begin{align*}
\int_{-\delta_0/K}^{\delta_0/K} \exp\left(\frac{\langle \theta^*,
g(\alpha)\cdot\theta^*\rangle}{\sigma^2}\right)d\alpha
&\geq
\exp\left(\frac{\|\theta^*\|^2}{\sigma^2}\right)\int_{-\delta_0/K}^{\delta_0/K}
\revise{\exp\left(-\frac{CK^2\|\theta^*\|^2}{\sigma^2}\alpha^2\right)}d\alpha\\
&=\exp\left(\frac{\|\theta^*\|^2}{\sigma^2}\right)\revise{\left(\frac{2CK^2\|\theta^*\|^2}{\sigma^2}\right)^{-1/2}}\cdot\sqrt{2\pi}
\left(1-2\widetilde{\Phi}\left(\revise{\sqrt{\frac{2CK^2\|\theta^*\|^2}{\sigma^2}}} \cdot \frac{\delta_0}{K}\right)
\right)
\end{align*}
where $\widetilde{\Phi}(x)=\int_x^\infty\frac{1}{\sqrt{2\pi}}e^{-t^2/2}dt$ is
the right tail probability of the standard Gaussian law. Applying the given
condition \revise{$\sigma^2 \leq \|\theta^*\|^2$}, the input to $\widetilde{\Phi}$ is bounded
below by a positive constant. Then the value for $\widetilde{\Phi}$ is bounded
away from $1/2$, so for a constant $C_2>0$,
\begin{equation}\label{eq:gaussianintegral}
\frac{1}{2\pi}\int_{-\delta_0/K}^{\delta_0/K} \exp\left(\frac{\langle \theta^*,
g(\alpha)\cdot\theta^*\rangle}{\sigma^2}\right)d\alpha
\geq \exp\left(\frac{\|\theta^*\|^2}{\sigma^2}\right)
\cdot \left(\frac{C_2K^2\|\theta^*\|^2}{\sigma^2}\right)^{-1/2}.
\end{equation}
Thus
\begin{equation}\label{eqn-KL-lower-3-3}
\mathrm{I} \geq
\frac{\|\theta^*\|^2}{\sigma^2}-\frac{1}{2}\log\left(\frac{C_2K^2\|\theta^*\|^2}{\sigma^2}\right)-\frac{C\|\theta^*\|}{\sigma}\cdot\sqrt{\log K}.
\end{equation}
Combining (\ref{eqn-KL-lower-1}), (\ref{eqn-KL-lower-2-1}), and
(\ref{eqn-KL-lower-3-3}) and applying
\[\min_{\alpha \in \cA} \|\theta^*-g(\alpha) \cdot \theta\|^2
=\|\theta^*\|^2+\|\theta\|^2-2 \sup_{\alpha \in \cA} \langle \theta^*,
g(\alpha) \cdot \theta \rangle\]
yields the lemma.
\end{proof}

The following lemma establishes concentration of $R_N(\theta)$ around
its mean $R(\theta)$, uniformly over bounded domains of $\theta$.

\begin{Lemma}\label{lemma-bounded-norm-start}
For a universal constant $c>0$, any $M>0$, and any $t>0$,
\begin{align}
&\P\left[\sup_{\theta:\|\theta\| \leq M} |R_N(\theta)-R(\theta)|>4t\right]
\nonumber\\
&\hspace{1in}\leq
2\left(1+\tfrac{2M\sqrt{2\|\theta^*\|^2+(4K+4t)\sigma^2}}{t\sigma^2}\right)^{2K}
e^{-\frac{cN\sigma^2t^2}{M^2}}+4e^{-cN(t \wedge \frac{t^2}{K} \wedge
\frac{t^2\sigma^2}{\|\theta^*\|^2})}.\label{eq:unifconc}
\end{align}
\end{Lemma}
\begin{proof}
Recalling the form of $R_N(\theta)$ from (\ref{eq:llexpanded}), we have
\[R_N(\theta)=\mathrm{I}-\mathrm{II}(\theta)+\operatorname{const}(\theta)\]
where
\begin{align*}
\mathrm{I}&=\frac{1}{N}\sum_{m=1}^N \frac{\|\theta^*+\sigma
\eps^{(m)}\|^2}{2\sigma^2},\\
\mathrm{II}(\theta)&=\frac{1}{N}\sum_{m=1}^N f(\eps^{(m)},\theta)
:=\frac{1}{N}\sum_{m=1}^N \log \frac{1}{2\pi} \int_{-\pi}^\pi
\exp\left(\frac{\langle \theta^*+\sigma \eps^{(m)},
g(\alpha) \cdot \theta \rangle}{\sigma^2}\right) d\alpha,
\end{align*}
and $\operatorname{const}(\theta)$ is a term not depending on the randomness
$\{\eps^{(m)}\}$. We analyze separately the concentration of the terms
$\mathrm{I}$ and $\mathrm{II}(\theta)$.

For the given value $t>0$, define the event
$\cE=\{|\mathrm{I}-\E[\mathrm{I}]|<2t\}$. We have
\[\frac{\|\theta^*+\sigma \eps^{(m)}\|^2}{2\sigma^2}=\frac{\|\theta^*\|^2}{2\sigma^2}+\frac{\langle \theta^*, \varepsilon^{(m)}\rangle}{\sigma}+\frac{\|\varepsilon^{(m)}\|^2}{2}.\]
Here the first term is deterministic, and the latter two terms satisfy
$N^{-1}\sum_{m=1}^N \langle \theta^*, \eps^{(m)} \rangle \sim
\mathcal{N}(0,\|\theta^*\|^2/N)$ and $\sum_{m=1}^N \|\eps^{(m)}\|^2 \sim
\chi_{2KN}^2$. Then by standard Gaussian and chi-squared tail bounds, for 
a universal constant $c>0$ and any $t>0$,
\[\P\left[\left|\frac{1}{N}\sum_{m=1}^N
\frac{\langle \theta^*,\eps^{(m)} \rangle}{\sigma}\right|
\geq t\right] \leq 2e^{-\frac{Nt^2\sigma^2}{2\|\theta^*\|^2}},
\qquad
\mathbb{P}\left[\left|\frac{1}{N}\sum_{m=1}^N\frac{\|\varepsilon^{(m)}\|^2}{2}-K\right| \geq t\right]
\leq 2e^{-cNK(\frac{t}{K} \wedge \frac{t^2}{K^2})}.\]
So
\begin{equation}\label{eq:Etprobbound}
\P[\cE^c]=\P\Big[|\mathrm{I}-\E[\mathrm{I}]| \geq 2t\Big]
\leq 2e^{-\frac{Nt^2\sigma^2}{2\|\theta^*\|^2}}
+2e^{-cNK(\frac{t}{K} \wedge \frac{t^2}{K^2})}
\leq 4e^{-c'N(t \wedge \frac{t^2}{K} \wedge
\frac{t^2\sigma^2}{\|\theta^*\|^2})}.
\end{equation}

On the event $\cE$, we have
$|R_N(\theta)-R(\theta)|
\leq 2t+|\mathrm{II}(\theta)-\E[\mathrm{II}(\theta)]|$ as well as
\begin{equation}\label{eq:ynormbound}
\frac{1}{N}\sum_{m=1}^N \|\theta^*+\sigma \eps^{(m)}\|^2
\leq \E[\|\theta^*+\sigma \eps^{(m)}\|^2]
+4t\sigma^2=\|\theta^*\|^2+(2K+4t)\sigma^2.
\end{equation}
Recalling the probability law $\cP_{\theta,\eps}$ from
(\ref{eq:Pthetaeps}), the $\eps$-gradient of the function
$f(\eps,\theta)$ defining $\mathrm{II}(\theta)$ is bounded as
\[\|\nabla_\eps f(\eps,\theta)\|=\left\|\frac{1}{\sigma}\mathbb{E}_{\alpha \sim
\cP_{\theta,\eps}} \left[g(\alpha)\cdot\theta\right]\right\|
\leq \frac{1}{\sigma}\mathbb{E}_{\alpha \sim \cP_{\theta,\eps}}
\left[\|g(\alpha)\cdot\theta\|\right]
=\frac{\|\theta\|}{\sigma}.\]
Thus $f(\eps,\theta)$ is $\frac{\|\theta\|}{\sigma}$-Lipschitz
in $\eps$. Then by Gaussian concentration of measure,
for universal constants $C,c>0$, we have that
$f(\eps,\theta)-\E[f(\eps,\theta)]$
is $\frac{C\|\theta\|}{\sigma}$-subgaussian, and Hoeffding's inequality yields
\[\P\left[|\mathrm{II}(\theta)-\E[\mathrm{II}(\theta)]|>t\right]
=\P\left[\left|\frac1N\sum_{m=1}^N
f(\eps^{(m)},\theta)-\E[f(\eps^{(m)},\theta)]\right|>t\right]\leq
2e^{-\frac{cN\sigma^2 t^2}{\|\theta\|^2}}.\]

Now to obtain uniform concentration over $\{\theta:\|\theta\| \leq M\}$,
set $\delta=t\sigma^2/\sqrt{2\|\theta^*\|^2+(4K+4t)\sigma^2}$,
and let $N_\delta$ be a $\delta$-net of
$\{\theta \in \R^{2K}:\|\theta\| \leq M\}$ having cardinality
\[|N_\delta| \leq \left(1+\frac{2M}{\delta}\right)^{2K}
=\left(1+\frac{2M\sqrt{2\|\theta^*\|^2+(4K+4t)\sigma^2}}{t\sigma^2}\right)^{2K}.\]
The $\theta$-gradient
of $\E[\mathrm{II}(\theta)]=\E[f(\eps^{(m)},\theta)]$ is bounded as
\begin{align*}
\|\nabla_\theta \E[\mathrm{II}(\theta)]\|
&=\left\|\frac{1}{\sigma^2}\E_{\eps \sim \mathcal{N}(0,I)}
\left[\E_{\alpha \sim \cP_{\theta,\eps}}[g(\alpha)^{-1}
(\theta^*+\sigma\eps)]\right]\right\|\\
&\leq \frac{1}{\sigma^2}\E[\|\theta^*+\sigma \eps\|]
\leq \frac{1}{\sigma^2}\sqrt{\E[\|\theta^*+\sigma \eps\|^2]}
=\frac{1}{\sigma^2}\sqrt{\|\theta^*\|^2+2K\sigma^2}.
\end{align*}
Similarly, on the event $\cE$, the $\theta$-gradient of $\mathrm{II}(\theta)$
without expectation is bounded as
\begin{align*}
\|\nabla_\theta \mathrm{II}(\theta)\|&=\left\|\frac{1}{\sigma^2}\cdot\frac{1}{N}\sum_{m=1}^{N}\mathbb{E}_{\alpha
\sim \cP_{\theta,\eps^{(m)}}} [g(\alpha)^{-1}(\theta^*+\sigma \eps^{(m)})]
\right\|\\
&\leq \frac{1}{\sigma^2} \cdot \frac{1}{N}\sum_{m=1}^N
\|\theta^*+\sigma\eps^{(m)}\|
\leq \frac{1}{\sigma^2}\sqrt{\frac{1}{N}\sum_{m=1}^N
\|\theta^*+\sigma \eps^{(m)}\|^2}
\leq \frac{1}{\sigma^2} \sqrt{\|\theta^*\|^2+(2K+4t)\sigma^2},
\end{align*}
the last inequality applying (\ref{eq:ynormbound}). Therefore
$\mathrm{II}(\theta)-\E[\mathrm{II}(\theta)]$ is Lipschitz in $\theta$, with
Lipschitz constant at most
\[\frac{1}{\sigma^2}\sqrt{\|\theta^*\|^2+2K\sigma^2}
+\frac{1}{\sigma^2} \sqrt{\|\theta^*\|^2+(2K+4t)\sigma^2}
\leq \frac{1}{\sigma^2}\sqrt{2\|\theta^*\|^2+(4K+4t)\sigma^2}=\frac{t}{\delta}.\]
Then
\begin{align*}
\P\left[\sup_{\theta:\|\theta\| \leq M}
|R_N(\theta)-R(\theta)|>4t \text{ and } \cE\right]
&\leq \P\left[\sup_{\theta:\|\theta\| \leq M}
|\mathrm{II}(\theta)-\E[\mathrm{II}(\theta)]|>2t \text{ and } \cE\right]\\
&\leq \P\left[\sup_{\theta \in N_\delta}
|\mathrm{II}(\theta)-\E[\mathrm{II}(\theta)]|>t\right]
\leq 2|N_\delta|e^{-\frac{cN\sigma^2t^2}{M^2}}.
\end{align*}
Combining this with (\ref{eq:Etprobbound}) gives (\ref{eq:unifconc}).
\end{proof}

\begin{proof}[Proof of Lemma \ref{lemma-highprob-bound}]
For the given value of $\delta_1$ and each integer $n \geq 1$, define
\[\Gamma_n=\left\{\theta:n\delta_1\|\theta^*\|\leq \min_{\alpha \in \cA}
\|\theta^*-g(\alpha) \cdot \theta\|<(n+1)\delta_1\|\theta^*\|\right\}
\subset \R^{2K}.\]
Observe that $\|\theta\| \leq [1+(n+1)\delta_1]\|\theta^*\|$ for $\theta \in
\Gamma_n$, so Lemma \ref{lemma-KL-lower} implies
\[D_{\KL}(p_{\theta^*}\| p_\theta)
\geq \frac{n^2\delta_1^2\|\theta^*\|^2}{2\sigma^2}
-\frac{1}{2}\log \left(\frac{C_2K^2\|\theta^*\|^2}{\sigma^2}\right)
-\frac{[2+(n+1)\delta_1]C_3\|\theta^*\|}{\sigma}\sqrt{\log K}
\text{ for all } \theta \in \Gamma_n.\]
\revise{Then, under the
given assumption that $\frac{\|\theta^*\|^2}{\sigma^2} \geq C_1\log K$ for a
sufficiently large constant $C_1>0$ (depending on $c_\gen,\delta_1$),
setting
\[t_n=c_0n^2\|\theta^*\|^2/\sigma^2\]}
for a sufficiently small constant $c_0>0$, the above implies that
\[D_{\KL}(p_{\theta^*}\| p_\theta) \geq 10t_n
\text{ for all } \theta \in \Gamma_n.\]

Applying (\ref{eq:unifconc}) with $t=t_n$ 
and $M=[1+(n+1)\delta_1]\|\theta^*\|$ gives,
for some constants
$C,C',c,c'>0$ depending on $\delta_1$,
\begin{align*}
&\P\left[\sup_{\theta:\|\theta\| \leq [1+(n+1)\delta_1]\|\theta^*\|}
|R_N(\theta)-R(\theta)|>4t_n\right]\\
&\leq \revise{2\left(1+C\sqrt{1+\tfrac{K\sigma^2}{n^2\|\theta^*\|^2}}\right)^{2K}
e^{-cn^2N \cdot \frac{\|\theta^*\|^2}{\sigma^2}}
+4e^{-cn^2N \cdot \frac{\|\theta^*\|^2}{\sigma^2} (1 \wedge
\frac{n^2\|\theta^*\|^2}{K\sigma^2})}}\\
&\leq (C'K)^K e^{-c'nN \log K}+e^{-c'nN(\log K)^2/K}
\end{align*}
where the last line applies $n \geq 1$ and $\frac{\|\theta^*\|^2}{\sigma^2} \geq C_1\log K$ to simplify the bound.
On the event where $|R_N(\theta)-R(\theta)| \leq 4t_n$ and
$|R_N(\theta^*)-R(\theta^*)| \leq 4t_n$, since
$D_{\KL}(p_{\theta^*}\|p_\theta)=R(\theta)-R(\theta^*) \geq 10t_n$,
we must then have
$R_N(\theta)-R_N(\theta^*) \geq 2t_n>0$ so that $\theta$ is not the MLE. Thus,
\begin{equation}\label{eq:Gammanprob}
\P\left[\hat{\theta}^{\text{MLE}} \in \Gamma_n\right]
\leq (C'K)^K e^{-c'nN \log K}+e^{-c'nN(\log K)^2/K}.
\end{equation}
Summing over all $n \geq 1$ and recalling our choice of rotation for
$\hat{\theta}^{\text{MLE}}$ that satisfies (\ref{eq:MLEversion}),
\begin{align*}
\P\left[\|\hat{\theta}^{\text{MLE}}-\theta^*\| \geq \delta_1\|\theta^*\|\right]
&\leq \sum_{n=1}^\infty
\P\left[\hat{\theta}^{\text{MLE}} \in \Gamma_n\right]
\leq (C'K)^K \sum_{n=1}^\infty
e^{-c'nN \log K}+\sum_{n=1}^\infty e^{-c'nN(\log K)^2/K}.
\end{align*}
Under the given assumption $N \geq C_0K$
for a sufficiently large constant $C_0>0$, both exponents $c'N\log K$ and
$c'N(\log K)^2/K$ are bounded below by a constant. Then summing these
geometric series gives, for some modified constants $C,c,c'>0$,
\[\P\left[\|\hat{\theta}^{\text{MLE}}-\theta^*\| \geq \delta_1\|\theta^*\|\right]
\leq (CK)^K e^{-cN\log K}+e^{-cN(\log K)^2/K}
\leq e^{-c'N(\log K)^2/K}\]
where  the second inequality holds again under the assumption $N \geq C_0K$.
This shows (\ref{eq:MLElocalization}).

To show (\ref{eq:MLE4thmoment}), we apply $\|\theta\| \leq
[1+(n+1)\delta_1]\|\theta^*\|$ for $\theta \in \Gamma_n$ and
(\ref{eq:Gammanprob}) to get, for some constants $C,C',c'>0$ depending on
$\delta_1$,
\begin{align*}
\E[\|\hat{\theta}^{\text{MLE}}\|^4]
&\leq [(1+\delta_1)\|\theta^*\|]^4+\sum_{n=1}^\infty [(1+(n+1)\delta_1)
\|\theta^*\|]^4
\cdot \P[\hat{\theta}^{\text{MLE}} \in \Gamma_n]\\
&\leq C\|\theta^*\|^4\left(1+(C'K)^K\sum_{n=1}^\infty
n^4 e^{-c'nN\log K}+\sum_{n=1}^\infty n^4e^{-c'nN(\log K)^2/K}\right).
\end{align*}
For a sufficiently large constant $A>0$, we have $n^4e^{-An}<e^{-An/2}$ for
all $n \geq 1$. Hence, under the condition $N \geq C_0K$ for sufficiently
large $C_0>0$, we have
\begin{align*}
\E[\|\hat{\theta}^{\text{MLE}}\|^4]
&\leq C\|\theta^*\|^4\left(1+(C'K)^K\sum_{n=1}^\infty
e^{-(c'/2)nN\log K}+\sum_{n=1}^\infty e^{-(c'/2)nN(\log K)^2/K}\right)\\
&\leq C\|\theta^*\|^4\left(1+(C'K)^K e^{-c'' N\log K}
+e^{-c''N(\log K)^2/K}\right) \leq C'\|\theta^*\|^4.
\end{align*}
\end{proof}

\subsection{Lower bound for the information matrix}

Define the domain
\begin{equation}\label{eq:epsthetaalign}
\cF_1(\theta,\delta_1)=\left\{\eps:
\sup_{\alpha \in \cA} |\langle \eps,g(\alpha) \cdot \theta \rangle|
\leq \frac{\delta_1\|\theta^*\|^2}{\sigma}\right\} \subset \R^{2K}.
\end{equation}
The following deterministic lemma guarantees that the law
$\cP_{\theta,\eps}$ concentrates near 0 when
$\eps \in \cF_1(\theta,\delta_1)$.

\begin{Lemma}\label{lemma:angleconc}
\revise{Suppose Assumption \ref{assump:gen} holds.}
Fix any $\delta_0>0$. Then there exist constants $C_1,\delta_1>0$ depending
only on \revise{$c_\gen,\delta_0$ such that if
$\sigma^2 \leq \frac{\|\theta^*\|^2}{C_1\log K}$}, then the following holds:
For any $\theta \in \cB(\delta_1)$ and any (deterministic)
$\eps \in \cF_1(\theta,\delta_1)$,
\begin{equation}\label{eq:maxangle}
\sup_{\alpha \in [-\frac{\delta_0}{K},\frac{\delta_0}{K}]}
\langle \theta^*+\sigma \eps,\,g(\alpha) \cdot \theta \rangle
>\sup_{\alpha \in [-\pi,\pi) \setminus [-\frac{\delta_0}{K},\frac{\delta_0}{K}]}
\langle \theta^*+\sigma \eps,\,g(\alpha)\cdot \theta \rangle.
\end{equation}
Furthermore, for a constant $c>0$ depending only on $\clower,\cupper,\delta_0$,
\begin{equation}\label{eq:angleconc}
\P_{\alpha \sim \cP_{\theta,\eps}}\Big[|\alpha|_\cA>\delta_0/K\Big]
\leq \revise{e^{-c\|\theta^*\|^2/\sigma^2}}.
\end{equation}
\end{Lemma}
\begin{proof}
Define $I_1=[-\frac{\delta_0}{K},\frac{\delta_0}{K}]$ and
$I_2=[-\pi,\pi) \setminus [-\frac{\delta_0}{K},\frac{\delta_0}{K}]$.
Let us write
\[\langle \theta^*+\sigma \eps, g(\alpha)\cdot\theta\rangle
=\langle \theta^*,g(\alpha)\theta^* \rangle
+\langle \theta^*,g(\alpha)(\theta-\theta^*) \rangle
+\sigma \langle \eps,g(\alpha)\theta \rangle.\]
The conditions $\theta \in \cB(\delta_1)$ and
$\eps \in \cF_1(\theta,\delta_1)$ show for the second and third terms
\begin{equation}\label{eq:exp23}
|\langle \theta^*,g(\alpha)(\theta-\theta^*) \rangle|
\leq \|\theta^*\|\cdot \|\theta-\theta^*\| \leq \delta_1 \|\theta^*\|^2,
\qquad |\sigma\langle \eps,g(\alpha)\theta \rangle|
\leq \delta_1 \|\theta^*\|^2.
\end{equation}
For the first term, Lemma \ref{lemma-infor-2} implies that for constants
$C>0$ and $\iota>0$,
\begin{equation}\label{eq:exp1}
\langle \theta^*,g(\alpha)\theta^* \rangle
\leq (1-\iota) \cdot \|\theta^*\|^2
\text{ if } \alpha \in I_2, \qquad
\langle \theta^*,g(\alpha)\theta^* \rangle
\geq \|\theta^*\|^2-\revise{CK^2\|\theta^*\|^2}\alpha^2 \text{ if } \alpha \in I_1.
\end{equation}
Then for all $\alpha \in I_2$, we have 
$\langle \theta^*+\sigma \eps, g(\alpha)\cdot\theta\rangle
\leq (1-\iota+2\delta_1)\|\theta^*\|^2$, while for
$\alpha=0 \in I_1$, we have
$\langle \theta^*+\sigma \eps, g(\alpha)\cdot\theta\rangle
\geq (1-2\delta_1)\|\theta^*\|^2$. Setting $\delta_1<\iota/4$, this shows
(\ref{eq:maxangle}).

To show (\ref{eq:angleconc}), we may correspondingly write
the density (\ref{eq:Pthetaeps}) for
the distribution $\cP_{\theta,\eps}$ as
\[\frac{d\cP_{\theta,\eps}(\alpha)}{d\alpha}
\propto \exp\left(\frac{\langle \theta^*+\sigma \eps,
g(\alpha)\cdot\theta\rangle}{\sigma^2}\right)=\exp\left(
\frac{\langle \theta^*,g(\alpha)\theta^* \rangle}{\sigma^2}
+\frac{\langle \theta^*,g(\alpha)(\theta-\theta^*) \rangle}{\sigma^2}
+\frac{\sigma \langle \eps,g(\alpha)\theta \rangle}{\sigma^2}\right).\]
Then
\begin{align*}
\int_{I_2}\exp\left(\frac{\langle \theta^*+\sigma \eps,
g(\alpha)\cdot\theta\rangle}{\sigma^2}\right)d\alpha
&\leq \exp\left(\frac{(1-\iota+2\delta_1)\|\theta^*\|^2}{\sigma^2}\right),\\
\int_{I_1}\exp\left(\frac{\langle \theta^*+\sigma \eps,
g(\alpha)\cdot\theta\rangle}{\sigma^2}\right)d\alpha
&\geq \exp\left(\frac{(1-2\delta_1)\|\theta^*\|^2}{\sigma^2}\right)
\cdot \int_{-\delta_0/K}^{\delta_0/K}
\revise{\exp\left(-\frac{CK^2\|\theta^*\|^2\alpha^2}{\sigma^2}\right)}d\alpha.
\end{align*}
Lower bounding this Gaussian integral using the same argument as
(\ref{eq:gaussianintegral}), for a constant $C'>0$,
\[\int_{-\delta_0/K}^{\delta_0/K}
\exp\left(-\frac{CK^2\|\theta^*\|^2\alpha^2}{\sigma^2} \right)d\alpha
\geq \left(\frac{C'K^2\|\theta^*\|^2}{\sigma^2}\right)^{-1/2}\sqrt{2\pi}.\]
Combining these bounds and choosing $\delta_1<\iota/4$,
for a constant $c>0$,
\[J:=\frac{\int_{I_1}\exp\left(\frac{\langle \theta^*+\sigma \eps,
g(\alpha)\cdot\theta\rangle}{\sigma^2}\right)d\alpha}
{\int_{I_2}\exp\left(\frac{\langle \theta^*+\sigma \eps,
g(\alpha)\cdot\theta\rangle}{\sigma^2}\right)d\alpha}
\geq \revise{\left(\frac{C'K^2\|\theta^*\|^2}{\sigma^2}\right)^{-1/2}
\sqrt{2\pi}\exp\left(\frac{c\|\theta^*\|^2}{\sigma^2}\right)
\geq \exp\left(\frac{c\|\theta^*\|^2}{2\sigma^2}\right)}\]
where the last inequality \revise{holds for
$\frac{\|\theta^*\|^2}{\sigma^2} \geq C_1\log K$ and a sufficiently large
constant $C_1>0$.}
Since $\P_{\alpha \sim \cP_{\theta,\eps}}[|\alpha|_\cA \geq \frac{\delta_0}{K}]
=1/(1+J)<1/J$, this shows (\ref{eq:angleconc}).
\end{proof}

Next, recall the complex representations
\[\tilde{\theta}=(\theta_1,\ldots,\theta_K) \in \C^K,
\qquad \tilde{v}=(v_1,\ldots,v_K) \in \C^K,
\qquad \tilde{\eps}=(\eps_1,\ldots,\eps_K) \in \C^K\]
and define $\cF_2(\theta,v,\delta_1) \subset \R^{2K}$ as the set of
vectors $\eps \in \R^{2K}$ that satisfy the following three conditions:
\revise{
\begin{align}
\sup_{\alpha \in \cA} \Big|\langle \eps,g(\alpha) \cdot v \rangle\Big|
&\leq \frac{\|\theta^*\|}{\sigma}\label{eq:epsvthetaalign1}\\
\sup_{\alpha,\alpha' \in [-\pi,\pi)}
\frac{1}{\alpha^2}\left|\Re \sum_{k=1}^K \overline{\eps_k}e^{ik\alpha'}
\Big(e^{ik\alpha}-1-ik\alpha\Big)\theta_k\right|
&\leq \frac{\delta_1 K^2 \|\theta^*\|^2}{\sigma}\label{eq:epsvthetaalign2}\\
\sup_{\alpha,\alpha' \in [-\pi,\pi)}
\frac{1}{|\alpha-\alpha'|}
\left|\Re \sum_{k=1}^K \overline{\eps_k}\Big(e^{ik\alpha}-e^{ik\alpha'}\Big)
v_k\right| &\leq \frac{\delta_1 K \|\theta^*\|}{\sigma}
\label{eq:epsvthetaalign3}
\end{align}}
The domain $\cE(\theta,v,\delta_1)$ in Lemma \ref{lemma-infor-3} is given by
$\cF_1(\theta,\delta_1) \cap \cF_2(\theta,v,\delta_1)$.

\begin{proof}[Proof of Lemma \ref{lemma-infor-3}]
Let us denote
\[y=\theta^*+\sigma \eps\]
The idea will be to approximate $\Var_{\alpha \sim \cP_{\theta,\eps}}$ by the
variance with respect to a Gaussian law over $\alpha$. We fix a small constant
$\delta_0>0$ to be determined, and take $C_1>0$ large enough
and $\delta_1>0$ small enough so that the conclusions of Lemma
\ref{lemma:angleconc} hold. Let
\begin{equation}\label{eq:alpha0}
\alpha_0=\argmax_{\alpha \in \cA} \langle y,\,
g(\alpha) \cdot \theta \rangle
\end{equation}
(where we may take any maximizer if it is not unique). Then (\ref{eq:maxangle})
guarantees that $\alpha_0 \in [-\frac{\delta_0}{K},\frac{\delta_0}{K}]$.

In the sense of Section \ref{sec:complexrepr}, denote the complex
representations of $\theta^*,\eps,y,\theta,v \in \R^{2K}$ by
\[\tilde{\theta}^*=(\theta_1^*,\ldots,\theta_K^*) \in \C^K,
\qquad \tilde{\eps}=(\eps_1,\ldots,\eps_K) \in \C^K,
\qquad \tilde{y}=(y_1,\ldots,y_K) \in \C^K,\]
\[\tilde{\theta}=(\theta_1,\ldots,\theta_K) \in \C^K, \qquad
\tilde{v}=(v_1,\ldots,v_K) \in \C^K.\]
Then the complex representation of $g(\alpha) \cdot \theta$
is $(e^{ik\alpha}\theta_k:k=1,\ldots,K)$.
By the inner-product relation (\ref{eq:CRisometry}), we have
\begin{equation}\label{eq:weightexpr}
\langle y,g(\alpha) \cdot \theta \rangle
=\Re \sum_{k=1}^K \overline{y_k} \cdot e^{ik\alpha}\theta_k.
\end{equation}
The first-order condition for optimality of $\alpha_0$ in
(\ref{eq:alpha0}) yields
\[0=\frac{d}{d\alpha} \langle y,g(\alpha) \cdot \theta
\rangle\Big|_{\alpha=\alpha_0}=\Re \sum_{k=1}^K \overline{y_k} \cdot
ike^{ik\alpha_0} \cdot \theta_k.\]
Applying this condition and the decomposition
\begin{align*}
e^{ik\alpha}=e^{ik\alpha_0}\Big[1+ik(\alpha-\alpha_0)+
\Big(e^{ik(\alpha-\alpha_0)}-1-ik(\alpha-\alpha_0)\Big)\Big]
\end{align*}
to (\ref{eq:weightexpr}),
we may write the density function (\ref{eq:Pthetaeps}) for
the distribution $\cP_{\theta,\eps}$ as
\[\frac{d\cP_{\theta,\eps}(\alpha)}{d\alpha}
\propto \exp\left(\frac{\langle y,g(\alpha)\cdot\theta\rangle}{\sigma^2}\right)
\propto \exp\left(\frac{p(\alpha)}{\sigma^2}\right),\]
where (also dropping constant terms that do not depend on $\alpha$)
\[p(\alpha)=\Re \sum_{k=1}^K \overline{y_k} \cdot
e^{ik\alpha_0}\Big(e^{ik(\alpha-\alpha_0)}-1-ik(\alpha-\alpha_0)\Big)\theta_k.\]

For $\alpha \in [-\frac{\delta_0}{K},\frac{\delta_0}{K}]$, we now establish a quadratic
approximation for $p(\alpha)$. We have
\[p(\alpha)=\mathrm{I}(\alpha)+\mathrm{II}(\alpha)+\mathrm{III}(\alpha)\]
where
\begin{align*}
\mathrm{I}(\alpha)&=\Re \sum_{k=1}^K \overline{\theta_k^*} \cdot
e^{ik\alpha_0}\Big(e^{ik(\alpha-\alpha_0)}-1-ik(\alpha-\alpha_0)\Big)
\theta_k^*\\
\mathrm{II}(\alpha)&=\Re \sum_{k=1}^K \overline{\theta_k^*} \cdot
e^{ik\alpha_0}\Big(e^{ik(\alpha-\alpha_0)}-1-ik(\alpha-\alpha_0)\Big)
(\theta_k-\theta_k^*)\\
\mathrm{III}(\alpha)&=\Re \sum_{k=1}^K \sigma \overline{\eps_k} \cdot
e^{ik\alpha_0}\Big(e^{ik(\alpha-\alpha_0)}-1-ik(\alpha-\alpha_0)\Big)\theta_k.
\end{align*}
For $\mathrm{I}(\alpha)$, observe that
$\overline{\theta_k^*}\theta_k^*=|\theta_k^*|^2$ is real, and
\begin{align*}
\Re e^{ik\alpha_0}\Big(e^{ik(\alpha-\alpha_0)}-1-ik(\alpha-\alpha_0)\Big)
&=\cos(k\alpha)-\cos(k\alpha_0)+k(\alpha-\alpha_0)\sin(k\alpha_0)\\
&=-\frac{k^2(\alpha-\alpha_0)^2}{2} \cos(k\tilde{\alpha}),
\end{align*}
for some $\tilde{\alpha}$ between $\alpha$ and $\alpha_0$.
Since $\alpha,\alpha_0 \in [-\frac{\delta_0}{K},\frac{\delta_0}{K}]$
and $k \leq K$, for sufficiently small $\delta_0$ this implies
\[-\frac{k^2(\alpha-\alpha_0)^2}{4} \geq 
\Re e^{ik\alpha_0}\Big(e^{ik(\alpha-\alpha_0)}-1-ik(\alpha-\alpha_0)\Big)
\geq -\frac{3k^2(\alpha-\alpha_0)^2}{4}.\]
Then, applying \revise{the second and third bounds of (\ref{eq:krsumbounds}),
there are constants $C,c>0$ (independent of $\delta_0$) such that
\[-cK^2\|\theta^*\|^2(\alpha-\alpha_0)^2 \geq \mathrm{I}(\alpha) \geq
-CK^2\|\theta^*\|^2(\alpha-\alpha_0)^2.\]}
For $\mathrm{II}(\alpha)$,
we apply $|e^{is}-1-is| \leq s^2$ for all real values $s \in \R$,
Cauchy-Schwarz, and the condition $\theta \in \cB(\delta_1)$ to obtain
\[|\mathrm{II}(\alpha)|
\leq \sum_{k=1}^K k^2(\alpha-\alpha_0)^2
|\overline{\theta_k^*}||\theta_k-\theta_k^*|
\leq \revise{(\alpha-\alpha_0)^2 K^2\sqrt{\sum_{k=1}^K |\theta_k^*|^2}
\sqrt{\sum_{k=1}^K
|\theta_k-\theta_k^*|^2} \leq \delta_1 (\alpha-\alpha_0)^2 K^2\|\theta^*\|^2}.\]
For $\mathrm{III}(\alpha)$, we
apply the condition (\ref{eq:epsvthetaalign2}) for $\eps \in
\cF_2(\theta,v,\delta_1)$ to obtain
\[|\mathrm{III}(\alpha)| \leq \revise{\delta_1(\alpha-\alpha_0)^2
K^2\|\theta^*\|^2}.\]
Combining these bounds, for sufficiently small $\delta_1>0$ and some constants
$C_0,c_0>0$ which we may take independent of $\delta_0,\delta_1$,
we arrive at the desired quadratic approximation
\revise{
\begin{equation}\label{eq:quadraticapprox}
-c_0K^2\|\theta^*\|^2(\alpha-\alpha_0)^2 \geq p(\alpha) \geq
-C_0K^2\|\theta^*\|^2(\alpha-\alpha_0)^2
\qquad \text{ for } \alpha \in [-\tfrac{\delta_0}{K},\tfrac{\delta_0}{K}].
\end{equation}}

This implies the following variance bound: Denote $I_1=[-\frac{\delta_0}{K},
\frac{\delta_0}{K}]$ and $I_2=[-\pi,\pi) \setminus [-\frac{\delta_0}{K},
\frac{\delta_0}{K}]$. For any bounded function
$f:[-\pi,\pi) \to \R$, denote $\|f\|_\infty=\sup_{\alpha \in [-\pi,\pi)}
|f(\alpha)|$. Then
\begin{align}
\Var_{\alpha \sim \cP_{\theta,\eps}}[f(\alpha)]
&=\inf_{x \in \R} \frac{\int_{-\pi}^\pi (f(\alpha)-x)^2
e^{p(\alpha)/\sigma^2}d\alpha}
{\int_{-\pi}^\pi e^{p(\alpha)/\sigma^2}d\alpha}\nonumber\\
&\leq \inf_{x \in \R} \frac{\int_{I_1} (f(\alpha)-x)^2
e^{p(\alpha)/\sigma^2}d\alpha+4\|f\|_\infty^2 \int_{I_2} e^{p(\alpha)/\sigma^2}
d\alpha}{\int_{-\pi}^\pi e^{p(\alpha)/\sigma^2}d\alpha}\nonumber\\
&\leq \revise{\inf_{x \in \R} \frac{\int_{I_1} (f(\alpha)-x)^2
e^{-c_0K^2\|\theta^*\|^2(\alpha-\alpha_0)^2/\sigma^2}d\alpha}
{\int_{I_1} e^{-C_0K^2\|\theta^*\|^2(\alpha-\alpha_0)^2/\sigma^2}d\alpha}
+4\|f\|_\infty^2 e^{-c(\delta_0)\|\theta^*\|^2/\sigma^2}}\label{eq:varexpr1}
\end{align}
where, in the last line, we have used (\ref{eq:quadraticapprox}) as well as
(\ref{eq:angleconc}) to bound $\P_{\alpha \sim \cP_{\theta,\eps}}[\alpha \in
I_2] \leq \revise{e^{-c(\delta_0)\|\theta^*\|^2/\sigma^2}}$ for a constant $c(\delta_0)>0$
depending on $\delta_0$. For the denominator of the first term of
(\ref{eq:varexpr1}), we may evaluate the Gaussian integral as
\revise{
\begin{align*}
&\int_{I_1} e^{-C_0K^2\|\theta^*\|^2(\alpha-\alpha_0)^2/\sigma^2}d\alpha\\
&=\left(\frac{2C_0K^2\|\theta^*\|^2}{\sigma^2}\right)^{-1/2}
\sqrt{2\pi}\left(1-\tilde{\Phi}\left[
\sqrt{\frac{2C_0K^2\|\theta^*\|^2}{\sigma^2}}
\left(\frac{\delta_0}{K}+\alpha_0\right)\right]
-\tilde{\Phi}\left[\sqrt{\frac{2C_0K^2\|\theta^*\|^2}{\sigma^2}}
\left(\frac{\delta_0}{K}-\alpha_0\right)\right]\right).
\end{align*}}
Here, since $|\alpha_0| \leq \delta_0/K$, both values of $\tilde{\Phi}$ are at
most $1/2$. Furthermore, under the condition
$\frac{\|\theta^*\|^2}{\sigma^2} \geq C_1\log K$ for $C_1>0$ large enough depending on
$\delta_0$, at least one value of $\tilde{\Phi}$
is less than $1/4$. Thus, for a constant $C>0$ independent of
$\delta_0,\delta_1$, we have simply
\revise{
\[\int_{I_1} e^{-C_0K^2\|\theta^*\|^2(\alpha-\alpha_0)^2/\sigma^2}d\alpha
\geq \left(\frac{CK^2\|\theta^*\|^2}{\sigma^2}\right)^{-1/2}.\]}
Combining this with the normalization constant for the Gaussian law in the
numerator of the first term of (\ref{eq:varexpr1}), we obtain
\begin{equation}\label{eq:varexpr2}
\Var_{\alpha \sim \cP_{\theta,\eps}}[f(\alpha)]
\leq C\Var_{\alpha \sim \Normal(\alpha_0,\tau^2)}
[f(\alpha)]+ 4\|f\|_\infty^2e^{-c(\delta_0)\|\theta^*\|^2/\sigma^2},
\qquad \tau^2:=\frac{C'\sigma^2}{K^2\|\theta^*\|^2}.
\end{equation}
Here $C,C'>0$ are some constants depending only on $c_\gen$ and
independent of $\delta_0,\delta_1$, whereas $c(\delta_0)$ depends also on
$\delta_0$.

Finally, we apply this bound (\ref{eq:varexpr2}) to the function
$f(\alpha)=v^\top g(\alpha)^{-1} y=\langle y,g(\alpha)v \rangle$.
Observe that
\begin{equation}\label{eq:finftybound}
\|f\|_\infty \leq \|\theta^*\|\|v\|+\sigma \sup_{\alpha \in [-\pi,\pi)}
\langle \eps,g(\alpha)v \rangle \leq C\|\theta^*\|,
\end{equation}
the last inequality using $\|v\|=1$ and (\ref{eq:epsvthetaalign1}) for $\eps \in
\cF_2(\theta,v,\delta_1)$. To bound
$\Var_{\alpha \sim \mathcal{N}(\alpha_0,\tau^2)}[f(\alpha)]$, we apply
again the inner-product relation
(\ref{eq:CRisometry}) to write the complex representation of $f(\alpha)$ as
\[f(\alpha)=\Re \sum_{k=1}^K \overline{y_k} \cdot e^{ik\alpha}v_k
=\mathrm{I}(\alpha)+\mathrm{II}(\alpha)+\mathrm{III}(\alpha)\]
where
\begin{align*}
\mathrm{I}(\alpha)&=\Re \sum_{k=1}^K \overline{\theta_k^*} \cdot
e^{ik\alpha_0}\Big(1+ik(\alpha-\alpha_0)\Big)v_k\\
\mathrm{II}(\alpha)&=\Re \sum_{k=1}^K \overline{\theta_k^*} \cdot
e^{ik\alpha_0}\Big(e^{ik(\alpha-\alpha_0)}-1-ik(\alpha-\alpha_0)\Big)v_k\\
\mathrm{III}(\alpha)&=\Re \sum_{k=1}^K \sigma \overline{\eps_k}
\cdot e^{ik\alpha} v_k
\end{align*}
Next, we are going to upper bound their variances under $\alpha\sim \mathcal
N(\alpha_0, \tau^2)$. For $\mathrm{I}(\alpha)$, we may drop the constant term that is independent of
$\alpha$ and write
\begin{equation}\label{eq:VarI}
\Var_{\alpha \sim \mathcal{N}(\alpha_0,\tau^2)}[\mathrm{I}(\alpha)]
=\Var_{\alpha \sim \mathcal{N}(\alpha_0,\tau^2)}\left[\Re \sum_{k=1}^K
\overline{\theta_k^*} \cdot ik\alpha \cdot v_k
+\Re \sum_{k=1}^K \overline{\theta_k^*} \cdot (e^{ik\alpha_0}-1)ik\alpha \cdot
v_k\right].
\end{equation}
Recalling the tangent vector $u^*$ from (\ref{eq:ustar}), observe that its
complex representation is
\[\tilde{u}^*=\frac{d}{d\alpha}\Big(e^{ik\alpha}\theta_k^*:
k=1,\ldots,K\Big)\Big|_{\alpha=0}=\Big(ik\theta_k^*:k=1,\ldots,K\Big).\]
Then the inner-product relation (\ref{eq:CRisometry}) and the given
orthogonality condition $\langle u^*,v \rangle=0$ imply
\[\Re \sum_{k=1}^K -ik \cdot \overline{\theta_k^*} \cdot v_k=0,\]
so the first term inside the variance of (\ref{eq:VarI}) is 0. Applying
$|e^{ik\alpha_0}-1| \leq k|\alpha_0| \leq \delta_0 k/K$ for the second term,
\begin{align*}
\Var_{\alpha \sim \mathcal{N}(\alpha_0,\tau^2)}[\mathrm{I}(\alpha)]
&=\Var_{\alpha \sim \mathcal{N}(\alpha_0,\tau^2)}[\alpha]
\cdot \left(\Re \sum_{k=1}^K \overline{\theta_k^*} \cdot
(e^{ik\alpha_0}-1)ik\cdot v_k\right)^2\\
&\leq \Var_{\alpha \sim \mathcal{N}(\alpha_0,\tau^2)}[\alpha]
\cdot \left(\sum_{k=1}^K |\theta_k^*| \cdot \frac{\delta_0k^2}{K} \cdot
|v_k|\right)^2\\
&\leq \revise{\Var_{\alpha \sim \mathcal{N}(\alpha_0,\tau^2)}[\alpha]
\cdot \delta_0^2 K^2 \sum_{k=1}^K |\theta^*_k|^2 \cdot\sum_{k=1}^K |v_k|^2
=\tau^2 \delta_0^2 K^2\|\theta^*\|^2.}
\end{align*}
For $\mathrm{II}(\alpha)$, applying $|e^{is}-1-is| \leq s^2$ for any real value
$s \in \R$,
\begin{align*}
\Var_{\alpha \sim \mathcal{N}(\alpha_0,\tau^2)}[\mathrm{II}(\alpha)]
&\leq \E_{\alpha \sim \mathcal{N}(\alpha_0,\tau^2)}[\mathrm{II}(\alpha)^2]\\
&\leq \E_{\alpha \sim \mathcal{N}(\alpha_0,\tau^2)}\left[
\left(\sum_{k=1}^K k^2(\alpha-\alpha_0)^2|\theta_k^*||v_k|\right)^2\right]\\
&\leq \revise{\E_{\alpha \sim \mathcal{N}(\alpha_0,\tau^2)}[(\alpha-\alpha_0)^4]
K^4\sum_{k=1}^K |\theta_k^*|^2 \sum_{k=1}^K |v_k|^2
=3\tau^4K^4\|\theta^*\|^2}.
\end{align*}
For $\mathrm{III}(\alpha)$, we may center by a constant independent of $\alpha$
and apply (\ref{eq:epsvthetaalign3}) to obtain
\begin{align*}
\Var_{\alpha \sim \mathcal{N}(\alpha_0,\tau^2)}[\mathrm{III}(\alpha)]
&\leq \E_{\alpha \sim \mathcal{N}(\alpha_0,\tau^2)}
\left[\left(\Re \sum_{k=1}^K \sigma \overline{\eps_k}\Big(e^{ik\alpha}
-e^{ik\alpha_0}\Big)v_k\right)^2\right]\\
&\leq \revise{\sigma^2 \cdot
\E_{\alpha \sim \mathcal{N}(\alpha_0,\tau^2)}[(\alpha-\alpha_0)^2]
\cdot \frac{\delta_1^2K^2\|\theta^*\|^2}{\sigma^2}
=\tau^2 \delta_1^2 K^2\|\theta^*\|^2}.
\end{align*}
Combining all of the above, we have
\begin{align}
\Var_{\alpha \sim \mathcal{N}(\alpha_0,\tau^2)}[f(\alpha)]
&\leq 3\Var[\mathrm{I}(\alpha)]+3\Var[\mathrm{II}(\alpha)]
+3\Var[\mathrm{III}(\alpha)]\nonumber\\
&\leq \revise{\tau^2 \cdot C(\delta_0^2+\delta_1^2)K^2\|\theta^*\|^2
+\tau^4 \cdot CK^4\|\theta^*\|^2}\label{eq:finalvarfbound}
\end{align}
for a constant $C>0$ independent of $\delta_0,\delta_1$.

Applying (\ref{eq:finftybound}) and (\ref{eq:finalvarfbound}) and the value of
$\tau^2$ to (\ref{eq:varexpr2}), for a constant $C'>0$ independent of
$\delta_0,\delta_1$,
\revise{
\[\Var_{\alpha \sim \cP_{\theta,\eps}}[f(\alpha)]
\leq C'\left(\delta_0^2+\delta_1^2+\frac{\sigma^2}{\|\theta^*\|^2}
+\frac{\|\theta^*\|^2}{\sigma^2}e^{-c(\delta_0)\|\theta^*\|^2/\sigma^2}\right)\sigma^2.\]}
Then, choosing $\delta_0,\delta_1>0$ sufficiently small depending on $\eta$,
and applying $\frac{\|\theta^*\|^2}{\sigma^2} \geq C_1\log K$ for a sufficiently large
constant $C_1>0$ depending on $\delta_0$ and $\eta$, we obtain
$\Var_{\alpha \sim \cP_{\theta,\eps}}[f(\alpha)] \leq \eta \sigma^2$ as desired.
\end{proof}

\begin{proof}[Proof of Lemma \ref{lemma:epsgoodprob}]
Applying (\ref{eq:GPtail}) with $t=\delta_1\|\theta^*\|^2/\sigma$,
\[\P\left[\sup_{\alpha \in \cA} |\langle \eps,g(\alpha) \cdot \theta \rangle|
>\frac{\delta_1\|\theta^*\|^2}{\sigma} \text{ and } \|\eps\| \leq s\right]
\leq \revise{\frac{C\sigma}{\|\theta^*\|} \cdot Ks \cdot
e^{-c\|\theta^*\|^2/\sigma^2}}\]
for constants $C,c>0$ (depending on $\delta_1$). Let us take
\[s=\max(\sqrt{4K},\revise{\|\theta^*\|/\sigma}).\]
Then applying $\frac{\|\theta^*\|^2}{\sigma^2} \geq C_1\log K$ for sufficiently large
$C_1>0$, this probability bound is at most $e^{-c'\|\theta^*\|^2/\sigma^2}$. By a
chi-squared tail bound, since $\|\eps\|^2 \sim \chi_{2K}^2$ and
$s^2 \geq 4K$, we have
$\P[\|\eps\|^2>s^2] \leq e^{-cs^2} \leq e^{-c\|\theta^*\|^2/\sigma^2}$.
Combining these bounds gives, for a constant $c>0$,
\[\P\left[\sup_{\alpha \in \cA} |\langle \eps,g(\alpha) \cdot \theta \rangle|
>\frac{\delta_1\|\theta^*\|^2}{\sigma}\right] \leq
e^{-c\|\theta^*\|^2/\sigma^2},\]
so $\eps \in \cF_1(\theta,\delta_1)$ with probability at least
$1-e^{-c\|\theta^*\|^2/\sigma^2}$.
The same argument applied with the unit vector $v$ in place of $\theta$ shows
\revise{
\[\P\left[\sup_{\alpha \in \cA} |\langle \eps,g(\alpha) \cdot v \rangle|
>\frac{\|\theta^*\|}{\sigma}\right] \leq e^{-c\|\theta^*\|^2/\sigma^2},\]}
so (\ref{eq:epsvthetaalign1}) holds with probability at least
$1-e^{-c\|\theta^*\|^2/\sigma^2}$.

For the condition (\ref{eq:epsvthetaalign2}),
note that $\eps_k \sim \mathcal{N}_\C(0,2)$ and
these are independent for $k=1,\ldots,K$. Then
\[f_{\alpha,\alpha'}(\eps):=\alpha^{-2}\sum_{k=1}^K \overline{\eps_k}
e^{ik\alpha'}(e^{ik\alpha}-1-ik\alpha)\theta_k\]
has distribution $\mathcal{N}_\C(0,2\tau^2)$ where
\[\tau^2=\alpha^{-4}\sum_{k=1}^K 
|e^{ik\alpha'}(e^{ik\alpha}-1-ik\alpha)\theta_k|^2.\]
So $\Re f_{\alpha,\alpha'}(\eps) \sim \mathcal{N}(0,\tau^2)$. Applying
$|e^{is}-1-is| \leq s^2$ for all $s \in \R$, we have
\revise{
\[\tau^2 \leq \sum_{k=1}^K k^4|\theta_k|^2
\leq K^4\|\theta^*\|^2.\]
Then setting
\[t=\frac{\delta_1 K^2\|\theta^*\|^2}{\sigma},\]}
a Gaussian tail bound yields, for a constant $c>0$ (depending on $\delta_1$),
\[\P[|\Re f_{\alpha,\alpha'}(\eps)|>t/2] \leq e^{-c\|\theta^*\|^2/\sigma^2}.\]
Differentiating in $\alpha$ and $\alpha'$ and applying
$|e^{is}-1-is| \leq s^2$ and $|e^{is}-1-is+s^2/2| \leq |s|^3$ for $s \in \R$, we
have
\begin{align*}
|\partial_{\alpha'} \Re f_{\alpha,\alpha'}(\eps)|
&=\left|\alpha^{-2}\Re \sum_{k=1}^K \overline{\eps_k}
\cdot ike^{ik\alpha'}(e^{ik\alpha}-1-ik\alpha)\theta_k\right|
\leq \sum_{k=1}^K k^3|\eps_k||\theta_k|
\leq K^3\|\eps\|\|\theta\|,\\
|\partial_\alpha \Re f_{\alpha,\alpha'}(\eps)|
&=\left|\alpha^{-3}\Re \sum_{k=1}^K \overline{\eps_k}
e^{ik\alpha'}(e^{ik\alpha}(ik\alpha-2)+2+ik\alpha)\theta_k\right|
\leq C\sum_{k=1}^K k^3|\eps_k||\theta_k| \leq CK^3\|\eps\|\|\theta\|.
\end{align*}
(For the second line, we have explicitly evaluated the derivative,
and then applied $e^{ik\alpha}=1+ik\alpha-k^2\alpha^2/2+O(k^3\alpha^3)$
and canceled terms to yield the first inequality.)
Then, on an event $\{\|\eps\|<s\}$, $\Re f_{\alpha,\alpha'}(\eps)$ is
$L$-Lipschitz in both $\alpha$ and $\alpha'$, for 
\revise{$L=C_0K^3\|\theta^*\|s$} and a constant $C_0>0$.
We set $\delta=t/(4L)$ and let $N_\delta$ be a $\delta$-net of
$[-\pi,\pi)$ having cardinality
\[|N_\delta|=\frac{8\pi L}{t} \leq \revise{\frac{C\sigma}{\|\theta^*\|}} \cdot Ks.\]
Then
\[\P\left[\sup_{\alpha,\alpha' \in [-\pi,\pi)} |\Re
f_{\alpha,\alpha'}(\eps)|>t \text{ and } \|\eps\|<s\right]
\leq \P\left[\sup_{\alpha,\alpha' \in N_\delta} |\Re
f_{\alpha,\alpha'}(\eps)|>t/2 \right]
\leq |N_\delta|^2 \cdot e^{-c\|\theta^*\|^2/\sigma^2}.\]
Then, setting $s=\max(\sqrt{4K},\revise{\|\theta^*\|/\sigma})$ and applying the same
argument as above shows
\[\P\left[\sup_{\alpha,\alpha' \in [-\pi,\pi)}
|\Re f_{\alpha,\alpha'}(\eps)|>t\right] \leq e^{-c\|\theta^*\|^2/\sigma^2},\]
so (\ref{eq:epsvthetaalign2}) holds with probability at least
$1-e^{-c\|\theta^*\|^2/\sigma^2}$.

The argument for (\ref{eq:epsvthetaalign3}) is analogous. We define
$\gamma=\alpha-\alpha'$,
\[f_{\alpha',\gamma}(\eps):=\gamma^{-1}
\sum_{k=1}^K \overline{\eps_k}e^{ik\alpha'}(e^{ik\gamma}-1)v_k
\sim \mathcal{N}_\C(0,2\tau^2),
\qquad \tau^2:=\gamma^{-2}\sum_{k=1}^K
|(e^{ik\alpha'}(e^{ik\gamma}-1)v_k|^2\]
and set \revise{$t=\delta_1K\|\theta^*\|/\sigma$}. Applying $|e^{ik\gamma}-1| \leq
k|\gamma|$ and $\|v\|=1$,
we have $\tau^2 \leq \sum_{k=1}^K k^2|v_k|^2 \leq K^2$, so that a Gaussian tail
bound yields $\P[|\Re f_{\alpha',\gamma}(\eps)|>t/2] \leq e^{-c\|\theta^*\|^2/\sigma^2}$.
On the event $\{\|\eps\|<s\}$, we have the Lipschitz bounds
\begin{align*}
|\partial_{\alpha'} \Re f_{\alpha',\gamma}(\eps)|
&=\left|\gamma^{-1} \Re \sum_{k=1}^K \overline{\eps_k} \cdot ike^{ik\alpha'}
(e^{ik\gamma}-1)v_k\right| \leq \sum_{k=1}^K k^2|\eps_k||v_k|
\leq K^2\|\eps\|\|v\| \leq K^2s,\\
|\partial_\gamma \Re f_{\alpha',\gamma}(\eps)|
&=\left|\gamma^{-2} \Re \sum_{k=1}^K \overline{\eps_k} \cdot e^{ik\alpha'}
(e^{ik\gamma}(ik\gamma-1)+1)v_k\right|
\leq C\sum_{k=1}^K k^2|\eps_k||v_k| \leq CK^2\|\eps\|\|v\| \leq CK^2s,
\end{align*}
where we have applied $e^{ik\gamma}=1+ik\gamma+O(k^2\gamma^2)$. Then
applying a covering net argument as above, whose details we omit for brevity,
we obtain
\[\P\left[\sup_{\alpha' \in [-\pi,\pi),\;\gamma \in [-2\pi,2\pi)}
|\Re f_{\alpha',\gamma}(\eps)|>t\right] \leq e^{-c\|\theta^*\|^2/\sigma^2},\]
so (\ref{eq:epsvthetaalign3}) holds with probability at least
$1-e^{-c\|\theta^*\|^2/\sigma^2}$. Combining these bounds yields the lemma.
\end{proof}

\begin{proof}[Proof of Lemma \ref{lemma-delta-A}]
Throughout the proof, $C,C',c,c'$ etc.\ are positive constants depending only on
$c_\gen,\eta$ and changing from instance to instance.
Recall the expression (\ref{eq:hessexpansion}).
Let $C_1,\delta_1>0$ be such that the conclusion of Lemma \ref{lemma-infor-3}
holds with $\eta/6$ in place of $\eta$.
For any $\theta \in \cB(\delta_1)$ and unit vector $v$ 
satisfying $\langle u^*,v\rangle=0$, let us apply
\[\Var_{\alpha \sim \cP_{\theta,\eps}}
\Big[v^\top g(\alpha)^{-1}(\theta^*+\sigma \eps)\Big]
\leq \E_{\alpha \sim \cP_{\theta,\eps}}
\Big[\big(v^\top g(\alpha)^{-1}(\theta^*+\sigma \eps)\big)^2\Big]
\leq \|\theta^*+\sigma \eps\|^2\]
to upper-bound the second term of (\ref{eq:hessexpansion}) as
\begin{equation}\label{eq:hessboundI123}
\frac{1}{N}\sum_{m=1}^N \Var_{\alpha \sim \cP_{\theta,\eps^{(m)}}}\Big[v^\top
g(\alpha)^{-1}(\theta^*+\sigma \eps^{(m)})\Big] \leq I_1(\theta,v)+I_2(\theta,v)+I_3
\end{equation}
where
\begin{align*}
I_1(\theta,v) &= \frac1N\sum_{m=1}^N \Var_{\alpha \sim \cP_{\theta,\eps^{(m)}}}\Big[v^\top
g(\alpha)^{-1}(\theta^*+\sigma \eps^{(m)})\Big]
\cdot \mathbf{1}[\varepsilon^{(m)}\in \cE(\theta,v,\delta_1)],\\
I_2(\theta,v) &= \frac1N\sum_{m=1}^N \|\theta^*+\sigma\eps^{(m)}\|^2
\cdot \mathbf{1}\left[\varepsilon^{(m)} \notin \cE(\theta,v,\delta_1)
\text{ and } \|\varepsilon^{(m)}\|^2 \leq
4K+\revise{\frac{\|\theta^*\|^2}{\sigma^2}}\right],\\
I_3 &= \frac1N\sum_{m=1}^N \|\theta^*+\sigma \eps^{(m)}\|^2
\cdot
\mathbf{1}\left[\|\varepsilon^{(m)}\|^2>4K+\revise{\frac{\|\theta^*\|^2}{\sigma^2}}\right].
\end{align*}
Here $I_1,I_2$ are dependent on $(\theta,v)$, whereas $I_3$ is independent of
$(\theta,v)$.

Lemma \ref{lemma-infor-3} applied with $\eta/6$ immediately gives
the deterministic bound
\begin{equation}\label{eq:I1bound}
I_1(\theta,v) \leq \eta\sigma^2/6.
\end{equation}
For $I_2(\theta,v)$, on the event $\|\eps\|^2 \leq 4K+\|\theta^*\|^2/\sigma^2$, we have
for a constant $C_2>0$ that
\[\|\theta^*+\sigma \eps\|^2 \leq 2\|\theta^*\|^2+2\sigma^2\|\eps\|^2
\leq \revise{C_2(\|\theta^*\|^2+K\sigma^2)}.\]
Thus
\[I_2(\theta,v) \leq \revise{C_2(\|\theta^*\|^2+K\sigma^2)} \cdot \frac1N\sum_{m=1}^N
\mathbf{1}[\varepsilon^{(m)} \notin \cE(\theta,v,\delta_1)].\]
Denote $p=\P[\eps^{(m)} \notin \cE(\theta,v,\delta_1)]$ and
\revise{$q=\frac{\eta\sigma^2}{6C_2(\|\theta^*\|^2+K\sigma^2)}$. By Lemma \ref{lemma:epsgoodprob},
\[p \leq e^{-c\|\theta^*\|^2/\sigma^2},\]
so in particular $p<q$ for $\frac{\|\theta^*\|^2}{\sigma^2} \geq C_1\log K$ and
sufficiently large $C_1>0$.} Then by a Chernoff bound for binomial random
variables \citep[Theorem 2.3.1]{vershynin2018high},
\[\P[I_2(\theta,v) \geq \eta \sigma^2/6]=\P\left[\sum_{m=1}^N
\mathbf{1}[\varepsilon^{(m)} \notin \cE(\theta,v,\delta_1)] \geq Nq\right] \leq
\left(\frac{ep}{q}\right)^{Nq}.\]
We have $(ep/q) \leq e^{-c'\|\theta^*\|^2/\sigma^2}$ when
$\frac{\|\theta^*\|^2}{\sigma^2} \geq C_1\log K$ for sufficiently large $C_1>0$,
so this yields
\begin{equation}\label{eq:I2bound}
\P[I_2(\theta,v) \geq \eta \sigma^2/6] \leq
\revise{e^{-\frac{cN}{1+K\sigma^2/\|\theta^*\|^2}}}.
\end{equation}

For $I_3$, we bound separately its mean and its concentration.
Denote the summand of $I_3$ as
\[z^{(m)}=\|\theta^*+\sigma \eps^{(m)}\|^2
\cdot \mathbf{1}\left[\|\eps^{(m)}\|^2>4K+\frac{\|\theta^*\|^2}{\sigma^2}\right].\]
Let $p'=\P[\|\eps^{(m)}\|^2>4K+\|\theta^*\|^2/\sigma^2]$. Then applying $\|\eps^{(m)}\|^2
\sim \chi_{2K}^2$ and a chi-squared tail bound,
$p' \leq e^{-c(2K+\|\theta^*\|^2/\sigma^2)} \leq e^{-c\|\theta^*\|^2/\sigma^2}$.
So by Cauchy-Schwarz,
\begin{equation}\label{eq:Ezbound}
\mathbb{E}[z^{(m)}] \leq \sqrt{\E \left[\|\theta^*+\sigma \eps^{(m)}\|^4\right]}
\cdot \sqrt{p'} \leq \revise{C(\|\theta^*\|^2+K\sigma^2)
\cdot e^{-c\|\theta^*\|^2/2\sigma^2}} \leq \eta\sigma^2/12
\end{equation}
the last inequality holding for $\frac{\|\theta^*\|^2}{\sigma^2} \geq C_1\log K$ and
sufficiently large $C_1>0$. For the concentration, let
$\|X\|_{\psi_1}=\inf\{s>0:\E[e^{|X|/s}] \leq 2\}$ denote the
sub-exponential norm of a random variable $X$. Observe that similarly by
Cauchy-Schwarz,
\begin{align*}
\E\Big[e^{\frac{|z^{(m)}|}{s}}\Big]
&=1-p'+\E\left[\exp\left(\frac{\|\theta^*+\sigma \eps^{(m)}\|^2}{s}\right)
\cdot \mathbf{1}\left[\|\eps^{(m)}\|^2>4K+\frac{\|\theta^*\|^2}{\sigma^2}\right]\right]\\
&\leq 1+\revise{\sqrt{\E\Big[\exp\Big(4\|\theta^*\|^2+4\sigma^2\|\eps^{(m)}\|^2)/s\Big)\Big]}} \cdot \sqrt{p'}.
\end{align*}
Applying the moment generating function bound
$\E[\exp(t\|\eps^{(m)}\|^2)]=(1-2t)^{-K} \leq e^{4tK}$ for $t<1/4$, we have
\[\E\Big[e^{\frac{|z^{(m)}|}{s}}\Big]
\leq 1+e^{2\|\theta^*\|^2/s} \cdot e^{8K\sigma^2/s} \cdot e^{-c\|\theta^*\|^2/2\sigma^2}
\leq 2\]
when \revise{$s \geq C'\sigma^2(1+K\sigma^2/\|\theta^*\|^2)$ for a sufficiently large constant
$C'>0$. So $\|z^{(m)}\|_{\psi_1} \leq C\sigma^2(1+K\sigma^2/\|\theta^*\|^2)$,
and Bernstein's inequality \citep[Theorem 2.8.1]{vershynin2018high} gives
\[\mathbb{P}\left[\frac1N\sum_{m=1}^N z^{(m)}-\mathbb{E}[z^{(m)}]>\eta
\sigma^2/12\right] \leq e^{-\frac{cN}{(1+K\sigma^2/\|\theta^*\|^2)^2}}.\]
Combining with (\ref{eq:Ezbound}),
\begin{equation}\label{eq:I3bound}
\P[I_3 \geq \eta\sigma^2/6]
=\P\left[\frac{1}{N}\sum_{m=1}^N z^{(m)}>\eta \sigma^2/6\right] \leq
e^{-\frac{cN}{(1+K\sigma^2/\|\theta^*\|^2)^2}}.
\end{equation}}

Applying (\ref{eq:hessexpansion}), (\ref{eq:hessboundI123}),
(\ref{eq:I1bound}), and (\ref{eq:I2bound}), we have
\begin{align}
\P\left[v^\top \nabla^2 R_N(\theta)v \leq
\frac{1}{\sigma^2}-\frac{\eta}{2\sigma^2} \text{ and }
I_3 \leq \frac{\eta\sigma^2}{6}\right]
&\leq \P\left[I_1(\theta,v)+I_2(\theta,v)+I_3 \geq \frac{\eta\sigma^2}{2}
\text{ and } I_3 \leq \frac{\eta\sigma^2}{6}\right]\nonumber\\
&\leq \P\left[I_2(\theta,v) \geq \frac{\eta\sigma^2}{6}\right]
\leq e^{-\frac{cN}{1+K\sigma^2/\|\theta^*\|^2}}.\label{eq:hesspointwisebound}
\end{align}
Let us apply a covering net argument to take a union bound over $(\theta,v)$,
and then combine with the bound (\ref{eq:I3bound}) for $I_3$ which
is independent of $(\theta,v)$.
We compute the Lipschitz constant of $v^\top \nabla^2
R_N(\theta)v$ in both $v$ and $\theta$: Taking the gradient in $v$,
\[\left\|\nabla_v \Big[v^\top \nabla^2 R_N(\theta)v\Big]\right\|
=\left\|2\nabla^2 R_N(\theta)v\right\|
=2\sup_{u:\|u\|=1} u^\top \nabla^2 R_N(\theta)v.\]
Then denoting $y^{(m)}=\theta^*+\sigma \eps^{(m)}$ and applying (\ref{eq:hess}),
\begin{align}
\left\|\nabla_v \Big[v^\top \nabla^2 R_N(\theta)v\Big]\right\|
&\leq 2\left(\sup_{u:\|u\|=1}
\frac{u^\top v}{\sigma^2}-\frac{1}{N\sigma^4}
\sum_{m=1}^N \Cov_{\alpha \sim \cP_{\theta,\eps^{(m)}}}
[u^\top g(\alpha)^{-1}y^{(m)}, v^\top g(\alpha)^{-1}y^{(m)}]\right)\nonumber\\
&\leq \frac{2}{\sigma^2}
+\frac{2}{N\sigma^4}\sum_{m=1}^N \|y^{(m)}\|^2
\leq \frac{2}{\sigma^2}+\frac{4\|\theta^*\|^2}{\sigma^4}
+\frac{4}{N\sigma^2} \sum_{m=1}^N \|\eps^{(m)}\|^2.\label{eq:gradquadv}
\end{align}
Similarly, taking the gradient in $\theta$,
\[\left\|\nabla_\theta \left[v^\top \nabla^2 R_N(\theta)v\right]\right\|
=\sup_{u:\|u\|=1} \nabla^3 R_N(\theta)[u,v,v]\]
where $\nabla^3 R_N(\theta)[u,v,v]$ is the 3rd-derivative tensor of
$R_N(\theta)$ evaluated at $u \otimes v \otimes v \in \R^{2K \times 2K \times
2K}$. We have
\[\nabla^3 R_N(\theta)[u,v,v]
=-\frac{1}{N\sigma^6}\sum_{m=1}^N \kappa_{\alpha \sim
\cP_{\theta,\eps^{(m)}}}^3 [u^\top g(\alpha)^{-1}y^{(m)},
v^\top g(\alpha)^{-1}y^{(m)}, v^\top g(\alpha)^{-1}y^{(m)}]\]
where $\kappa_{\alpha \sim \cP_{\theta,\eps^{(m)}}}^3[\cdot,\cdot,\cdot]$
denotes the 3rd-order mixed cumulant with respect to $\alpha \sim
\cP_{\theta,\eps^{(m)}}$. For any random variables $X,Y,Z$,
the moment-cumulant relations and H\"older's inequality give
\[|\kappa^3[X,Y,Z]| \leq C \cdot
\E[|X|^3]^{1/3}\E[|Y|^3]^{1/3}\E[|Z|^3]^{1/3}.\]
Thus,
\begin{equation}\label{eq:gradquadtheta}
\left\|\nabla_\theta \left[v^\top \nabla^2 R_N(\theta)v\right]\right\|
\leq \frac{C}{N\sigma^6}\sum_{m=1}^N \|y^{(m)}\|^3
\leq \frac{4C\|\theta^*\|^3}{\sigma^6}+\frac{4C}{N\sigma^3}
\sum_{m=1}^N \|\eps^{(m)}\|^3.
\end{equation}

On an event
\[\cA=\left\{\|\eps^{(m)}\|^2 \leq N \text{ for all } m=1,\ldots,N\right\},\]
(\ref{eq:gradquadv}) and (\ref{eq:gradquadtheta}) imply that $v^\top \nabla^2
R_N(\theta)v$ is $L_v$-Lipschitz in $v$ and $L_\theta$-Lipschitz in $\theta$,
for
\revise{
\[L_v=C'\left(\frac{\|\theta^*\|^2}{\sigma^2}+N\right)\frac{1}{\sigma^2},
\qquad
L_\theta=C'\left(\frac{\|\theta^*\|^2}{\sigma^2}+N\right)^{3/2}\frac{1}{\sigma^3}.\]}
Let $N_v$ be a $\delta_v$-net of $\{v:\|v\|=1,\langle
u^*,v \rangle=0\}$ and $N_\theta$ a $\delta_\theta$-net of $\{\theta:
\theta \in \cB(\delta_1)\}$, for $\delta_v=\eta/(4L_v\sigma^2)$ and
$\delta_\theta=\eta/(4L_\theta \sigma^2)$. This guarantees, for each
$\theta \in \cB(\delta_1)$ and unit vector $v$ with $\langle u^*,v \rangle=0$,
there exists $(\theta',v') \in N_\theta \times N_v$ such that
\[\Big|v^\top \nabla^2 R_N(\theta)v-{v'}^\top \nabla^2 R_N(\theta')v'\Big|
\leq L_\theta \|\theta-\theta'\|+L_v \|v-v'\| \leq \eta/2\sigma^2.\]
Then, applying these Lipschitz
bounds together with the pointwise bound (\ref{eq:hesspointwisebound}),
\begin{align}
&\P\left[\sup_{\theta \in \cB(\delta_1)}\,
\sup_{v:\|v\|=1,\,\langle u^*,v \rangle=0}\,
v^\top \nabla^2 R_N(\theta)v \leq \frac{1}{\sigma^2}-\frac{\eta}{\sigma^2}
\text{ and } I_3 \leq \frac{\eta \sigma^2}{6} \text{ and }
\cA \right]\nonumber\\
&\hspace{0.2in}\leq \P\left[\sup_{\theta \in N_\theta}\,\sup_{v \in N_v}
v^\top \nabla^2 R_N(\theta)v \leq \frac{1}{\sigma^2}
-\frac{\eta}{2\sigma^2} \text{ and } I_3 \leq \frac{\eta\sigma^2}{6}\right]
\leq |N_v| \cdot |N_\theta| \cdot e^{-\frac{cN}{1+K\sigma^2/\|\theta^*\|^2}}.
\label{eq:hessunionbound}
\end{align}

We may take the above nets to have cardinalities
\revise{
\begin{align*}
|N_v| \leq \left(1+\frac{2}{\delta_v}\right)^{2K}
&\leq \left[C'\left(\frac{\|\theta^*\|^2}{\sigma^2}+N\right)\right]^{2K},\\
|N_\theta| \leq \left(1+\frac{C\|\theta^*\|}{\delta_\theta}\right)^{2K}
&\leq \left[C'\left(\frac{\|\theta^*\|^2}{\sigma^2}+N\right)^{3/2}
\frac{\|\theta^*\|}{\sigma}\right]^{2K}.
\end{align*}
Observe that under the given assumptions $N \geq
C_0K(1+\frac{K\sigma^2}{\|\theta^*\|^2})
\log (K+\frac{\|\theta^*\|^2}{\sigma^2})$ and $\frac{\|\theta^*\|^2}{\sigma^2}
\geq C_1\log K$, we have
\[\frac{N}{\log N} \geq \frac{C_0K(1+\frac{K\sigma^2}{\|\theta^*\|^2})
\log (K+\frac{\|\theta^*\|^2}{\sigma^2})}{\log
[C_0K(1+\frac{K\sigma^2}{\|\theta^*\|^2})\log
(K+\frac{\|\theta^*\|^2}{\sigma^2})]}.\]
Considering separately the cases $\frac{K\sigma^2}{\|\theta^*\|^2} \leq 1$
and $1 \leq \frac{K\sigma^2}{\|\theta^*\|^2} \leq \frac{K}{C_1\log K}$,
it may be checked that this implies
\begin{equation}\label{eq:logKlogN}
\frac{N}{\log N} \geq C_0'K\left(1+\frac{K\sigma^2}{\|\theta^*\|^2}\right)
\end{equation}
where $C_0'$ may be taken to be a small absolute constant times $C_0/\log C_0$.}
Then for sufficiently large $C_0,C_0'>0$ and some constant $c'>0$, the right side of
(\ref{eq:hessunionbound}) may then be bounded as
$|N_v| \cdot |N_\theta| \cdot e^{-\frac{cN}{1+K\sigma^2/\|\theta^*\|^2}}
\leq e^{-\frac{c'N}{1+K\sigma^2/\|\theta^*\|^2}}$.

Combining this with (\ref{eq:I3bound}) and the chi-squared tail bound $\P[\cA^c]
\leq N \cdot \P[\|\eps^{(m)}\|^2>N] \leq Ne^{-cN} \leq e^{-c'N}$ for $N \geq
C_0K$, this gives
\[\P\left[\sup_{\theta \in \cB(\delta_1)}\,
\sup_{v:\|v\|=1,\,\langle u^*,v \rangle=0}\,
v^\top \nabla^2 R_N(\theta)v \leq \frac{1-\eta}{\sigma^2} \right]
\leq e^{-\frac{c'N}{1+K\sigma^2/\|\theta^*\|^2}}+\P[I_3>\tfrac{\eta \sigma^2}{6}]
+\P[\cA^c] \leq e^{-\frac{cN}{(1+K\sigma^2/\|\theta^*\|^2)^2}}\]
which implies the lemma.
\end{proof}

\section{Proofs for minimax lower bound}\label{appendix:lower}

\revise{For expositional clarity, we first show Lemma \ref{lemma:minimaxlowerP}
in the setting $\beta=0$ where the Fourier coefficients do not decay. At the
conclusion of this section, we extend the proof to all $\beta \in
[0,\frac{1}{2})$.}

\begin{proof}[Proof of Lemma \ref{lemma:KLupperbound}]
The first bound (\ref{eq:KLupperlownoise}) is basic and due to
the data processing inequality: Let $q_\theta(\alpha,y)$ denote the joint
density of $\alpha \sim \Unif([-\pi,\pi))$ and $y=g(\alpha) \cdot \theta+\sigma
\eps$. Then the data processing inequality implies
\begin{align*}
D_{\KL}(p_\theta \| p_{\theta'})
\leq D_{\KL}(q_\theta \| q_{\theta'})
&=\E_{(\alpha,y) \sim q_\theta}
\left[-\frac{\|y-g(\alpha) \cdot \theta\|^2}{2\sigma^2}
+\frac{\|y-g(\alpha) \cdot \theta'\|^2}{2\sigma^2}\right]\\
&=\mathop{\E_{\alpha \sim \Unif([-\pi,\pi))}}_{\eps \sim \mathcal{N}(0,I)}
\left[-\frac{\|\sigma \eps\|^2}{2\sigma^2}
+\frac{\|g(\alpha) \cdot (\theta-\theta')+\sigma \eps\|^2}{2\sigma^2}\right]
=\frac{\|\theta-\theta'\|^2}{2\sigma^2}.
\end{align*}

In the remainder of the proof, we show (\ref{eq:KLupperhighnoise}). Let us write
$\alpha,\alpha' \sim \Unif([-\pi,\pi))$ for independent
random rotations, $\E$ for the expectation over only $\alpha,\alpha'$ (fixing
$y$ and $\eps$), and $g=g(\alpha)$ and $g'=g(\alpha')$. We have
\[D_{\KL}(p_{\theta}\|p_{\theta'}) \leq \chi^2(p_{\theta}\|p_{\theta'})
:=\int_{\R^{2K}} \frac{[p_\theta(y)-p_{\theta'}(y)]^2}{p_\theta(y)}dy\]
where the right side is the $\chi^2$-divergence,
see e.g.\ \cite[Lemma 2.7]{tsybakov2008introduction}. We derive an upper bound
for $\chi^2(p_\theta\|p_{\theta'})$ using the idea of
\cite[Theorem 9]{bandeira2020optimal}: Let $\varphi(z)=(2\pi
\sigma^2)^{-K}\exp(-\|z\|^2/2\sigma^2)$ be the density of
$\mathcal{N}(0,\sigma^2 I_{2K})$. Then
\begin{equation}\label{eq:pthetaexpr}
p_\theta(y)=\E[\varphi(y-g\theta)]
=\E\Big[\varphi(y) \cdot e^{\frac{y^\top
g\theta}{\sigma^2}-\frac{\|\theta\|^2}{2\sigma^2}}\Big].
\end{equation}
By Jensen's inequality and the condition $\E[g]=0$,
\begin{equation}\label{eq:pthetalower}
p_\theta(y) \geq \varphi(y) \cdot e^{\frac{y^\top \E[g]\theta}{\sigma^2}
-\frac{\|\theta\|^2}{2\sigma^2}}
=\varphi(y) \cdot e^{-\frac{\|\theta\|^2}{2\sigma^2}}.
\end{equation}
Then applying (\ref{eq:pthetaexpr}), (\ref{eq:pthetalower}), and the 
moment generating function
\[\int \varphi(y)e^{\frac{y^\top (g\theta+g'\theta')}{\sigma^2}}dy
=e^{\frac{\|g\theta+g'\theta'\|^2}{2\sigma^2}}
=e^{\frac{\|\theta\|^2+\|\theta'\|^2+2\langle g\theta,g'\theta'\rangle}
{2\sigma^2}},\]
we get
\begin{align}
\chi^2(p_\theta\|p_{\theta'})
&\leq \int \frac{\Big(\E\Big[\varphi(y) \cdot
e^{\frac{y^\top g\theta}{\sigma^2}-\frac{\|\theta\|^2}{2\sigma^2}} \Big]
-\E\Big[\varphi(y) \cdot e^{\frac{y^\top g\theta'}{\sigma^2}
-\frac{\|\theta'\|^2}{2\sigma^2}}\Big]\Big)^2}
{\varphi(y) \cdot e^{-\frac{\|\theta\|^2}{2\sigma^2}}}dy\nonumber\\
&=\int \varphi(y)
\Big(e^{-\frac{\|\theta\|^2}{2\sigma^2}}
\E\Big[e^{\frac{y^\top g\theta}{\sigma^2}+\frac{y^\top g'\theta}{\sigma^2}}\Big]
-2e^{-\frac{\|\theta'\|^2}{2\sigma^2}}\E\Big[e^{\frac{y^\top g\theta}{\sigma^2}
+\frac{y^\top g'\theta'}{\sigma^2}}\Big]
+e^{-\frac{\|\theta'\|^2}{\sigma^2}+\frac{\|\theta\|^2}{2\sigma^2}}
\E\Big[e^{\frac{y^\top g\theta'}{\sigma^2}
+\frac{y^\top g'\theta'}{\sigma^2}}\Big]\Big)dy\nonumber\\
&=e^{\frac{\|\theta\|^2}{2\sigma^2}} 
\E\left[e^{\frac{\langle g\theta,g'\theta \rangle}{\sigma^2}}
-2e^{\frac{\langle g\theta,g'\theta' \rangle}{\sigma^2}}
+e^{\frac{\langle g\theta',g'\theta' \rangle}{\sigma^2}}\right]\nonumber\\
&=e^{\frac{\|\theta\|^2}{2\sigma^2}}
\E\left[e^{\frac{\langle \theta,g\theta \rangle}{\sigma^2}}
-2e^{\frac{\langle \theta,g\theta' \rangle}{\sigma^2}}
+e^{\frac{\langle \theta',g\theta' \rangle}{\sigma^2}}\right]\nonumber\\
&=e^{\frac{\|\theta\|^2}{2\sigma^2}}
\sum_{m=0}^\infty \frac{1}{\sigma^{2m}m!}
\E\big[\langle \theta,g\theta \rangle^m
-2\langle \theta,g\theta' \rangle^m+\langle \theta',g\theta' \rangle^m\big].
\label{eq:chisqbound}
\end{align}

For $m=0$ and $m=1$, the summand of (\ref{eq:chisqbound}) is 0.
We evaluate the summand for $m=2$, and upper bound it for $m \geq 3$. Let
\[\tilde{\theta}=(\theta_1,\ldots,\theta_K) \in \C^K,
\qquad \tilde{\theta}'=(\theta_1',\ldots,\theta_K') \in \C^K\]
be the complex representations of $\theta,\theta'$ as defined in Section
\ref{sec:complexrepr}. For $m=2$, applying (\ref{eq:CRisometry}),
\[\E[\langle \theta,g(\alpha)\theta' \rangle^2]
=\sum_{k_1,k_2=1}^K \E\Big[\Re(\overline{\theta_{k_1}}e^{ik_1\alpha}
\theta_{k_1}') \cdot \Re
(\overline{\theta_{k_2}}e^{ik_2\alpha}\theta_{k_2}')\Big].\]
For any $k_1,k_2 \in \{1,\ldots,K\}$, applying $\Re \bar{x}y
=(x\bar{y}+\bar{x}y)/2$ and $\E[e^{ik\alpha}]=0$ for any non-zero integer $k$,
\begin{align*}
\E\Big[\Re(\overline{\theta_{k_1}}e^{ik_1\alpha} \theta_{k_1}') \cdot \Re
(\overline{\theta_{k_2}}e^{ik_2\alpha}\theta_{k_2}')\Big]
&=\frac{1}{4}\E\Big[(e^{-ik_1\alpha}\theta_{k_1}\overline{\theta_{k_1}'}+e^{ik_1\alpha}\overline{\theta_{k_1}}\theta_{k_1}')
(e^{-ik_2\alpha}\theta_{k_2}\overline{\theta_{k_2}'}+e^{ik_2\alpha}
\overline{\theta_{k_2}}\theta_{k_2}')\Big]\\
&=\frac{1}{2}\1\{k_1=k_2\}|\theta_{k_1}|^2|\theta_{k_1}'|^2.
\end{align*}
Then
$\E[\langle \theta,g\theta' \rangle^2]=\frac{1}{2}\sum_{k=1}^K r_k^2{r_k'}^2$.
This identity holds also with $\theta=\theta'$, so
\begin{equation}\label{eq:m2term}
\E\big[\langle \theta,g\theta \rangle^2
-2\langle \theta,g\theta' \rangle^2+\langle \theta',g\theta' \rangle^2\big]
=\frac{1}{2}\sum_{k=1}^K (r_k^2-{r_k'}^2)^2.
\end{equation}

For any $m \geq 3$ and every $k_1,\ldots,k_m \in \{1,\ldots,K\}$, 
applying again $\E[e^{ik\alpha}]=0$ for $k \neq 0$, we have similarly
\begin{align*}
\E\left[\prod_{\ell=1}^m \Re(\overline{\theta_{k_\ell}}
e^{ik_\ell\alpha}\theta_{k_\ell}')\right]
&=\frac{1}{2^m}\E\left[\prod_{\ell=1}^m
(e^{-ik_\ell\alpha}\theta_{k_\ell}\overline{\theta_{k_\ell}'}+e^{ik_\ell\alpha}
\overline{\theta_{k_\ell}} \theta_{k_\ell}')\right]\\
&=\frac{1}{2^m}\sum_{s_1,\ldots,s_m \in \{+1,-1\}}
\1\{s_1k_1+\ldots+s_mk_m=0\} \cdot \prod_{\ell:s_\ell=+1}
\theta_{k_\ell}\overline{\theta_{k_\ell}'} \cdot \prod_{\ell:s_\ell=-1}
\overline{\theta_{k_\ell}}\theta_{k_\ell}'
\end{align*}
Noting that the left side is real and taking the real part on the right side,
this is equal to
\[\frac{1}{2^m} \sum_{s_1,\ldots,s_m \in \{+1,-1\}}
\1\{s_1k_1+\ldots+s_mk_m=0\} \left(\prod_{\ell=1}^m r_{k_\ell}r_{k_\ell}'\right)
\cos\left(\sum_{\ell=1}^m s_\ell \phi_{k_\ell}
-s_\ell \phi_{k_\ell}'\right).\]
Then, summing over all $k_1,\ldots,k_m \in \{1,\ldots,K\}$ and
applying this also for $\theta=\theta'$,
\begin{align}
&\E\big[\langle \theta,g\theta \rangle^m
-2\langle \theta,g\theta' \rangle^m+\langle \theta',g\theta'
\rangle^m\big]\nonumber\\
&=\frac{1}{2^m}\sum_{k_1,\ldots,k_m=1}^K
\sum_{s_1,\ldots,s_m \in \{+1,-1\}}
\1\{s_1k_1+\ldots+s_mk_m=0\} \cdot \nonumber\\
&\hspace{1in}
\left[\left(\prod_{\ell=1}^m r_{k_\ell}-\prod_{\ell=1}^m r_{k_\ell}'\right)^2
+2\left(\prod_{\ell=1}^m r_{k_\ell}r_{k_\ell}'\right)
\left(1-\cos\left(\sum_{\ell=1}^m s_\ell \phi_{k_\ell}
-s_\ell \phi_{k_\ell}'\right)\right)\right]\label{eq:mthmoment}
=:\mathrm{I}+\mathrm{II},
\end{align}
where $\mathrm{I}$ is the term involving
$(\prod_\ell r_{k_\ell}-\prod_\ell r_{k_\ell}')^2$, and $\mathrm{II}$ is the term
involving $\cos(\sum_\ell s_\ell \phi_{k_\ell}-s_\ell \phi_{k_\ell}')$.

To upper bound $\mathrm{I}$, let us write
\[\prod_{\ell=1}^m r_{k_\ell}-\prod_{\ell=1}^m r_{k_\ell}'
=\sum_{j=1}^m (r_{k_j}-r_{k_j}')r_{k_1}\ldots r_{k_{j-1}}
r_{k_{j+1}}'\ldots r_{k_m}'.\]
Then
\[\left(\prod_{\ell=1}^m r_{k_\ell}-\prod_{\ell=1}^m r_{k_\ell}'\right)^2
\leq m \cdot \sum_{j=1}^m (r_{k_j}-r_{k_j}')^2
\left(r_{k_1}\ldots r_{k_{j-1}} r_{k_{j+1}}'\ldots r_{k_m}'\right)^2.\]
So
\begin{align*}
\mathrm{I} &\leq \sum_{j=1}^m
\sum_{k_1,\ldots,k_m=1}^K \sum_{s_1,\ldots,s_m \in \{+1,-1\}}
\1\{s_1k_1+\ldots+s_mk_m=0\} \cdot
\frac{m}{2^m} (r_{k_j}-r_{k_j}')^2
\left(r_{k_1}\ldots r_{k_{j-1}} r_{k_{j+1}}'\ldots r_{k_m}'\right)^2
\end{align*}
Consider this summand for $j=1$. Note that fixing $s_1,\ldots,s_m$ and
$k_1,\ldots,k_{m-1}$, there is at most one choice for the remaining index $k_m
\in \{1,\ldots,K\}$
that satisfies $s_1k_1+\ldots+s_mk_m=0$. Thus, the summand for $j=1$ is at most
\[\max_{k_m=1}^K {r_{k_m}'}^2 \cdot
\sum_{k_1,\ldots,k_{m-1}=1}^K \sum_{s_1,\ldots,s_m \in \{+1,-1\}}
\frac{m}{2^m} (r_{k_1}-r_{k_1}')^2
\left(r_{k_2}'\ldots r_{k_{m-1}}'\right)^2
\leq m\sum_{k=1}^K (r_k-r_k')^2
\rupper^2 R^{2(m-2)}.\]
The same bound holds for each summand $j=1,\ldots,m$, yielding
\[\mathrm{I} \leq m^2\sum_{k=1}^K (r_k-r_k')^2 \rupper^2 R^{2(m-2)}.\]

To bound $\mathrm{II}$, observe that when $s_1k_1+\ldots+s_mk_m=0$, we have
\[\cos\left(\sum_{\ell=1}^m s_\ell \phi_{k_\ell}-s_\ell \phi_{k_\ell}'\right)=
\cos\left(\sum_{\ell=1}^m s_\ell \phi_{k_\ell}-s_\ell \phi_{k_\ell}'
+\alpha \cdot s_\ell k_\ell\right)\]
for any $\alpha \in \R$. Then applying $1-\cos(x) \leq x^2/2$ for any
$x \in \R$, we obtain
\[2\left(1-\cos\left(\sum_{\ell=1}^m s_\ell \phi_{k_\ell}
-s_\ell \phi_{k_\ell}'\right)\right)
\leq \inf_{\alpha \in \R}
\left(\sum_{\ell=1}^m s_\ell \phi_{k_\ell} -s_\ell \phi_{k_\ell}'
+\alpha \cdot s_\ell k_\ell\right)^2
=\inf_{\alpha \in \R} m\sum_{j=1}^m (\phi_{k_j}-\phi_{k_j}'+\alpha k_j)^2.\]
So
\begin{equation}\label{eq:IIbound}
\mathrm{II} \leq \inf_{\alpha \in \R} \sum_{j=1}^m
\sum_{k_1,\ldots,k_m=1}^K \sum_{s_1,\ldots,s_m \in \{+1,-1\}}
\1\{s_1k_1+\ldots+s_mk_m=0\} \cdot \frac{m}{2^m}(\phi_{k_j}-\phi_{k_j}'+\alpha k_j)^2
\left(\prod_{\ell=1}^m r_{k_\ell}r_{k_\ell}'\right).
\end{equation}
Applying $\sum_{k=1}^K r_kr_k' \leq
\sum_{k=1}^K (r_k^2+{r_k'}^2)/2 \leq R^2$ and a similar argument as above,
for any fixed $\alpha \in \R$, this summand for $j=1$ is at most
\begin{align*}
&\max_{k_m=1}^K r_{k_m}r_{k_m}'
\cdot \sum_{k_1,\ldots,k_{m-1}=1}^K \sum_{s_1,\ldots,s_m \in \{+1,-1\}}
\frac{m}{2^m}(\phi_{k_1}-\phi_{k_1}'+\alpha k_1)^2\Big(r_{k_1}\ldots r_{k_{m-1}}
r_{k_1}'\ldots r_{k_{m-1}}'\Big)\\
&\leq m\sum_{k=1}^K r_kr_k'(\phi_k-\phi_k'+\alpha k)^2\rupper^2 R^{2(m-2)}.
\end{align*}
The same bound holds for each summand $j=1,\ldots,m$, yielding
\[\mathrm{II} \leq \inf_{\alpha \in \R} m^2\sum_{k=1}^K
r_kr_k'(\phi_k-\phi_k'+\alpha k)^2
\cdot \rupper^2 \cdot R^{2(m-2)}.\]
Combining these bounds for $\mathrm{I}$ and $\mathrm{II}$, we arrive at
\[\E\big[\langle \theta,g\theta \rangle^m
-2\langle \theta,g\theta' \rangle^m+\langle \theta',g\theta' \rangle^m\big]
\leq m^2\rupper^2 \cdot R^{2(m-2)} \cdot \inf_{\alpha \in \R}
\sum_{k=1}^K (r_k-r_k')^2+r_kr_k'(\phi_k-\phi_k'+\alpha k)^2.\]

Let us now apply this to (\ref{eq:chisqbound}) and sum over $m \geq 3$:
We have
\[\sum_{m=3}^\infty \frac{m^2R^{2(m-2)}}{\sigma^{2m}m!}
=\sum_{m=3}^\infty \frac{mR^2}{(m-1)(m-2)\sigma^6} \cdot
\frac{R^{2(m-3)}}{\sigma^{2(m-3)}(m-3)!}
\leq \frac{3R^2}{2\sigma^6}e^{R^2/\sigma^2}.\]
Then
\begin{align*}
&\sum_{m=3}^\infty \frac{e^{\|\theta\|^2/2\sigma^2}}{\sigma^{2m}m!}
\E\big[\langle \theta,g\theta \rangle^m
-2\langle \theta,g\theta' \rangle^m+\langle \theta',g\theta' \rangle^m\big]\\
&\leq \frac{3\rupper^2 R^2e^{3R^2/2\sigma^2}}{2\sigma^6}
\cdot \inf_{\alpha \in \R}
\sum_{k=1}^K (r_k-r_k')^2+r_kr_k'(\phi_k-\phi_k'+\alpha k)^2.
\end{align*}
Applying this and (\ref{eq:m2term}) to (\ref{eq:chisqbound})
gives (\ref{eq:KLupperhighnoise}).
\end{proof}

For a specific regime of parameters $\theta,\theta' \in \R^{2K}$,
we simplify the lower bound for the loss in Proposition
\ref{prop-lossgeneral} by expressing the squared distance
$|\phi_k-\phi_k'+k\alpha|_\cA^2$ on the circle $\cA$ in terms of the
usual squared distance $(\phi_k-\phi_k'+k\alpha)^2$ on $\R$.
\begin{Lemma}\label{lemma:lossbounds}
Fix any $\theta,\theta' \in \R^{2K}$ and let
$\theta=(r_k\cos\phi_k,r_k\sin\phi_k)_{k=1}^K$ and
$\theta'=(r_k'\cos\phi_k',r_k'\sin\phi_k')_{k=1}^K$.
For each $\alpha \in \R$, let $K_0(\alpha) \in [0,K]$ be the largest integer for
which $|K_0(\alpha) \cdot \alpha| \leq \pi/2$. If $r_k,r_k' \geq \rlower$ 
and $|\phi_k-\phi_k'| \leq \pi/3$ for each $k=1,\ldots,K$, then
for a universal constant $c>0$,
\begin{equation}\label{eq:losslowerbound}
L(\theta,\theta') \geq \sum_{k=1}^K (r_k-r_k')^2
+c \inf_{\alpha \in \R} \left(
(K-K_0(\alpha))\rlower^2+
\sum_{k=1}^{K_0(\alpha)} r_kr_k'(\phi_k-\phi_k'+k\alpha)^2 \right)
\end{equation}
where the second summation is understood as 0 if $K_0(\alpha)=0$.
\end{Lemma}
\begin{proof}
Recall the form (\ref{eq:lossrphi}) of the loss
from Proposition \ref{prop-lossgeneral},
where the infimum over $\alpha$ may be restricted to $[-\pi,\pi]$
by periodicity. We provide a lower bound for
$\alpha \in [0,\pi]$, and the case $\alpha \in [-\pi,0]$ is analogous.

If $\alpha=0$, let $K_0=K_1=\ldots=K$. Otherwise if $\alpha \in (0,\pi]$,
let $0 \leq K_0 \leq K_1 \leq K_2 \leq \ldots$ be such that each $K_m$ is the
largest integer in $[0,K]$
for which $K_m \cdot \alpha \leq 2\pi m+\frac{\pi}{2}$. Note that if $K_m<K$
strictly, then we must have also $K_m<K_{m+1}$. If $K_0 \geq 1$, then
for every $k \in [1,K_0]$, we have $k\alpha \in [0,\pi/2]$, so
$\phi_k-\phi_k'+k\alpha \in [-\pi/3,5\pi/6]$ and
\[1-\cos(\phi_k-\phi_k'+k\alpha) \geq c(\phi_k-\phi_k'+k\alpha)^2\]
for a universal constant $c>0$. Thus
\begin{equation}\label{eq:losslowerI}
\sum_{k=1}^{K_0} r_kr_k'\Big[1-\cos(\phi_k-\phi_k'+k\alpha)\Big]
\geq c\sum_{k=1}^{K_0} r_kr_k'(\phi_k-\phi_k'+k\alpha)^2.
\end{equation}
This bound is also trivially true if $K_0=0$.

Now fix any $m \geq 0$ where $K_m<K$ strictly. Consider the values
\[k \in \{K_m+1,\ldots,K_{m+1}\}.\]
For each such $k$, we have $k\alpha \in (2\pi m+\frac{\pi}{2},2\pi m+\frac{5\pi}{2}]$.
Let $a$ be the number of such values $k$ where
$k\alpha \in (2\pi m+\frac{\pi}{2},2\pi m+\frac{3\pi}{2}]$, and let
$b$ be the number of such values $k$ where
$k\alpha \in (2\pi m+\frac{3\pi}{2},2\pi m+\frac{5\pi}{2}]$.
Then we must have $a \geq 1$ because $\alpha \in [0,\pi]$. Also, the number of multiples
of $\alpha$ belonging to $(2\pi m+\frac{3\pi}{2},2\pi m+\frac{5\pi}{2}]$ is at most 1
more than the number of multiples of $\alpha$ belonging to
$(2\pi m+\frac{\pi}{2},2\pi m+\frac{3\pi}{2}]$, so $b \leq a+1$. Thus
\[\frac{a}{K_{m+1}-K_m}=\frac{a}{a+b} \geq \frac{a}{2a+1} \geq \frac13.\]
For $k=K_m+1,\ldots,K_m+a$, we must have $\phi_k-\phi_k'+k\alpha \in
(2\pi m+\frac{\pi}{6}, 2\pi m+\frac{11\pi}{6}]$, so
$1-\cos(\phi_k-\phi_k'+k\alpha) \geq c$
for a universal constant $c>0$. Then
\[\sum_{k=K_m+1}^{K_{m+1}} r_kr_k'\Big[1-\cos(\phi_k-\phi_k'+k\alpha)\Big]
\geq \sum_{k=K_m+1}^{K_m+a} c \cdot r_kr_k'
\geq c\rlower^2 \cdot \frac{K_{m+1}-K_m}{3}.\]
Now summing over all $m \geq 0$ where $K_m<K$,
\begin{equation}\label{eq:losslowerII}
\sum_{k=K_0+1}^K r_kr_k'\Big[1-\cos(\phi_k-\phi_k'+k\alpha)\Big]
\geq \frac{c\rlower^2}{3} \cdot (K-K_0).
\end{equation}
Applying (\ref{eq:losslowerI}) and (\ref{eq:losslowerII}) and the analogous
bounds for $\alpha \in [-\pi,0]$ to (\ref{eq:lossrphi}),
and taking the infimum over $\alpha \in [-\pi,\pi]$, we obtain
(\ref{eq:losslowerbound}).
\end{proof}

We conclude the proof of Lemma \ref{lemma:minimaxlowerP} \revise{for $\beta=0$}
using the following version of Assouad's hypercube lower bound from
\cite[Lemma 2]{cai2012optimal}.

\begin{Lemma}\label{lemma:assouad}
Fix $m \geq 1$, let $\{P_\tau:\tau \in \{0,1\}^m\}$
be any $2^m$ probability distributions, and let $\psi(P_\tau)$ 
take values in a metric space with metric $d$. Then for any $s>0$
and any estimator $\hat{\psi}(X)$ based on $X \sim P_\tau$,
\begin{align*}
&\sup_{\tau \in \{0,1\}^m}
\E_{X \sim P_\tau}\Big[d(\hat{\psi}(X),\psi(P_\tau))^s\Big]\\
&\geq \frac{m}{2^{s+1}} \cdot
\min_{H(\tau,\tau') \geq 1} \frac{d(\psi(P_\tau),\psi(P_{\tau'}))^s}
{H(\tau,\tau')} \cdot \min_{H(\tau,\tau')=1}
\Big(1-D_\TV(P_\tau,P_{\tau'})\Big).
\end{align*}
Here, $H(\tau,\tau')=\sum_{i=1}^m \1\{\tau_i \neq \tau_i'\}$ is the Hamming
distance between $\tau$ and $\tau'$, and $D_\TV(P_\tau,P_{\tau'})$ is the
total-variation distance between $P_\tau$ and $P_{\tau'}$.
\end{Lemma}

\begin{proof}[Proof of Lemma \ref{lemma:minimaxlowerP}, \revise{$\beta=0$}]
We define $2^K$ parameters $\theta^\tau \in \cP_0$
indexed by $\tau \in \{0,1\}^K$: Fix a value $\phi \in [0,\pi/3]$ to be
determined. For each $\tau \in \{0,1\}^K$, set
\begin{equation}\label{eq:lowerboundPconstruction}
\phi_k^\tau=\tau_k \phi=\begin{cases} \phi & \text{ if } \tau_k=1 \\
0 & \text{ if } \tau_k=0.
\end{cases}
\end{equation}
Then let $\theta^\tau$ be the vector where
$r_k(\theta)=1$ and $\phi_k(\theta)=\phi_k^\tau$ for each $k=1,\ldots,K$.

Let $P_\tau=p_{\theta^\tau}^N$ denote the law of $N$
samples $y^{(1)},\ldots,y^{(N)} \overset{iid}{\sim} p_{\theta^\tau}$.
Let $\orbit_\theta=\{g(\alpha) \cdot \theta:\alpha \in \cA\}$ be the rotational
orbit of $\theta$. Then
$d(\orbit_\theta,\orbit_{\theta'}):=L(\theta,\theta')^{1/2}
=\min_{\alpha \in \cA} \|\theta'-g(\alpha) \cdot \theta\|$
defines a metric over the space of all such orbits. We apply
Lemma \ref{lemma:assouad} with $m=K$, $\psi(P_\tau)=\orbit_{\theta^\tau}$,
this metric $d(\orbit_\theta,\orbit_{\theta'})$, and $s=2$. Applying
(\ref{eq:KLupperlownoise}) and $|e^{is}-e^{it}| \leq |s-t|$ for all $s,t \in
\R$,
\begin{equation}\label{eq:DKLHamminglownoise}
D_{\KL}(p_{\theta^\tau}\|p_{\theta^{\tau'}})
\leq \frac{\|\theta^\tau-\theta^{\tau'}\|^2}{2\sigma^2}
=\frac{1}{2\sigma^2} \sum_{k=1}^K \big|e^{i\phi_k^\tau}
-e^{i\phi_k^{\tau'}}\big|^2
\leq \frac{\phi^2}{2\sigma^2} \cdot H(\tau,\tau')
\end{equation}
where $H(\tau,\tau')$ is the Hamming distance.
Applying (\ref{eq:KLupperhighnoise}) with the right side evaluated at
$\alpha=0$, \revise{where $\rupper=1$ and $R^2=K$}, also
\begin{equation}\label{eq:DKLHamminghighnoise}
D_{\KL}(p_{\theta^\tau}\|p_{\theta^{\tau'}})
\leq \frac{\phi^2}{A} \cdot H(\tau,\tau'),
\qquad \revise{A:=\frac{2\sigma^6}{3Ke^{3K/2\sigma^2}}}
\end{equation}
Then, setting
\begin{equation}\label{eq:phistar}
\phi=\min\left(\frac{1}{\sqrt{N}} \cdot \max(\sqrt{2\sigma^2},\sqrt{A}),
\frac{\pi}{3}\right),
\end{equation}
these bounds imply for both cases of the max that
$D_{\KL}(p_{\theta^\tau}\|p_{\theta^{\tau'}}) \leq H(\tau,\tau')/N$.
Then by Pinsker's
inequality (see e.g.\ \citep[Lemma 2.5]{tsybakov2008introduction}),
\[D_{\TV}(P_\tau,P_{\tau'})
\leq \sqrt{\frac{1}{2}D_{\KL}(P_\tau\|P_{\tau'})}
=\sqrt{\frac{N}{2}D_{\KL}(p_{\theta^\tau}\|p_{\theta^{\tau'}})}
\leq \sqrt{\frac{1}{2}H(\tau,\tau')},\]
so
\[\min_{H(\tau,\tau')=1} \Big(1-D_{\TV}(P_\tau,P_{\tau'})\Big)
\geq 1-\sqrt{1/2}>0.\]

Since $\phi_k \in [0,\pi/3]$ for every $k$, we may apply
Lemma \ref{lemma:lossbounds} to lower-bound the loss: For a universal
constant $c>0$, we have
\begin{equation}\label{eq:losslowerP}
L(\theta^\tau,\theta^{\tau'}) \geq c \inf_{K_0 \in [0,K]}
\inf_{\alpha \in \R} \left(K-K_0
+\sum_{k=1}^{K_0} (\phi_k^\tau-\phi_k^{\tau'}+k\alpha)^2 \right).
\end{equation}
For any fixed $K_0 \in [2,K]$, the inner infimum over $\alpha$ is
attained at $\alpha=-\sum_{k=1}^{K_0} k(\phi_k^\tau-\phi_k^{\tau'})
/\sum_{k=1}^{K_0} k^2$, and we have
\begin{align*}
\inf_{\alpha \in \R} \sum_{k=1}^{K_0} (\phi_k^\tau-\phi_k^{\tau'}+k\alpha)^2
&=\sum_{k=1}^{K_0} (\phi_k^\tau-\phi_k^{\tau'})^2-\frac{\left(\sum_{k=1}^{K_0}
k(\phi_k^\tau-\phi_k^{\tau'})\right)^2}{\sum_{k=1}^{K_0} k^2}\\
&=\phi^2 \cdot \left(H(\tau^{K_0},{\tau'}^{K_0})-\frac{\left(\sum_{k=1}^{K_0}
k(\tau_k-\tau_k')\right)^2}{\sum_{k=1}^{K_0} k^2}\right)
\end{align*}
where $\tau^{K_0}=(\tau_1,\ldots,\tau_{K_0})$,
${\tau'}^{K_0}=(\tau_1',\ldots,\tau_{K_0}')$, and $H(\tau^{K_0},{\tau'}^{K_0})$
is the Hamming distance of these subvectors in $\{0,1\}^{K_0}$.
Subject to a constraint that $H(\tau^{K_0},{\tau'}^{K_0})=h$, we have
\begin{align*}
\left(\sum_{k=1}^{K_0} k(\tau_k-\tau_k')\right)^2
&\leq \Big(K_0+(K_0-1)+\ldots+(K_0-h+1)\Big)^2\\
&=\left(\frac{h(2K_0-h+1)}{2}\right)^2
=h \cdot \frac{h(2K_0-h+1)^2}{4} \leq h \cdot \frac{(2K_0+1)^3}{27},
\end{align*}
where the last inequality is tight at the maximizer $h=(2K_0+1)/3$. Then
for any $K_0 \in [2,K]$,
\begin{align*}
H(\tau^{K_0},{\tau'}^{K_0})-\frac{\left(\sum_{k=1}^{K_0}
k(\tau_k-\tau_k')\right)^2}{\sum_{k=1}^{K_0} k^2}
&\geq H(\tau^{K_0},{\tau'}^{K_0}) \cdot
\left(1-\frac{(2K_0+1)^3/27}{K_0(K_0+1)(2K_0+1)/6}\right)
\geq \frac{2H(\tau^{K_0},{\tau'}^{K_0})}{27}.
\end{align*}
Applying also $K-K_0 \geq
H((\tau_{K_0+1},\ldots,\tau_K),(\tau_{K_0+1}',\ldots,\tau_K'))$
and $\phi \leq \pi/3$, we get
\begin{equation}\label{eq:finallossbound}
\inf_{\alpha \in \R} \left(K-K_0+\sum_{k=1}^{K_0}
(\phi_k^\tau-\phi_k^{\tau'}+k\alpha)^2\right)
\geq c\phi^2 H(\tau,\tau')
\end{equation}
for a universal constant $c>0$. For $K_0=0$ or $K_0=1$ (and any $K \geq 2$), we
may instead lower bound the left side by $K-K_0 \geq K/2 \geq H(\tau,\tau')/2$,
so that this bound (\ref{eq:finallossbound}) holds also. Thus, taking the
infimum in (\ref{eq:losslowerP}) over $K_0 \in [0,K]$,
\begin{equation}\label{eq:dpsilower}
d(\psi(P_\tau),\psi(P_{\tau'}))^2
=L(\theta^\tau,\theta^{\tau'}) \geq c'\phi^2 \cdot H(\tau,\tau')
\end{equation}
for a universal constant $c'>0$. Applying Lemma \ref{lemma:assouad}
with $m=K$ and $s=2$, we obtain
\begin{equation}\label{eq:assouadapplication}
\inf_{\hat{\theta}} \sup_{\theta^* \in \cP(r)}
\E_{\theta^*}[L(\theta^*,\hat{\theta})]
\geq \inf_{\hat{\theta}} \sup_{\theta^\tau:\tau \in \{0,1\}^K}
\E_{\theta^\tau}[L(\theta^\tau,\hat{\theta})]
\geq c''K\phi^2.
\end{equation}
Applying the form of $\phi$ from (\ref{eq:phistar}) concludes the proof.
\end{proof}

\revise{We now extend this argument to the more general case of $\beta \in
[0,\frac{1}{2})$. Consider the subset of $\cP_\beta$ defined by
\[\cP_\beta^*=\Big\{\theta^* \in \R^{2K}: r_k(\theta^*)=k^{-\beta}
\text{ for all } k\in \{1,\ldots,K\}~~\text{and}~~\phi_k(\theta^*)=0\text{ for
all } k \leq K/2\Big\}\]
where the phases of the first half of the Fourier frequencies are fixed to 0.
For $\theta,\theta'$ belonging to $\cP_\beta^*$, the following lemma 
modifies the KL upper bound (\ref{eq:KLupperhighnoise}) from
Lemma \ref{lemma:KLupperbound}, replacing the factor
$\rupper=\max_k r_k=1$ by a multiple of $\rlower=K^{-2\beta}$.

\begin{Lemma}\label{lemma-6.2.1}
Fix any $\beta \in [0,\frac{1}{2})$, and denote
$R^2=\sum_{k=1}^K k^{-2\beta}$ and $\rlower^2=K^{-2\beta}$.
Then for all $\theta, \theta' \in \cP_\beta^*$,
\begin{equation}
D_{\KL}(p_{\theta}\|p_{\theta'}) \leq
\frac{27\rlower^2 R^2e^{5R^2/2\sigma^2}}{\sigma^6}
\cdot \inf_{\alpha \in \R} \sum_{k=1}^K k^{-2\beta}(\phi_k-\phi_k'+\alpha k)^2
\label{eq:KLupperhighnoise-new}
\end{equation}
\end{Lemma}
\begin{proof}
Following the proof of Lemma \ref{lemma:KLupperbound}, we provide a new bound for
the quantity $\mathrm{I}+\mathrm{II}$ in (\ref{eq:mthmoment}).

Since $r_k(\theta)=r_k(\theta')$ for all $k$,
we have $\mathrm{I}=0$. For $\mathrm{II}$, notice that we may restrict the
summation over $k_1,\ldots,k_m$ in its definition to tuples where
$\max(k_1,\ldots,k_m)>K/2$, since otherwise the summand is 0 upon setting
$\alpha=0$, by the condition $\phi_k(\theta)=\phi_k(\theta')$ for all
$k \leq K/2$. Therefore, we have similarly to (\ref{eq:IIbound}),
\begin{align*}
\mathrm{II} &\leq \inf_{\alpha \in \R} \sum_{j=1}^m
\mathop{\sum_{k_1,\ldots,k_m=1}}_{\max(k_1,\ldots,k_m)>K/2}^K\;\;
\mathop{\sum_{s_1,\ldots,s_m \in \{+1,-1\}}}_{s_1k_1+\ldots+s_mk_m=0}
 \frac{m}{2^m}(\phi_{k_j}-\phi_{k_j}'+\alpha k_j)^2
\left(\prod_{\ell=1}^m r_{k_\ell}r_{k_\ell}'\right).
\end{align*}
Fix $\alpha \in \R$, consider the summand for $j=1$, and notice
that if $s_1k_1+\ldots + s_mk_m = 0$ and $\max(k_1,\ldots,k_m)>K/2$, then
the second largest index amongst $k_1,\ldots,k_m$ is at least $K/(2m)$.
Then, the summand for $j=1$ is bounded by
\begin{align*}
&\sum_{i=2}^m
\mathop{\mathop{\sum_{k_1,\ldots,k_m=1}^K}_{\max(k_1,\ldots,k_m)>K/2}}_{
k_i \geq \max(k_2,\ldots,k_m)}\;\;
\mathop{\sum_{s_1,\ldots,s_m \in \{+1,-1\}}}_{s_1k_1+\ldots+s_mk_m=0}
 \frac{m}{2^m}(\phi_{k_1}-\phi_{k_1}'+\alpha k_1)^2
\left(\prod_{\ell=1}^m r_{k_\ell}r_{k_\ell}'\right)\\
&\leq \sum_{i=2}^m\;\max_{k_i>\frac{K}{2m}} r_{k_i}r_{k_i}'
\;\sum_{k_1,\ldots,k_{i-1},k_{i+1},\ldots,k_m=1}^K
\sum_{s_1,\ldots,s_m \in \{+1,-1\}} 
\frac{m}{2^m}(\phi_{k_1}-\phi_{k_1}'+\alpha k_1)^2
\left(\prod_{\ell \neq i} r_{k_\ell}r_{k_\ell}'\right)\\
&\leq m^2 r^2\left(\frac{K}{2m}\right)^{-2\beta}\sum_{k_1=1}^K
r_{k_1}r_{k_1}'(\phi_{k_1}-\phi_{k_1}'+\alpha k_1)^2 R^{2(m-2)}.
\end{align*}
In the second line, for each fixed $i \in \{2,\ldots,m\}$,
we have used that for every choice of
$s_1,\ldots,s_m$ and $k_1,\ldots,k_{i-1},k_{i+1},\ldots,k_m$, there is at most
one choice of $k_i$ for which conditions $s_1k_1+\ldots+s_mk_m=0$,
$\max(k_1,\ldots,k_m)>K/2$, and $k_i \geq \max(k_2,\ldots,k_m)$ all hold, and
such a value $k_i$ must be at least $K/(2m)$. The same bound holds for each
summand $j=1,\ldots,m$. Then, applying $(K/2m)^{-2\beta}=(2m)^{2\beta}\rlower^2
<2m\rlower^2$ for $\beta<1/2$,
and taking the infimum over $\alpha \in \R$, this gives
\[\mathrm{II} \leq \inf_{\alpha \in \R} 2m^4 R^{2(m-2)}\rlower^2 
\sum_{k=1}^K k^{-2\beta}(\phi_k-\phi_k'+\alpha k)^2.\]
Thus for all $m \geq 3$,
\[\E[\langle \theta,g\theta \rangle^m-2\langle \theta,g\theta' \rangle^m
+\langle \theta',g\theta' \rangle^m]
\leq 2m^4 R^{2(m-2)}\rlower^2  \cdot \inf_{\alpha \in \R} \sum_{k=1}^K k^{-2\beta}
(\phi_k-\phi_k'+\alpha k)^2,\]
and the left side is 0 for $m=0,1,2$ because
$r_k(\theta)=r_k(\theta')$ for all $k$.

Now we apply this to (\ref{eq:chisqbound}) and sum over $m \geq 3$, using
\begin{align*}
\sum_{m=3}^\infty \frac{2m^4R^{2(m-2)}}{\sigma^{2m}m!}
\leq \max_{m=3}^\infty \frac{2m^2R^2}{(m-1)(m-2)\sigma^6}
\cdot \sum_{m=3}^\infty \frac{m \cdot R^{2(m-3)}}{\sigma^{2(m-3)}(m-3)!}
=\frac{9R^2}{\sigma^6} \cdot
\left(\frac{R^2}{\sigma^2}+3\right)e^{R^2/\sigma^2}
\leq \frac{27R^2}{\sigma^6}e^{2R^2/\sigma^2}
\end{align*}
This yields as desired
\[D_{\KL}(p_\theta\|p_{\theta'}) \leq \chi^2(p_\theta\|p_{\theta'})
\leq \frac{27R^2}{\sigma^6}e^{5R^2/2\sigma^2}
\rlower^2 \cdot \inf_{\alpha \in \R}
\sum_{k=1}^K k^{-2\beta}(\phi_k-\phi_k'+\alpha k)^2.\]
\end{proof}

\begin{proof}[Proof of Lemma \ref{lemma:minimaxlowerP}]
Consider the subset of points $\tau \in \{0,1\}^K$ such that $\tau_k=0$ for all
$k \leq K/2$. Fixing a value $\phi \in [0,\pi/3]$, we define corresponding to
each such $\tau \in \{0,1\}^K$ the parameter $\theta^\tau \in \cP_\beta^*$ where
$r_k(\theta)=k^{-\beta}$ and $\phi_k(\theta)=\tau_k\phi$ for each
$k=1,\ldots,K$. These parameter vectors $\theta^\tau$ are thus identified with
the vertices of a hypercube of dimension $m=\lfloor K/2 \rfloor$.

For any two such vectors $\theta^\tau$ and $\theta^{\tau'}$,
applying (\ref{eq:KLupperlownoise}), we have
analogously to (\ref{eq:DKLHamminglownoise}) that
\[D_{\KL}(p_{\theta^\tau} \| p_{\theta^{\tau'}})
\leq \frac{\|\theta^\tau-\theta^{\tau'}\|^2}{2\sigma^2}
=\frac{1}{2\sigma^2}\sum_{k>K/2} k^{-2\beta}
\big|e^{i\phi_k^\tau}-e^{i\phi_k^{\tau'}}\big|^2
\leq \frac{\phi^2}{K^{2\beta}\sigma^2} \cdot H(\tau,\tau'),\]
the last inequality using $(K/2)^{-2\beta}<2K^{-2\beta}$ when $\beta<1/2$.
Applying (\ref{eq:KLupperhighnoise-new}) with $\alpha=0$,
where $\rlower^2=K^{-2\beta}$ and $R^2=CK^{1-2\beta}$,
we have analogously to (\ref{eq:DKLHamminghighnoise}) that for constants $C,c>0$
depending only on $\beta$,
\[D_{\KL}(p_{\theta^{\tau}} \| p_{\theta^{\tau'}})
\leq \frac{\phi^2}{A} \cdot H(\tau,\tau'),
\qquad A:=\frac{c\sigma^6K^{4\beta}}{K^{1-2\beta}e^{CK^{1-2\beta}/2\sigma^2}}\]
Lower bounding both $\rlower^2$ and $r_kr_k'$ by $K^{-2\beta}$ in
Lemma \ref{lemma:lossbounds}, and applying (\ref{eq:finallossbound}), we have
\[L(\theta^{\tau}, \theta^{\tau'}) \geq c' K^{-2\beta} \phi^2 H(\tau, \tau')\]
for a universal constant $c'>0$. We can choose 
\[\phi=\min\left(\frac{1}{\sqrt{N}}\cdot
\max(\sqrt{K^{2\beta}\sigma^2},\sqrt{A}),\frac{\pi}{3}\right)\]
to ensure $1-D_{\TV}(P_\tau,P_{\tau'}) \geq 1-\sqrt{1/2}>0$ whenever
$H(\tau,\tau')=1$, as before. Finally, by using Lemma \ref{lemma:assouad}
with $m=\lfloor K/2 \rfloor$, we can conclude analogously to
(\ref{eq:assouadapplication}) that
\[\inf_{\hat{\theta}} \sup_{\theta^* \in \cP_\beta^*}
\E_{\theta^*}[L(\theta^*,\hat{\theta})] \geq c''K^{1-2\beta}\phi^2,\]
and applying the above choice of $\phi$ concludes the proof.
\end{proof}}

\paragraph{Acknowledgements}
The authors would like to thank Yihong Wu for a helpful discussion about
KL-divergence in mixture models.

\paragraph{Funding}
Z. Fan is supported in part by NSF DMS 1916198, DMS 2142476. H. H. Zhou is supported in part by NSF grants DMS 2112918, DMS 1918925, and NIH grant 1P50MH115716. 

\bibliographystyle{plainnat}
\bibliography{bibliography}

\end{document}